\documentclass{article}
\usepackage[utf8]{inputenc}
\usepackage{commath}

\usepackage{authblk}

\usepackage{mathrsfs, amssymb}
\usepackage{amsthm}
\usepackage{makecell}
\usepackage{booktabs}
\usepackage[intlimits]{empheq}
\usepackage{bbm} 
\usepackage{amsmath}
\usepackage{xcolor}
\usepackage[scientific-notation=true]{siunitx}
\usepackage{float}
\usepackage{dsfont}	
\usepackage[numbers]{natbib}
\usepackage[nottoc,numbib]{tocbibind}
\usepackage{bbm}
\usepackage{subcaption}
\usepackage[inner=2cm,outer=2cm,top=2cm,bottom=2cm]{geometry}
\usepackage[ruled,vlined]{algorithm2e}
\usepackage{enumitem}
\usepackage[colorlinks=true,linkcolor=black]{hyperref}
\usepackage{pdfpages}
\usepackage[dvipsnames]{xcolor}

\usepackage[export]{adjustbox}
\usepackage{caption}
\DeclareMathOperator*{\argmin}{arg\,min}

\DeclareMathOperator*{\E}{\mathbb{E}}

\DeclareMathOperator*{\KL}{\operatorname{KL}}

\DeclareMathOperator*{\Ent}{\operatorname{Ent}}
\newcommand{\dd}{\mathop{}\!\mathrm{d}}

\newcommand{\Law}{\operatorname{Law}}

\DeclareMathOperator{\Tr}{Tr}
\DeclareMathOperator{\Reg}{Reg}

\DeclareMathOperator*{\argmax}{arg\,max}
    
\theoremstyle{plain}
\newtheorem{theorem}{Theorem}[section]
\newtheorem{lemma}[theorem]{Lemma}
\newtheorem{corollary}[theorem]{Corollary}
\newtheorem{proposition}[theorem]{Proposition}
\newtheorem{assumption}{Assumption}

\newtheorem*{theorem*}{Theorem}
\newtheorem*{proposition*}{Proposition}

\theoremstyle{definition}
\newtheorem{definition}[theorem]{Definition}

\newtheorem{remark}[theorem]{Remark}

    \usepackage[T1]{fontenc}
    \usepackage{lipsum}
    \usepackage{tocloft}
\makeatletter
\@tfor\next:=abcdefghijklmnopqrstuvwxyzABCDEFGHIJKLMNOPQRSTUVWXYZ\do{%
\def\command@factory#1{%
\expandafter\def\csname b#1\endcsname{\mathbf{#1}}
\expandafter\def\csname cl#1\endcsname{\mathcal{#1}}
\expandafter\def\csname bb#1\endcsname{\mathbb{#1}}
}
\expandafter\command@factory\next
}

\title{The Fenchel Game of Underdamped Langevin Dynamics: \\ Insights into Accelerated Convergence} 

\author[1]{A. Borkowski\thanks{\texttt{alexandra.borkowski@kcl.ac.uk}}}
\author[1]{N. N\"usken\thanks{\texttt{nikolas.nusken@kcl.ac.uk}}}

\affil[1]{Department of Mathematics, King's College London, United Kingdom}

\date{\today}

\begin{document}

\maketitle

\abstract{
For $(Q_t,P_t)$ governed by suitably damped underdamped Langevin dynamics, we quantify the convergence in KL divergence of the positional marginal to a $\sigma$-strongly log-concave target $\pi(\dd q) = \frac{1}{Z}e^{-V(q)}\dd q$ as 
\begin{align*}
\operatorname{KL}(\operatorname{Law}(Q_t) \| \pi) \leq e^{-\sqrt{\sigma} t}\operatorname{KL}(\operatorname{Law}(Q_{0}, P_{0})\, \|\, \Pi_0),
\end{align*}
where $\Pi_0$ denotes an appropriately selected reference measure. When $\pi$ is merely log-concave, the estimate 
\begin{align*}
    \operatorname{KL}(\operatorname{Law}(Q_t) \| \pi) \leq \frac{\tau^2}{t^2}\operatorname{KL}(\operatorname{Law}(Q_{\tau}, P_{\tau})\, \|\, \Pi_{\tau})
\end{align*}
is derived, where $\Pi_\tau$ denotes another correspondingly chosen reference measure at time $\tau > 0$. Both rates match precisely the canonical rates of the corresponding accelerated gradient flows in $\mathbb{R}^d$. They are achieved by the novel interpretation of the underdamped Langevin dynamics as a combination of strategies in an online sampling game and by estimating the KL divergence using a cost function informed by fictitious competitors.}

\setcounter{tocdepth}{1}

\section{Introduction}
Underdamped Langevin dynamics, defined by
\begin{subequations}
\label{eq:UL}
\begin{align}
    \dd Q_t & = P_t \dd t, \\
    \dd P_t & = - \nabla V(Q_t) \dd t - \gamma_t P_t \dd t + \sqrt{2\gamma_t} \dd W_t,
\end{align}
\end{subequations}
models the motion of a particle with position $Q_t$ and velocity $P_t$, subject to a conservative force $-\nabla V$, frictional damping with strength $\gamma_t > 0$, and stochastic perturbations. If the potential $V:\mathbb{R}^d \rightarrow \mathbb{R}$ is smooth and confining and $\gamma_t$ is constant in time, then, for any initial distribution of $(Q_0, P_0)$, the distribution of $(Q_t, P_t)$ converges to the unique invariant measure 
\begin{equation}
\label{eq:qp target}
\Pi = \pi(\dd q)\mathcal{N}(0,I_d)(\dd p),
\end{equation}
where the $Q$-marginal is given by the Gibbs measure 
\begin{align}\label{eq:q_marginal}
\pi = \frac{1}{Z} e^{-V(q)} \dd q,
\end{align}
see, e.g., \cite[Chapter 6]{pavliotis2014stochastic} for a detailed overview. As a consequence, underdamped Langevin dynamics plays a central role for sampling from probability distributions and finds applications in numerous fields of statistical and computational science, such as molecular simulation \cite{lelievre2016partial, BussiParrinello2007, LeimkuhlerMatthews2013}, Bayesian inference \cite{pmlr-v32-cheni14, pmlr-v75-cheng18a, ChatterjiEtAl2018VarianceReduction}, and machine learning \cite{DockhornVahdatKreis2022, ZhengEtAl2024UnderdampedThompson, NIPS2016_a96d3afe}.
\\

In this work, we consider convex and strongly convex potentials $V$; recall that $V$ is said to be $\sigma$-strongly convex if $x \mapsto V(x) - \tfrac{\sigma}{2}\| x \|^2$ is convex.
Our results quantify the convergence of $\Law(Q_t)$ towards $\pi$: 

\begin{theorem}
[Accelerated convergence in the strongly convex case]
\label{theorem:UL_strongly-convex}
Assume that $V \in C^1(\mathbb{R}^d)$ is $\sigma$-strongly convex, for some $\sigma > 0$, and let the regularity assumptions from Assumption \ref{ass:FP-regularity} below be satisfied. Set $\gamma_t = 2 \sqrt{\sigma}$ for all $t \ge 0$, and let $r_0$ denote the Wasserstein-2 optimal transport map from $\operatorname{Law}(Q_0)$ to $\pi$. Then, for every $t \ge 0$, the solution to \eqref{eq:UL} satisfies
\begin{equation}
\label{eq:sc convergence}
    \operatorname{KL}(\operatorname{Law}(Q_t) \| \pi) \leq e^{-\sqrt{\sigma} t}\operatorname{KL}(\operatorname{Law}(Q_{0}, P_{0})\, \|\, \Pi_0),
\end{equation}
where the initial reference measure $\Pi_0$ is defined by
\begin{align*}
    \Pi_0\left(\dd q,\dd p\right) = \pi\left(\dd q\right) \mathcal{N}\big(\sqrt{\sigma}(r_{0}(q)-q), I_d\big)(\dd p).
\end{align*}
\end{theorem}

In the case when $V$ is convex (but not necessarily strongly convex), we have the following result: 

\begin{theorem}
[Accelerated convergence in the convex case]
\label{theorem:UL_convex} 
Assume that $V \in C^1(\mathbb{R}^d)$ is convex and fix an initial time $\tau>0$. Let the regularity assumptions from Assumption \ref{ass:FP-regularity} below be satisfied. Set $\gamma_t = \frac{3}{t}$, and let $r_\tau$ denote the Wasserstein-2 optimal transport map from $\operatorname{Law}(Q_\tau)$ to $\pi$. Then, for every $t\ge\tau$, the solution to \eqref{eq:UL}
satisfies
\begin{equation}
\label{eq:convex rate}
    \operatorname{KL}(\operatorname{Law}(Q_t) \| \pi) \leq \frac{\tau^2}{t^2}\operatorname{KL}(\operatorname{Law}(Q_{\tau}, P_{\tau})\, \|\, \Pi_{\tau}),
\end{equation}
where the initial reference measure $\Pi_{\tau}$ is defined by
\begin{align*}
    \Pi_{\tau}(\dd q,\dd p) = \pi\left(\dd q\right)\, \mathcal{N}\big(2 \tau^{-1}(r_{\tau}(q)-q), I_d\big)(\dd p).
\end{align*}
\end{theorem}

\begin{remark}
The specific choice $\gamma_t = \tfrac{3}{t}$ renders \eqref{eq:UL} time-inhomogeneous and will be motivated below in Section \ref{sec:optimisation}.
We note that even though \eqref{eq:UL} is time-inhomogeneous, $\Pi$ is still invariant: $(Q_\tau,P_\tau) \sim \Pi$ implies $(Q_t,P_t) \sim \Pi$ for all $t \ge \tau$. 
\end{remark}

The essence of Theorems \ref{theorem:UL_strongly-convex}
and \ref{theorem:UL_convex}
are so-called \emph{accelerated} rates of convergence. More context is given below in Section \ref{sec:optimisation}; for now, recall that under mild conditions on $V$, overdamped Langevin dynamics,
\begin{align}
\label{eq:OL}
    \dd Q_t = - \nabla V(Q_t)\dd t + \sqrt{2} \dd W_t,
\end{align}
is ergodic with respect to $\pi$ and 
satisfies 
\begin{equation}
\label{eq:overdamped sc}
\KL(\Law(Q_t)\| \pi) \le e^{-2 \sigma t} \KL(\Law(Q_0) \| \pi)
\end{equation}
if $V$ is $\sigma$-strongly convex \cite[Theorems 1.2.26 and 1.2.30]{Chewi26Book}, and
\begin{equation}
\label{eq:overdamped convex}
\KL(\Law(Q_t)\| \pi) \le \frac{ \mathcal{W}_2^2(\Law(Q_0), \pi)}{4t}
\end{equation}
if $V$ is (only) convex \cite[Equation (1.4.14)]{Chewi26Book}. For small $\sigma$, the exponential rate $\sqrt{\sigma}$ in \eqref{eq:sc convergence} improves the overdamped rate $2\sigma$ in \eqref{eq:overdamped sc}. In the convex case, the $t^{-2}$ estimate in \eqref{eq:convex rate} improves the overdamped $t^{-1}$ rate in \eqref{eq:overdamped convex}. Thus, the underdamped Langevin dynamics enjoys accelerated convergence rates. 
\\ 

\textbf{Outline.} In the remainder of this paper, we prove Theorems \ref{theorem:UL_strongly-convex} and \ref{theorem:UL_convex}, but also aim to contribute to the conceptual understanding of this acceleration phenomenon. To this end, we first review connections to convex optimisation in Section \ref{sec:optimisation} and then lay out a game-theoretic framework (together with a roadmap for the proof) in Section \ref{sec:games}. In Section \ref{sec:related}, we discuss related work, while Section \ref{sec:preliminaries} provides preliminaries and assumptions. The main part consists of Sections \ref{sec:main_convex}
and \ref{sec:main_sconvex}, which specify the game-theoretic setting, propose and analyse suitable strategies within this framework, and relate these strategies to underdamped Langevin dynamics. Section \ref{sec:main_convex} assumes that the potential $V$ in the positional marginal \eqref{eq:q_marginal} is convex (leading to Theorem \ref{theorem:UL_convex}), while Section \ref{sec:main_sconvex} focuses on the strongly convex case (leading to Theorem \ref{theorem:UL_strongly-convex}). Section \ref{app:Fenchel Rd} in the Appendix details how accelerated gradient flows on $\mathbb{R}^d$  similarly exhibit a game-theoretic structure, while Section \ref{sec:saddle_point} comments on the saddle-point pattern that underlies this framework. Finally, Sections \ref{app:lemmas} and \ref{sec:proof_L2-distance_convex} state and prove technical and auxiliary lemmas.\\

\textbf{Notation.} Throughout, $\|\cdot\|$ and $\langle \cdot, \cdot \rangle$ denote the Euclidean norm and inner product, respectively. For matrices $A, B \in \mathbb{R}^{d \times d}$, we write $A:B$ for $\Tr(A^TB)$, $A \succ 0$ means that $A$ is symmetric positive definite, and $A \succeq 0$ that $A$ is symmetric positive semidefinite. We denote by $W = (W)_{t \geq 0}$  the standard d-dimensional Brownian motion. For absolutely continuous probability measures $\mu(\dd x) = \rho_{\mu}(x)\dd x$ and $\nu(\dd x) = \rho_{\nu}(x)\dd x$ on $\mathbb{R}^d$ with $\mu \ll \nu$, we define the Kullback-Leibler (KL) divergence of $\mu$ with respect to $\nu$ by
\begin{align*}
\operatorname{KL}\bigl(\mu\,\|\,\nu\bigr)
:= \int_{\mathbb{R}^d}\ln\!\left(
\frac{\rho_{\mu}(x)}{\rho_{\nu}(x)}
\right)\rho_{\mu}(x)\dd x.
\end{align*}
Similarly, we define the (negative) differential entropy of $\mu$ by
\begin{align}
\label{eq:ent}
    \operatorname{Ent}(\mu) := \int_{\mathbb{R}^d} \rho_{\mu}(x)\ln{\rho_{\mu}(x)}\dd x.
\end{align}
Finally, if $\mu$ and $\nu$ have finite second moments, we define their Wasserstein distance by 
\begin{align*}
    \mathcal{W}_2^2(\mu, \nu) := \inf_{\gamma \in \Gamma(\mu, \nu)} \int_{\mathbb{R}^{2d}} \|x-y \|^2\dd\gamma(x,y),
\end{align*}
where $\Gamma(\mu, \nu)$ denotes the set of all couplings between $\mu$ and $\nu$. We often use the same symbol for an absolutely continuous measure and its density with respect to Lebesgue measure, e.g., we write $\mu(\dd x)=\mu(x)\dd x$.

\subsection{Connection to convex optimisation on $\mathbb{R}^d$}
\label{sec:optimisation}
According to Jordan, Kinderlehrer, and Otto \cite{JKO}, the evolution of $\Law(Q_t)$ under the overdamped Langevin dynamics \eqref{eq:OL} follows the Wasserstein gradient flow of $\KL(\cdot \| \pi)$, or equivalently, of the free energy 
\begin{equation}
\label{eq:free energy}
    \mathcal{F}(\rho)  := \int_{\mathbb{R}^d} V\dd\rho + \int_{\mathbb{R}^d} \ln{\rho}\dd \rho = \KL(\rho \| \pi) - \ln Z.
\end{equation}
In other words, under \eqref{eq:OL}, the distribution $\Law(Q_t)$ follows a curve of steepest descent, approaching the minimiser $\pi = \argmin_\rho \mathcal{F}(\rho)$. From this perspective, the Euclidean counterpart to \eqref{eq:OL} is given by the (standard) gradient flow
\begin{align}
\label{eq:gf}
    \dot{q}_t = - \nabla f(q_t),
\end{align}
for a convex objective function $f \in C^1(\mathbb{R}^d)$ with $ \argmin_q f(q) \neq \emptyset$. Drawing the parallel between \eqref{eq:OL} and \eqref{eq:gf} more explicitly, one associates the free energy $\mathcal{F}$ with the convex objective $f$, the distribution $\Law(Q_t) \in \mathcal{P}_2(\mathbb{R}^d)$ with $q_t \in \mathbb{R}^d$, and the Wasserstein-2 geometry with the Euclidean geometry. The analogy also extends to the decay estimates \eqref{eq:overdamped sc} and \eqref{eq:overdamped convex}: For $\sigma$-strongly convex $f$ and denoting $q^* \in \argmin_q f(q)$, we have
\begin{equation}\label{eq:conv_gf_sc}
f(q_t) - f(q^*) \leq  e^{-2\sigma t} \left( f(q_0) - f(q^*)\right), 
\end{equation}
along the solution to \eqref{eq:gf}, see \cite[Proposition 1.1]{scieur2017integration}, 
and for convex (but not necessarily strongly convex) $f$,
\begin{equation}\label{eq:conv_gf_c}
f(q_t) - f(q^*) \le \frac{\| q_0 - q^* \|^2}{4t},
\end{equation}
see \cite[Exercise 2.1]{chewi2026lectures}. Indeed, the convergence results \eqref{eq:overdamped sc} and \eqref{eq:overdamped convex} can be obtained by arguments analogous to those proving \eqref{eq:conv_gf_sc} and \eqref{eq:conv_gf_c}, respectively, compare \cite[Section 1.4.2 and Exercise 2.1]{Chewi26Book} and \cite[Corollary 2.6 and Exercise 1.18]{chewi2026lectures}. \\
\\
\textbf{Accelerated gradient flows.} 
Time discretisations of \eqref{eq:gf} lead to gradient-descent-type algorithms \cite[Chapter 6]{vishnoi2021algorithms}. However, perhaps surprisingly, other first-order algorithms may achieve faster rates of convergence. In particular, the seminal works of Polyak \cite{polyak1964some} and Nesterov \cite{nesterov1983method} have built the foundation towards the corresponding class of so-called \emph{accelerated gradient methods}, see also \cite[Chapter 2 and 4]{daspremont2021acceleration}. Throughout this paper, we focus solely on the corresponding continuous-time models, often referred to as \emph{accelerated gradient flows}, brought forward and extended in work such as \cite{NesterovODE, shi2022understanding, AcceleratedMD_continuous_and_discrete, wilson2021lyapunov}. For a $\sigma$-strongly convex objective $f$, the solution to Polyak's heavy-ball ODE
\begin{subequations}
\label{eq:polyak_ODE}
\begin{align}
    \dot{q}_t & = p_t \\
    \dot{p}_t & = - \nabla f(q_t) - 2\sqrt\sigma p_t,
\end{align}
\end{subequations}
achieves the decay 
\begin{align}\label{eq:conv_Polyak}
    f(q_t) - f(q^*) \leq e^{-\sqrt{\sigma}t}\left(f(q_0) - f(q^*)+ \frac{\sigma}{2}\|q_0+\frac{1}{\sqrt{\sigma}}p_{0} - q^*\|^2\right),
\end{align}
see \cite{polyak1964some, siegel2019accelerated} and Section \ref{sec:polyak}. Similarly, for convex $f$, the solution to Nesterov's ODE 
\begin{subequations}
\label{eq:nesterov_ODE}
\begin{align}
 \dot{q}_t & = p_t \\
 \dot{p}_t & = - \nabla f(q_t) - \frac{3}{t}p_t,  
\end{align}
\end{subequations}
for $t \geq \tau > 0$, leads to
\begin{align}\label{eq:conv_Nesterov}
    f(q_t) - f(q^*) \leq \frac{\tau^2\left(f(q_{\tau}) - f(q^*) \right)+2\|q_{\tau} + \frac{\tau}{2}p_{\tau}-  q^*\|^2}{t^2},
\end{align}
see \cite[Theorem 2.7]{AttouchChbaniPeypouquetRedont2018} and Section \ref{sec:nesterov_ODE}. Importantly, the accelerated gradient flows \eqref{eq:polyak_ODE} and \eqref{eq:nesterov_ODE} are noiseless versions of underdamped Langevin dynamics as considered in Theorems \ref{theorem:UL_strongly-convex} and \ref{theorem:UL_convex}; as a consequence, underdamped Langevin dynamics has been described as an accelerated gradient flow on the space of probability distributions (see, for instance \cite{ma2021nesterov}), in a similar way in which overdamped Langevin dynamics \eqref{eq:OL} is linked to ordinary gradient flows. 

We highlight that, comparing the convergence rates in \eqref{eq:sc convergence} with that in \eqref{eq:conv_Polyak} and in \eqref{eq:convex rate} with that in \eqref{eq:conv_Nesterov}, our theorems reproduce the exact same exponential and quadratic convergence rates, employing the same canonical damping from the continuous-time Euclidean optimisation literature. If $P_t$ is appropriately initialised, the analogy extends even to the precise upper bound, as expressed in the following two corollaries (which should be compared to \eqref{eq:conv_Polyak} and \eqref{eq:conv_Nesterov} for vanishing initial velocity).
\begin{corollary}\label{cor:UL_strongly-convex}
Under the assumptions of Theorem \ref{theorem:UL_strongly-convex}, suppose additionally that $\operatorname{Law}(P_{0}\, | \, Q_0) = \mathcal{N}(0,I_d)$. Then, for $\gamma_t \equiv 2\sqrt{\sigma}$, the solution to \eqref{eq:UL} satisfies, for every $t \geq 0$,
\begin{equation}
    \operatorname{KL}(\operatorname{Law}(Q_t) \| \pi) \leq e^{-\sqrt{\sigma}t}\left(\operatorname{KL}(\operatorname{Law}(Q_{0}) \| \pi) + \frac{\sigma}{2}\mathcal{W}_2^2(\operatorname{Law}(Q_0), \pi) \right).
\end{equation}
\end{corollary}
\begin{corollary}\label{cor:UL_convex}
Under the assumptions of Theorem \ref{theorem:UL_convex}, suppose additionally that $\operatorname{Law}(P_{\tau}\, | \, Q_\tau) = \mathcal{N}(0,I_d)$. Then, for $\gamma_t = \frac{3}{t}$, the solution to \eqref{eq:UL} satisfies, for every $t \geq \tau > 0$,
\begin{equation}
    \operatorname{KL}(\operatorname{Law}(Q_t) \| \pi) \leq \frac{\tau^2\operatorname{KL}(\operatorname{Law}(Q_{\tau}) \| \pi) + 2 \mathcal{W}_2^2(\operatorname{Law}(Q_{\tau}), \pi)}{t^2}.
\end{equation}
\end{corollary}
The proofs of Theorem \ref{theorem:UL_strongly-convex} and Corollary \ref{cor:UL_strongly-convex} can be found in Section \ref{sec:main_sconvex} and those of Theorem \ref{theorem:UL_convex} and Corollary \ref{cor:UL_convex} in Section \ref{sec:main_convex}.

\subsection{Heuristics for describing accelerated gradient flows through Fenchel games}
\label{sec:games}

The accelerated gradient flows \eqref{eq:polyak_ODE} and \eqref{eq:nesterov_ODE} are often interpreted as damped Hamiltonian systems, e.g. \cite{wang2022accelerated, maddison2018hamiltonian} or \cite[Section 3.2.2]{ma2021nesterov}. Here, we will develop and follow an alternative viewpoint: after a change of variables, these dynamics can be understood as coupled strategies in a two-player game, called the \emph{Fenchel} game. This builds directly on the work \cite{fenchelgames}, in which the authors put forward such an interpretation for discrete-time convex optimisation algorithms. We refer the reader to Appendix \ref{app:Fenchel Rd}, in which we extend their work to continuous time, covering \eqref{eq:polyak_ODE} and \eqref{eq:nesterov_ODE}; this may serve as a gentle introduction to Fenchel games, as the main ideas, in particular those underpinning the acceleration phenomenon, already emerge in this setting of unconstrained convex optimisation on $\mathbb{R}^d$. \\
The core concepts in the context of convex $f$ and $V$ will also be explained in the following; modifications to the strongly convex case will be discussed in Section \ref{sec:main_sconvex} and Appendix \ref{sec:polyak}. Since several concepts emerge from the online learning literature, we refer to general references on this subject, such as \cite{hazan2019online, shalevshwartz2012online} and \cite{modernintroductiononlinelearning}. 

\paragraph{Accelerated gradient flows from a game perspective.} 
The main idea of \cite{fenchelgames} is to extend the optimisation problem 
\begin{align}\label{eq:optimisation_problem}
\min_{x \in \mathbb{R}^d}f(x)
\end{align}
to the saddle point problem
\begin{align}
\label{eq: saddle g}
    \min_{x \in \mathbb{R}^d} \max_{y \in \mathbb{R}^d} g_f(x,y)
\end{align}
by introducing the two-player payoff function 
\begin{equation}
\label{eq:g Rd}
g_f(x,y) := \langle x,y\rangle - f^*(y).
\end{equation}
Here, $f^*:\mathbb{R}^d \rightarrow (-\infty, + \infty]$ denotes the Fenchel conjugate  of $f$  (see Section \ref{sec:preliminaries}); crucially, it holds that $\max_{y \in \mathbb{R}^d} g_f(x,y) = f(x)$, connecting \eqref{eq:optimisation_problem} to \eqref{eq: saddle g}. In round $s$ of the game, the x-player plays the move $x_s$, the y-player plays the move $y_s$, and the function $g_f$ serves as a cost for both players: The x-player pays $g_f(x_s,y_s)$ and the y-player pays $-g_f(x_s,y_s)$. Roughly speaking, the x-player aims to minimise $g_f$ (reacting to the y-player's moves $(y_u)_{0 \le u \le s}$), the y-player aims to maximise $g_f$ (reacting to the x-player's moves $(x_u)_{0 \le u \le s}$), and the interplay of both over time identifies the saddle point in \eqref{eq: saddle g}. In particular, combining continuous-time versions of the follow-the-leader algorithm and online gradient descent, both well-established in the online learning literature \cite[Chapter 4]{modernintroductiononlinelearning}, leads to a coupled system that is equivalent to the aforementioned accelerated gradient flows; see Appendix \ref{app:Fenchel Rd}. There are subtleties: for stability reasons, the saddle point corresponds to long-time averages of $x_t$ and $y_t$, weighted in a particular way. These issues will be discussed below (for the stochastic case) and in Appendix \ref{app:Fenchel Rd}.

\paragraph{A Fenchel game for sampling.} In what follows, we develop a stochastic Fenchel game formulation that allows us to analyse underdamped Langevin dynamics with convex potential $V$. Here we give an overview; technical details will be explained in Section \ref{sec:main_convex}. To simplify the notation, we will abbreviate the differential entropy \eqref{eq:ent} of a random variable $X$ with Lebesgue density $\rho_X$ by $\Ent(X) = \Ent(\rho_X)$. \\
\\
Mimicking the construction in \eqref{eq:optimisation_problem}-\eqref{eq:g Rd}, we extend the free energy $\mathcal{F}$ \eqref{eq:free energy} as follows,
\begin{align}
\label{eq:G}
    \mathcal{G}(X,Y) := \mathbb{E}[g(X,Y)] + \operatorname{Ent}(X),
\end{align}
defined for appropriately integrable $\mathbb{R}^d$-valued random variables $X$ and $Y$ (not necessarily independent). Here, as in \eqref{eq:g Rd}, we set
\begin{equation}
    g(x,y):= g_V(x,y) = \langle x, y \rangle - V^*(y),
\end{equation}    
but drop the subscript for notational convenience.
Again, one recovers the free energy $\mathcal{F}$ from $\mathcal{G}$ via 
\begin{align}
    \mathcal{F}(\operatorname{Law}(X)) = \sup_Y \mathcal{G}(X,Y),
\end{align}
according to $\sup_y g(x,y) = V(x)$, which is implied by the properties of the Fenchel conjugate $V^*$ (see Section \ref{sec:preliminaries}). 
The analogue of the saddle point problem \eqref{eq: saddle g} becomes
\begin{equation}
\label{eq:saddle G}
\min_X \max_Y \mathcal{G}(X,Y),
\end{equation}
so that, in broad terms, 
the X-player's aim is to minimise $\mathcal{G}$, while the Y-player aims to maximise $\mathcal{G}$.
\begin{remark}[Optimality conditions]
\label{rem:saddle}
For a fixed random variable $X$, it can be shown that $\argmax_Y \mathcal{G}(X,Y) = \nabla V(X)$, see Appendix \ref{sec:saddle_point}. Furthermore, $X \in \argmin_X \mathcal{F}(\Law X)$ implies $X \sim \pi$. Consequently, we may interpret
\begin{equation}
\label{eq:XY saddle}
X^* \sim \pi, \quad  Y^*  =  \nabla V(X^*)
\end{equation}
as a saddle point in \eqref{eq:saddle G} that the two players strive to find. However, we would like to warn the reader that a rigorous definition of saddle points for functionals of the form \eqref{eq:G} is neither attempted nor needed in this work; see Appendix \ref{sec:saddle_point} for a brief discussion.
\end{remark}
The following strategies, played in round $0 \leq s \leq t$, are not uniquely determined by \eqref{eq:G} or \eqref{eq:saddle G}, but they i) recover underdamped Langevin dynamics \eqref{eq:UL} after a change of variables, and ii) will be shown to be \emph{no-regret} strategies (hence ``good''), in a sense to be made precise below.
\begin{itemize}
\item Given the X-player's moves $(X_u)_{0 \le u \le s}$ up to time $s$, the Y-player adopts a best-in-hindsight strategy (also known as \emph{follow-the-leader} \cite[Section 2.2]{shalevshwartz2012online} or \cite[Algorithm 4]{fenchelgames}):
\begin{equation}
\label{eq:Y FTL}
Y_s := \argmax_Y \left(\frac{1}{A_s} \int_0^s\mathcal{G}(X_u,Y) \, \alpha_u \dd u \right).  
\end{equation}
In this display, $\alpha \in C^1 (\mathbb{R}_{\ge 0},\mathbb{R}_{> 0})$ is a time-dependent weight, and
\begin{equation}
\label{eq:def A}
A_s = \int_0^s \alpha_u \dd u
\end{equation}
is its integral. The precise choice of $\alpha$ will be detailed in Section \ref{sec:main_convex} (and in Section \ref{sec:main_sconvex}) and will ultimately be linked to the friction $\gamma_t$ in \eqref{eq:UL}. By definition, the Y-player's move at time $s$ is obtained by maximising the payoffs $\mathcal{G}(X_u,\cdot)$, averaged over the past according to $\alpha$. Using properties of the Fenchel conjugate (see Section \ref{sec:preliminaries}), the maximisation can be carried out explicitly (see Section \ref{sec:FTL convex}) and leads to 
\begin{align}
    Y_s = \nabla V(\bar{X}_s),
\end{align}
where $\bar{X}_s$ is the weighted time average
\begin{equation}
\label{eq:X bar}
\bar{X}_s = \frac{1}{A_s} \int_0^s X_u \alpha_u \dd u.
\end{equation}

\item Given the Y-player's move $Y_s$ (fixed for the sake of this argument and as a rule of the game), the X-player updates her strategy in a way that resembles a Wasserstein gradient flow update on
\begin{equation}\label{eq:functional_x-player}
\rho \mapsto \int_{\mathbb{R}^d} \langle x, Y_s \rangle \rho (\dd x) + \Ent(\rho),
\end{equation}
which is \eqref{eq:G} for fixed (deterministic) $Y = Y_s$, expressed in terms of the density $\rho_X$ associated with $X$. The usual Wasserstein gradient flow heuristics lead to a drift $-Y_s$ (coming from the first term) and Brownian noise (induced by the entropic term), so that the X-player's strategy takes the form 
\begin{equation}
\label{eq:X OMD}
\dd X_s = - \beta^Y_s Y_s \dd s + \beta^W_s \dd W_s, 
\end{equation}
with time-dependent coefficients $\beta_s^Y$ and $\beta_s^W$ that depend on the choice of $\alpha$ (see Sections \ref{sec:main_convex} and \ref{sec:main_sconvex}). Since the functional \eqref{eq:functional_x-player} is time-dependent due to the presence of $Y_s$, the X-player's strategy resembles, more precisely, an online gradient flow.
\end{itemize}

After an appropriate change of variables (see also Sections \ref{sec:main_convex} and \ref{sec:main_sconvex}), the coupled dynamics defined by \eqref{eq:Y FTL}, \eqref{eq:X bar} and \eqref{eq:X OMD} transform into underdamped Langevin dynamics \eqref{eq:UL}; more specifically, the position $Q_s$ is identified with the time average $\bar{X}_s$ of the X-player's move, while the velocity $P_s$ is a time-dependent linear combination of $X_s$ and $\bar{X}_s$:
\begin{equation}
\label{eq:QP-X}
Q_s = \bar{X}_s, \qquad P_s = \frac{\alpha_s}{A_s} (X_s - \bar{X}_s), \quad s > 0. 
\end{equation}
In particular, $Q_s = \bar{X}_s$ corresponds to the common practice of returning an average of the strategies played in online learning, especially in online convex optimisation, known as online-to-batch conversion, see e.g. \cite[Section 15.2]{modernintroductiononlinelearning}). In the particular context of saddle point problems, online-to-batch conversion stabilises otherwise potentially unstable ascent-descent dynamics.

\paragraph{Assessing strategies via competitors and regret.} Beyond a novel interpretation of underdamped Langevin dynamics, the above framework also allows us to quantify its convergence rates. However, to prove Theorems \ref{theorem:UL_strongly-convex} and \ref{theorem:UL_convex} using the representation of \eqref{eq:UL} through \eqref{eq:Y FTL}-\eqref{eq:X OMD}, an additional ingredient is needed: the introduction of competitors who play benchmark strategies. These competing strategies can be motivated by the optimality conditions in Remark \ref{rem:saddle}. If the X- and Y-players perform sufficiently close to them in terms of accumulated loss (defined below), their joint actions will over time converge, with a rate depending on the comparative performance of their strategies (see \eqref{eq:regret_to_suboptimality} below).\\
\\
Assume we are at the final time $t$ of the game and the X- and Y-players have played the actions $\{X_s\}_{s \leq t}$ and $\{Y_s\}_{s \leq t}$. Their competitors have played $\{X^{\text{comp}}_s\}_{s \leq t}$ and $\{Y^{\text{comp}}_s\}_{s \leq t}$. We write the accumulated time-weighted losses of the players as
\begin{subequations}
\label{eq:losses_intro}
\begin{align}
\label{eq:X loss}
\mathcal{L}^X_t & = \int_0^t \mathbb{E}[g(X_s,Y_s)]\,\alpha_s \dd s + A_t \Ent(\bar{X}_t),
    \\
\label{eq:Y loss}
\mathcal{L}^Y_t & = - \int_0^t \mathbb{E}[g(X_s,Y_s)]\, \alpha_s \dd s.
\end{align}
\end{subequations}
These are constructed from \eqref{eq:G}, but there is a subtlety concerning the entropic term in \eqref{eq:X loss}, see Remark \ref{rem:FISTA} below. Notice that the entropic term could be dropped from \eqref{eq:Y loss} since it does not depend on $Y$. We highlight that these forms of $\mathcal{L}_t^X$ and $\mathcal{L}_t^Y$ only apply when $V$ is convex (but not necessarily strongly convex), i.e. for the proof of Theorem \ref{theorem:UL_convex}. For strongly convex potentials, slight adaptations must be made; see Section \ref{sec:main_sconvex}. The actions of the competitors are recorded similarly as 
\begin{subequations}
\label{eq:comp losses}
\begin{align}
   \label{eq:X comp loss} \mathcal{L}^{\text{comp}\,X}_t & = \int_0^t \mathbb{E}[g( X^{\text{comp}}_s,Y_s)] \alpha_s \dd s + A_t \Ent\left(X^{\text{comp}}_t\right), \\
    \label{eq:Y comp loss} \mathcal{L}^{\text{comp}\,Y}_t & = - \int_0^t \mathbb{E}[g(X_s, Y^{\text{comp}}_t)] \alpha_s \dd s.
\end{align}
\end{subequations}
Now, the total performance measures of the X- and Y-players are given by their so-called \textit{regrets} with respect to their competitors: 
\begin{subequations}
\label{eq:regrets_intro}
\begin{align}
\label{eq:X regret abstract}
    \operatorname{Reg}_t^X = \mathcal{L}^X_t - \mathcal{L}^{\text{comp}\,X}_t,\\
    \operatorname{Reg}_t^Y = \mathcal{L}^Y_t - \mathcal{L}^{\text{comp}\,Y}_t.
\end{align}
\end{subequations}
We highlight that the concept of regrets is very well-established in the online learning literature, see e.g. \cite{hazan2019online, modernintroductiononlinelearning, shalevshwartz2012online}. For appropriately defined competitors, Propositions \ref{theorem:sub-optimality} and \ref{theorem:sub-optimality_sc} below show that such regrets can  quantify the sub-optimality of $\operatorname{Law}(\bar{X}_t)$ with respect to the target distribution $\pi(\dd q) \propto e^{-V(q)}\dd q$ in terms of the free energy:
\begin{align}\label{eq:regret_to_suboptimality}
    \mathcal{F}(\operatorname{Law}(\bar{X}_t)) - \mathcal{F}(\pi) \leq \frac{\operatorname{Reg}_t^X +  \operatorname{Reg}_t^Y}{A_t}.
\end{align}
Since $\operatorname{KL}(\operatorname{Law}(\bar{X}_t) \| \pi) = \mathcal{F}(\operatorname{Law}(\bar{X}_t)) - \mathcal{F}(\pi)$, these regrets equivalently control the KL divergence of $\operatorname{Law}(\bar{X}_t)$ with respect to $\pi$. Strategies of the players for which $\operatorname{Reg}_t^X$ and $\operatorname{Reg}_t^Y$ are bounded are often called \textit{no-regret} strategies, see e.g. \cite[Section 3.1]{fenchelgames}, and are thus of particular interest because they allow one to establish convergence. Setting $Q_t = \bar{X}_t$ as in \eqref{eq:QP-X} directly extends this result to the physical underdamped Langevin system \eqref{eq:UL}. 
\paragraph{Proof roadmap.} The above considerations imply that the proofs of Theorems \ref{theorem:UL_strongly-convex} and \ref{theorem:UL_convex} can be broken down into three parts. First, we need to define competitors such that the regrets with respect to them bound the optimality gap in $\mathcal{F}$ as in \eqref{eq:regret_to_suboptimality}. Then, for proposed strategies of the X- and Y-player, we need to establish appropriate upper bounds on $\operatorname{Reg}_t^X$ and $\operatorname{Reg}_t^Y$. Finally, we need to relate the combined strategies to the physical underdamped Langevin systems via the change of variables in \eqref{eq:QP-X}.

\begin{remark}[Competitor constraints]
Nontrivial bounds on the regrets in \eqref{eq:regrets_intro} will only be possible if the actions of the competitors are appropriately constrained. In line with the usual set-up in online learning \cite{modernintroductiononlinelearning}, the Y-competitor will only be allowed to play a single, time-independent action, but he is allowed to make his move $Y^{\text{comp}}_t$ in hindsight, at the end of the game in round $t$. Similarly, the law of the X-player's action will be time-independent (but time-dependent as a random variable). Full details will be given in Section \ref{sec:main_convex}.
\end{remark}

\begin{remark}[Distinctive features of the sampling game]\label{rem:FISTA}
We will be able to analyse the Y-player's strategy and regret \emph{pathwise}, relying fairly straightforwardly on analogous arguments for the optimisation games on $\mathbb{R}^d$ detailed in Section \ref{sec:nesterov_ODE} and \ref{sec:polyak}, see also \cite{fenchelgames}.  The treatment of the X-player comes with more subtleties.  It is non-standard that $\mathcal{L}_t^X$ in \eqref{eq:X loss} incorporates the entropy of the time-averaged actions $\bar{X}_t$, but this construction is forced by the lack of appropriate convexity of the entropy functional with respect to the mixture of random variables (indeed, $\Ent(\lambda X_0 + (1-\lambda) X_1) \nleq \lambda \Ent(X_0) + (1-\lambda) \Ent(X_1)$). As a result, the analysis of the X-player's strategy becomes more involved, with the proof relying on Optimal Transport, see Section \ref{convex:noisy}. More fundamentally, the splitting in $\mathcal{G}$ (treating the potential and the entropic part of $\mathcal{F}$ differently) and the form of $\mathcal{L}_t^X$ suggest a connection to proximal methods, such as FISTA; see \cite{beck2009fast} or \cite[Section 4.5.3]{fenchelgames}. This is complementary to previous work on the (accelerated) minimisation of $\mathcal{F}$ which does not make such a distinction; see, e.g., \cite{wang2022accelerated, pmlr-v97-taghvaei19a, onlinelearningtransportminimal}. 
\end{remark}

\begin{remark}[Generalisation and initial conditions]
To accommodate general initial conditions and obtain the optimal rates from Theorem \ref{theorem:UL_convex}, the definitions in \eqref{eq:def A} and \eqref{eq:X bar} need slight extensions, which are detailed in Sections \ref{sec:main_convex} and \ref{sec:main_sconvex}. See in particular Definition \ref{def:weights} and Remark \ref{remark:technical_tau}.
\end{remark}
 
\subsection{Related work}
\label{sec:related}

\paragraph{Convergence analysis of underdamped Langevin dynamics.}
Since the early 2000s, the understanding of convergence to equilibrium for underdamped Langevin dynamics has deepened considerably, both in identifying the mechanisms that yield convergence despite diffusion acting only in the velocity variable, and in developing quantitative convergence estimates. Foundational approaches are based on hypocoercivity and related spectral methods \cite{DesvillettesVillani2001,HerauNier2004,Villani2009Hypocoercivity,DolbeaultMouhotSchmeiser2015}. More recent work has developed quantitative approaches based on functional inequalities, including space-time log-Sobolev inequalities \cite{LiLu2026SpaceTimeLSI} and space-time Poincar\'e or Poincar\'e--Lions inequalities \cite{AlbrittonArmstrongMourratNovack2024,cao2023explicit,BrigatiStoltz2025,BrigatiEtAl2024ExplicitConvergence}, as well as related lifting and flow-Poincar\'e perspectives \cite{EberleEtAl2025Convergence,EberleLorler2026Nonreversible,BrigatiLorlerWang2026}. Other work employs coupling-based arguments \cite{Eberle2017CouplingsAQ,pmlr-v75-cheng18a,DalalyanRiouDurand2020}. Interestingly, such analyses frequently exploit mixed, or ``twisted'', position--velocity coordinates, for instance through variables of the form $q+cp$, or their analogues for differences between coupled trajectories \cite{Eberle2017CouplingsAQ,pmlr-v75-cheng18a,DalalyanRiouDurand2020}. More recently, \cite{AltCheZha26SCIV} explicitly adopt this terminology and develop a time-dependent twisted coordinate system. Our change of variables \eqref{eq:QP-X} has a closely related form, but arises naturally from the Fenchel-game representation.
\\

The closest result to Theorem~\ref{theorem:UL_strongly-convex} is due to \cite{lu2026sharp}. Assuming a log-Sobolev inequality with parameter \(\sigma > 0\) for the invariant positional marginal \(\pi(\dd q) \propto e^{-V(q)}\dd q\), it is shown that suitably damped underdamped Langevin dynamics satisfy
\begin{equation}
\label{eq:lu rate}
\KL(\Law(Q_t,P_t)\|\Pi)
\le
C e^{-c\sqrt{\sigma}t}
\KL(\Law(Q_0,P_0)\|\Pi),
\end{equation}
for universal constants \(c<1\) and \(C>1\). Importantly, \eqref{eq:lu rate} controls convergence of the joint law of \((Q_t,P_t)\), whereas our bound \eqref{eq:sc convergence} is stated directly for the positional marginal. Thus, while the two results control different quantities, both exhibit the same characteristic \(\sqrt{\sigma}\) scaling of the decay rate. Moreover, both analyses rely on optimal-transport inequalities and related identities. The arguments in \cite{lu2026sharp}, however, follow a hypocoercive/kinetic perspective, whereas ours arise directly from the optimisation and Fenchel-game viewpoint, yielding a constructive convergence analysis specifically on the positional marginal and without an additional multiplicative prefactor.

\paragraph{Sampling via descent dynamics on spaces of measures.}
Our proof strategy draws on the now-standard analogy between sampling and optimisation through descent dynamics on probability measures. Following the variational formulation of the Fokker--Planck equation in \cite{JKO}, Wasserstein gradient flows have become a central tool in this area; see, e.g., \cite{ambrosio2008gradient,garciatrillos2023optimization,chen2023sampling}. More generally, constructing a flow-based sampler entails choosing an objective functional and a geometry, or Onsager operator, on probability space, followed by a numerical approximation of the resulting flow. The KL divergence is distinguished within the classes of $f$-divergences and Bregman divergences because its gradient flows do not require the target normalising constant \cite{chen2023sampling,crucinio2025unique}. With the $2$-Wasserstein geometry, the KL gradient flow is the Fokker--Planck equation associated with overdamped Langevin dynamics \cite{JKO}; with a kernelised Stein geometry, it gives the mean-field flow underlying SVGD \cite{liu2016svgd,liu2017svgdflow}. Discrete-time variational and splitting viewpoints have led both to new sampling algorithms \cite{wibisono2018sampling,bernton2018langevin,chen2022improvedanalysisproximalalgorithm} and to analyses or reinterpretations of existing methods \cite{durmus2019analysis,chopin2024connection}. We also note that \cite{onlinelearningtransportminimal} uses online learning to describe and analyse Wasserstein gradient flows in continuous and discrete time.

A related but distinct line of work concerns inertial or accelerated dynamics for objectives on spaces of probability measures. Because Polyak's ODE \eqref{eq:polyak_ODE} and Nesterov's ODE \eqref{eq:nesterov_ODE} are paradigmatic accelerated gradient flows, with accelerated convergence rates as detailed in Section~\ref{sec:optimisation}, it is natural to ask what their counterparts on spaces of probability measures are, and what theoretical and practical insights can be gained from these analogies. Hamiltonian dynamics on Wasserstein space were developed in \cite{ambrosio2008hamiltonian,chow2020wassersteinhamiltonian}. Deterministic accelerated transport and particle flows have been studied in \cite{pmlr-v97-taghvaei19a,liu2019acceleratingparvi,wang2022accelerated,Chen2025Accelerating,tang2026nesterov}; see also \cite{stein2026acceleratedsvgf} for an accelerated Stein variational flow. In the distinct, sampling-specific stochastic direction, \cite{ma2021nesterov} interprets underdamped Langevin dynamics as accelerated KL descent on phase space, while \cite{li2022hessianfree} constructs a Nesterov-inspired diffusion and corresponding discretisations. Several of these analyses adapt Lyapunov functionals from finite-dimensional acceleration, such as those in \cite{wilson2021lyapunov}.

Most closely related in terms of our explicit rates, \cite{Chen2025Accelerating} and \cite{pmlr-v97-taghvaei19a} lift the dynamics to phase space and state convergence results for the positional marginal, but consider deterministic evolutions. In contrast, our game-based approach separates the deterministic forcing from the stochastic forcing, thus retaining stochastic dynamics while restricting the analysis to the positional marginal. More explicit comparisons to \cite{wang2022accelerated, pmlr-v97-taghvaei19a, lu2026sharp} are given in Remark~\ref{rem:previous work technical}.

\paragraph{Accelerated gradient methods and saddle points} 
Nesterov's celebrated accelerated gradient method
for the minimisation of an unconstrained convex function \cite{nesterov1983method} is often understood via estimate sequences \cite[Section 2.2.1]{nesterov2018lectures}. Beyond several other interpretations, we highlight that the method can be viewed through the template of optimal gradient methods \cite[Section 4.3]{daspremont2021acceleration}. Indeed, once one seeks a first-order method with the best worst-case
guarantee, a minimax structure arises naturally: the algorithm is chosen to
minimise its error while an adversarial problem instance seeks to maximise
it. This perspective underlies performance-estimation-based approaches to the
design of optimized gradient methods
\cite{drori2014performance,kim2016optimized}. The description of accelerated gradient methods as strategies in a (Fenchel) game \cite[Section 4.5]{fenchelgames} is similar in spirit, as it is also based on a saddle point formulation. Since this latter work falls under the umbrella of online learning, we also refer to the discussion of solving saddle point problems using online convex optimisation in \cite[Chapter 15]{modernintroductiononlinelearning}.

\section{Assumptions and preliminaries}
\label{sec:preliminaries}

Throughout, recall that the variables $Q_t$ and $P_t$ refer to the physical system, i.e., they follow the underdamped Langevin dynamics in \eqref{eq:UL}, while the variables $X_t, \bar{X}_t, Y_t$ are those associated with the sampling game. Moreover, the following standing assumption is imposed. 

\begin{assumption}[Regularity]
\label{ass:FP-regularity}
Fix an initial time $\tau\ge 0$, and assume that $V \in C^1(\mathbb{R}^d)$ is convex and $e^{-V}$ is integrable,
\[
Z:=\int_{\mathbb R^d} e^{-V(q)}\dd q<\infty.
\]
Furthermore,
\begin{enumerate}
\item[i)]
suppose that \eqref{eq:UL} admits a non-explosive solution $(Q_t,P_t)$ with strictly positive Lebesgue density $\widehat{\rho}_t$, satisfying the associated Fokker--Planck equation 
\begin{equation}
\label{eq:FKP}
    \partial_t\widehat{\rho}_t
    =
    -\nabla_q\cdot(p\widehat{\rho}_t)
    +\nabla_p\cdot
    \bigl((\nabla V(q)+\gamma_t p)\widehat{\rho}_t\bigr)
    +\gamma_t\Delta_p\widehat{\rho}_t,
\end{equation}
classically on $[\tau,\infty)\times\mathbb R^{2d}$.
For every $T>\tau$, assume that there exists a constant $C_T>0$ such that
\begin{equation}
\label{eq:FKP-integrability1}
\sup_{t\in[\tau,T]}
    \int_{\mathbb R^{2d}}
    e^{C_T(\|q\|^2+\|p\|^2)}
    \left(
        \widehat{\rho}_t(q,p)
        +\|\nabla\widehat{\rho}_t(q,p)\|
        +\|\nabla_p^2\widehat{\rho}_t(q,p)\|
    \right)
    \dd q \dd p
    <\infty
\end{equation}
and
\begin{equation}
\label{eq:FkP-integrability2}
    \sup_{t\in[\tau,T]}
    \int_{\mathbb R^{2d}}
    \|\nabla V(q)\|^2
    \widehat{\rho}_t(q,p)
    \dd q\dd p
    <\infty,
\end{equation}
and assume that the conditional Fisher information is locally integrable in time:
\begin{equation}
\label{eq:FP-fisher-regularity}
    \int_{\tau}^T
    \int_{\mathbb R^{2d}}
    \left\|
        \nabla_p
        \ln
        \frac{\widehat{\rho}_t(q,p)}
             {\widehat{\pi}(q,p)}
    \right\|^2
    \widehat{\rho}_t(q,p)
    \dd q\dd p\dd t
    <\infty,
\end{equation}
\item[ii)] suppose that the marginal velocity 
\begin{equation}
\label{eq:def v}
v_t(q) := \mathbb{E} [P_t | Q_t = q], \qquad t>\tau,
\end{equation}
satisfies $v \in C^1((\tau, \infty) \times \mathbb{R}^d,\mathbb{R}^d)$ and generates a characteristic flow $\Phi_{t,s}$: On every compact subinterval $I \subset (\tau,\infty)$, the ODE 
\begin{equation}
\label{eq:flow ODE}
\frac{\dd}{\dd s}\Phi_{t,s}(q)
=
v_s(\Phi_{t,s}(q)),
 \qquad
\Phi_{t,t}(q)=q,
\end{equation}
is well-posed for $s,t \in I$, we have $(\Phi_{t,s})_{\#} \rho_t = \rho_s$ with $\rho_t = \Law(Q_t)$, and the acceleration 
\begin{equation}
\label{eq:acceleration}
a_t(q) := \partial_t v_t(q) + \nabla v_t(q) v_t(q)
\end{equation}
is well-defined and satisfies
\begin{equation}
\label{eq:acceleration bound}
 \sup_{t\in I}
    \int_{\mathbb R^d}
    \| a_t(q)\|^2\,\rho_t(\dd q)
    <\infty. 
\end{equation}
\end{enumerate}
\end{assumption}

\begin{remark}
The assumptions in part (i) are made in particular to justify integration by parts in the proofs of Propositions \ref{prop:xplayer_convex} and \ref{prop:xplaer_sconvex}, and to ensure that intermediate quantities (such as entropy and Fisher information) are finite. These conditions are expected under
standard smoothness, coercivity, and growth assumptions on the
potential (see, for example, \cite{Chaudru2023,HerauNier2004}) and could likely be relaxed with more (technical) effort. Part (ii) is only used to circumvent potential non-differentiability of the curve of optimal transport maps $s \mapsto r_s$ in the proofs of Lemmas \ref{lemma:L2-distance_convex} and \ref{L2-distance_sc}; see also Remark \ref{rem:Brenier regularity}. The assumptions in part (ii) can be established under appropriate conditions on $V$ according to \cite[Assumption 2.2]{lu2026sharp}.
\end{remark}

\paragraph{Basic Fenchel properties.}
The Fenchel conjugate $V^*:\mathbb{R}^d \rightarrow (-\infty, + \infty]$ is defined by 
\begin{equation}
\label{eq:def Fenchel}
        V^*(y) = \sup_{x \in \mathbb{R}^d} \{\langle x,y\rangle - V(x)\}.
\end{equation}
Following the literature on convex conjugates or Legendre transforms, see \cite[Section 12]{Rockafellar_1970} or \cite[Chapter 3]{carlier_book}, we recall their following basic properties, valid for all $(x,y) \in \mathbb{R}^{2d}$:
\begin{enumerate}\label{FY_In}
    \item[(F1)] The Fenchel-Young Inequality:
    \begin{equation}
    \label{eq:Fenchel IE}
        V(x) + V^*(y) \geq \langle x,y\rangle.
    \end{equation}
    \item[(F2)] The biconjugate is the identity: 
    \begin{align}\label{eq:biconjugate_id}
       \sup_{y \in \mathbb{R}^d}\{\langle x,y\rangle - V^*(y)\}=  V(x).
    \end{align}
    \item[(F3)] The Fenchel-Young Equality: 
    \begin{align}\label{FY_Eq}
        y = \nabla V(x) \iff x \in \partial V^*(y) \iff \langle x,y\rangle = V(x) + V^*(y).
    \end{align}
    \item[(F4)] In particular,  
    \begin{align}\label{gradient_maximises}
    \{\nabla V(x)\}  = \argmax_{y \in \mathbb{R}^d}\{\langle y,x\rangle - V^*(y)\}.
    \end{align}
\end{enumerate}

\paragraph{Weights and averages.}
We finish this section by generalising \eqref{eq:def A} and \eqref{eq:X bar}:

\begin{definition}\label{def:weights}
Fix an initial time $\tau \geq 0$, a corresponding initial weight $A_{\tau} > 0$ and initial average $\bar{X}_{\tau} \in \mathbb{R}^d$, as well as a positive weighting function $\alpha \in C^1([\tau,\infty),\mathbb{R}_{> 0})$. We define
\begin{enumerate}
\item[i)] the integrated weight
\begin{equation}
\label{eq:A def complicated}
A_t := A_{\tau} + \int_{\tau}^t \alpha_s \, \dd s, \qquad t \ge {\tau},
\end{equation}
\item[ii)] for any X-player trajectory $(X_s)_{\tau \le s\leq t}$, the corresponding time average
\begin{equation}
\label{eq:X bar complicated}
\bar{X}_t := \frac{A_{\tau}}{A_t}\bar{X}_{\tau} + \frac{1}{A_t} \int_{\tau}^t X_s \alpha_s \dd s, \qquad t \geq \tau.
\end{equation}
\end{enumerate}
\end{definition}
\begin{remark}\label{remark:technical_tau}
Notice that for $\tau = 0$ and  $A_{\tau} = 0$, the definitions in \eqref{eq:A def complicated} and \eqref{eq:X bar complicated} reduce to \eqref{eq:def A} and \eqref{eq:X bar}. The reasons for allowing $\tau > 0$ and $A_{\tau} >0$ are mostly technical. First, these choices allow us to circumvent the singularity at $t=0$ of $\gamma_t = \frac{3}{t}$ in Theorem \ref{theorem:UL_convex}. Second, note that \eqref{eq:def A} implies $A_0 = 0$, which makes it challenging to extend \eqref{eq:QP-X} to $t = 0$ or to define $\bar{X}_0$ in \eqref{eq:X bar}. We note in passing that $\bar{X}_t$ as defined in \eqref{eq:X bar complicated} maintains the interpretation of an average: indeed $\bar{X}_t = \int_{[\tau,t]} \hat{X}_s \, \mu_t (\dd s)$, for the probability measure  $ \mu_t(\dd s) = \frac{A_{\tau}}{A_t}\delta_{\tau}(\dd s) + \frac{1}{A_t}\alpha_s \dd s$ on $[\tau,t]$, and initial value extension $\hat{X}_s = X_s$ (for $s > \tau$), $\hat{X}_{\tau} = \bar{X}_{\tau}$.
\end{remark}

\section{Continuous-time Fenchel games for log-concave sampling}\label{sec:main_convex}

In this section, we present the proof of Theorem \ref{theorem:UL_convex}. Section \ref{sec:main_sconvex} makes appropriate adaptations for the case where $V \in C^1(\mathbb{R}^d)$ is $\sigma$-strongly convex and establishes Theorem \ref{theorem:UL_strongly-convex}. Recall the broad strategy outlined in Section \ref{sec:games}: We introduce two fictitious players, called the X- and the Y-player. Jointly, their dynamics approach the coupled state in \eqref{eq:XY saddle}, which may be thought of as a saddle point of \eqref{eq:G}. The quality of the players' strategies is compared against competitors and measured in terms of regrets $\Reg_t^X$ and $\Reg_t^Y$. Estimates on these regrets will then imply Theorem \ref{theorem:UL_convex} via \eqref{eq:regret_to_suboptimality}. \\
\\
Throughout this section, we start the game at $\tau > 0$ and fix a corresponding initial weight $A_\tau > 0$ as well as a positive weighting function $\alpha \in C^1([\tau,\infty),\mathbb{R}_{> 0})$ to construct the integrated weight $A_t$ and averaged variable $\bar{X}_t$ as in Definition \ref{def:weights}. The initial average $\bar{X}_{\tau}$ in \eqref{eq:X bar complicated} is chosen with  $\operatorname{Law}(\bar{X}_{\tau}) = \operatorname{Law}(Q_{\tau})$. \\
\\
We start by specifying the competitors and then the corresponding regrets  \eqref{eq:regrets_intro}. The losses of the players (denoted by $\mathcal{L}_t^X$ and $\mathcal{L}_t^Y$ in \eqref{eq:losses_intro}) present no particular difficulty, although nuanced weightings in the spirit of Definition \ref{def:weights} will apply.  To define the competing losses \eqref{eq:comp losses}, we need to specify the competitors:
\begin{itemize}
\item In line with established conventions in the theory of online learning, e.g. \cite[Chapter 1]{modernintroductiononlinelearning}, the Y-player is compared against the best \emph{fixed competitor in hindsight} 
\begin{subequations}
\label{eq:Y comp convex}
\begin{align}
\label{eq:Y comp def}
Y^\text{comp}_t & := \argmin_Y  \left\{A_\tau\left(\mathbb{E}\left[V^*(Y)-\langle \bar{X}_\tau,Y\rangle\right]\right) +\int_\tau^t \mathbb{E}\left[V^*(Y)-\langle X_s,Y\rangle \right]\alpha_s\dd s\right\}  
\\
& = \argmin_Y \left \{ A_t \mathbb{E} \left[ V^*(Y) - \langle \bar{X}_t,Y \rangle   \right]\right\},
\label{eq:Y comp simplified}
\end{align}
\end{subequations}
where the loss in \eqref{eq:Y comp def} is augmented by the $A_\tau$-contribution according to Definition \ref{def:weights}, and the simplification in \eqref{eq:Y comp simplified} follows directly from \eqref{eq:X bar complicated}. We emphasise that $Y^{\text{comp}}_t$ is chosen \emph{in hindsight} (i.e., for a given terminal time $t$ and knowing $(X_s)_{\tau \le s \le t}$), but is \emph{fixed} (i.e., has the same strategy in all rounds $s \leq t$). However, it is in general random, responding to the stochasticity of $(X_s)_{\tau \le s \le t}$. 
\item Following Remark \ref{rem:saddle}, it is natural to enforce $X_s^{\text{comp}} \sim \pi$, for all $s \in [\tau,t]$, so that at the level of distributions, the competitor strategy does not depend on time. Notice that  in \eqref{eq:X comp loss}, the entropic term depends on $X_s^{\text{comp}}$ only through its law, and so to maximise $\mathcal{L}^{\text{comp} \, X }_t$ (equivalently, to minimise $\Reg_t^X$, and hence, to optimise the bound \eqref{eq:regret_to_suboptimality}), we need to set
\begin{equation}
\label{eq:X comp def}
X_s^{\text{comp}} = \argmax_{X \sim \pi} \mathbb{E} [\langle X, Y_s \rangle],
\end{equation}
where $Y_s$ is prescribed. We remark that defining \eqref{eq:X comp def} selects an optimal transport coupling since 
\begin{align*}
    \sup_{X \sim \pi}\mathbb{E}[\langle X, Y_s\rangle] = C_s - \frac{1}{2}\mathcal{W}_2^2(\pi, \operatorname{Law}(Y_s)),
\end{align*}
where $C_s \in \mathbb{R}$ only depends on the fixed marginal laws of $X$ and $Y_s$.
\end{itemize}
The above considerations are formalised as follows:
\begin{definition}[Regrets for the convex setting]
\label{def:regrets}
For played actions $\{(X_s,Y_s)\}_{\tau \leq s\leq t}$ and abbreviating $\rho_t = \operatorname{Law}(\bar{X}_t)$, we define the regrets at time $t\geq\tau$ as 
\begin{subequations}
\label{eq:X regret}
\begin{align}   
\label{eq:X regret main}
\Reg_t^{X}  = & \int_{\tau}^t \mathbb{E}\left[\langle X_s,Y_s\rangle \right] \alpha_s \dd s + A_t\operatorname{Ent}\left(\rho_t\right) -\int_\tau^t\sup_{X \sim \pi}\mathbb{E}\left[\langle X,Y_s\rangle\right] \alpha_s \dd s - A_t\operatorname{Ent}\left(\pi\right)
\\
\label{eq:RegX offset}
& + A_\tau (\Ent(\pi) - \Ent(\rho_\tau))
\end{align}
\label{conv:x_regret}
\end{subequations}
and
\begin{subequations}
\label{eq:Y regret}
\begin{align}
\label{eq:Y player loss}
\Reg_t^{Y} & = A_\tau(\mathbb{E}\left[V^*(Y_\tau)-\langle \bar{X}_\tau,Y_\tau\rangle\right]) + \int_\tau^t\mathbb{E}\left[V^*(Y_s)-\langle X_s,Y_s\rangle \right]\, \alpha_s \dd s \\ & 
\label{conv:y_regret}
\;\;\;\ - \inf_Y \left \{ A_t \mathbb{E} \left[ V^*(Y) - \langle \bar{X}_t,Y \rangle   \right]\right\}. \end{align} \end{subequations}
\end{definition}

\begin{remark}
The X-player's regret in \eqref{eq:X regret main} follows the structure \eqref{eq:X regret abstract}, while the term in 
\eqref{eq:RegX offset} is a constant, included in order to ensure that $\Reg_\tau^X  = 0$. The Y-player's regret \eqref{eq:Y regret} includes an $A_\tau$-correction as suggested by Definition \ref{def:weights} ii). 
\end{remark}
\begin{remark}[Entropic games]
Definition \ref{def:regrets} establishes the rules of the game: Up to $A_\tau$-corrections, the players suffer instantaneous losses $\langle X_s, Y_s \rangle$ and $V^*(Y_s) - \langle X_s, Y_s \rangle$, appropriately weighted and in expectation. This structure is familiar from online learning, see, for example, \cite{kwon2014continuous} for a continuous-time formulation. The entropic term in \eqref{eq:X regret} is nonstandard: \emph{after the end of the game}, at time $t$, the X-player pays an additional penalty associated with her average action $\bar{X}_t$, encouraging randomised strategies. We believe that games of this exact type have not been studied, but note that exponential weight algorithms \cite{hoeven2018many} share certain similarities.
\end{remark}

Based on Definition \ref{def:regrets}, we have the following central regret-to-suboptimality reduction:
\begin{proposition}\label{theorem:sub-optimality}
Denoting $\rho_t = \Law(\bar{X}_t)$, the regrets from Definition \ref{def:regrets} satisfy 
\begin{align*}
   \KL(\rho_t \| \pi) =  \mathcal{F}(\rho_t) -  \mathcal{F}(\pi) \leq \frac{\Reg_t^{X} + \Reg_t^{Y}}{A_t} + \frac{A_{\tau}}{A_t}
\left(\mathcal{F}(\rho_\tau) -  \mathcal{F}(\pi) \right), \qquad  t \geq \tau > 0.
\end{align*}
\end{proposition}

\begin{proof}[Proof of Proposition \ref{theorem:sub-optimality}]
The proof follows the lines of \cite[Theorem 2]{fenchelgames}. We start by adding the regrets:
\begin{align*}
    \Reg_t^{X} + \Reg_t^{Y}
    & = -\int_\tau^t \sup_{X \sim \pi}\mathbb{E}\left[\langle X,Y_s\rangle\right]\alpha_s \dd s
    + A_\tau(\mathbb{E}\left[V^*(Y_\tau)-\langle \bar{X}_{\tau}, Y_{\tau}\rangle\right])
    + \int_\tau^t \mathbb{E}\left[V^*(Y_s)\right]\alpha_s \dd s \\
    & \;\;\;- \inf_Y \left \{ A_t \mathbb{E} \left[ V^*(Y) - \langle \bar{X}_t,Y \rangle   \right]\right\} + A_t(\Ent\left(\rho_t\right) - \Ent\left(\pi\right)) -  A_\tau (\Ent(\rho_\tau) - \Ent(\pi)).
\end{align*}
We now reformulate and estimate the appearing terms. First, notice that, by \eqref{eq:biconjugate_id}, the infimum over $Y$ can be rewritten as
\begin{align*}
    \inf_Y \left \{ A_t \mathbb{E} \left[ V^*(Y) - \langle \bar{X}_t,Y \rangle   \right]\right\}
    &= -A_t\sup_Y\left\{\mathbb{E}\left[\langle \bar{X}_t,Y\rangle-V^*(Y)\right]\right\} = -A_t\mathbb{E}\left[V(\bar{X}_t)\right].
\end{align*}
Moreover, by the Fenchel--Young inequality \eqref{eq:Fenchel IE},
\begin{align*}
A_\tau(\mathbb{E}\left[V^*(Y_\tau)-\langle \bar{X}_{\tau},Y_{\tau}\rangle\right])
    \geq -A_\tau\mathbb{E}\left[V(\bar{X}_{\tau})\right].
\end{align*}
Using again \eqref{eq:biconjugate_id}, it holds that
\begin{align*}
    \langle X,Y_s\rangle - V^*(Y_s)
    \leq \sup_Y \left\{\langle X,Y\rangle - V^*(Y)\right\}
    = V(X).
\end{align*}
Taking expectations, taking the supremum over $X \sim \pi$, and integrating in time gives
\begin{align*}
    \int_\tau^t\sup_{X \sim \pi}\mathbb{E}\left[\langle X,Y_s\rangle\right]\alpha_s \dd s
    - \int_\tau^t\mathbb{E}\left[V^*(Y_s)\right]\alpha_s \dd s
    \leq \int_\tau^t\mathbb{E}_{X \sim \pi}\left[V(X)\right]\alpha_s \dd s
    = \left(A_t-A_{\tau}\right)\mathbb{E}_{X \sim \pi}\left[V(X)\right].
\end{align*}
In total, we get the estimate 
\begin{align*}
    \Reg_t^{X} + \Reg_t^{Y}
    & \geq  A_t \mathbb{E}[V(\bar{X}_t)] - A_{\tau} \mathbb{E}[V(\bar{X}_{\tau})] + A_{\tau}\mathbb{E}_{X \sim \pi}[V(X)]-A_t \mathbb{E}_{X \sim \pi}[V(X)] \\
    & \;\;\; + A_t(\Ent\left(\rho_t\right) - \Ent\left(\pi\right)) -  A_\tau (\Ent(\rho_\tau) - \Ent(\pi)) \\ 
    & = A_t\left(\mathcal{F}(\rho_t) - \mathcal{F}(\pi) \right)  - A_\tau \left(\mathcal{F}(\rho_{\tau}) - \mathcal{F}(\pi)\right)  = A_t \KL(\rho_t \| \pi) - A_\tau \left(\mathcal{F}(\rho_{\tau}) - \mathcal{F}(\pi)\right).
\end{align*}
Rearranging yields the statement. 
\end{proof}

\textbf{From Fenchel games to underdamped Langevin dynamics.} Sections \ref{sec:FTL convex} and \ref{convex:noisy} will establish bounds on $\Reg_t^X$ and $\Reg_t^Y$ (associated with the X and Y-players' strategies) that will imply Theorem \ref{theorem:UL_convex} on the basis of Proposition \ref{theorem:sub-optimality}. Here, we already give an overview that details the connections of these strategies to underdamped Langevin dynamics \eqref{eq:UL}. The Y-player will use the strategy
\begin{equation}
\label{eq:Yt XY section}
Y_t = \nabla V(\bar{X}_t),\quad\text{where }\bar{X}_t = \frac{A_{\tau}}{A_t}\bar{X}_{\tau} + \frac{1}{A_t}\int_{\tau}^t X_s\alpha_s\dd s,
\end{equation}
analysed in Section \ref{sec:FTL convex}, while
the X-player will update her moves according to 
\begin{equation}
\label{convex_game}
\dd X_t  = - \alpha_t Y_t \dd t + \sqrt{3\alpha_t}\dd  W_t,
\end{equation}
analysed in Section \ref{convex:noisy}. These coupled dynamics can generally be related to \eqref{eq:UL} via the change of variables seen in \eqref{eq:QP-X}:
\begin{align}
\label{eq:QP-X_c}
Q_t = \bar{X}_t, \qquad P_t = \frac{\alpha_t}{A_t} (X_t - \bar{X}_t), \quad t \geq \tau. 
\end{align}
In particular note that differentiating $\bar{X}_t$ indeed yields
\begin{equation}
\label{eq:average evolution}
\dd \bar{X}_t = \frac{\alpha_t}{A_t}(X_t - \bar{X}_t) \dd t = P_t\dd t.
\end{equation}
Now, choosing the specific \emph{Nesterov schedule}
\begin{align}\label{Nesterov-Schedule}
    A_t = \frac{t^2}{4},\;\;\; \alpha_t = \frac{t}{2},
\end{align}
and differentiating $P_t$,  establishes the correspondence between \eqref{eq:UL} for $\gamma_t = \tfrac{3}{t}$ and \eqref{eq:Yt XY section}-\eqref{convex_game}. A reasoning of the explicit weighting choice \eqref{Nesterov-Schedule} is given in Remark \ref{rem:schedule}.

\subsection{Regret of the Y-player}
\label{sec:FTL convex}

By the properties of the Fenchel conjugate (see Section \ref{sec:preliminaries}), the instantaneous loss for the Y-player can be optimised in closed form, see Remark \ref{remark: details_y} below, leading to zero regret:

\begin{proposition}[Regret bound for the Y-player]
\label{y_player}
Fix $t \ge \tau$. The strategy
\begin{align}
\label{FTL}
    Y_s = \nabla V(\bar{X}_{s}),\qquad s \in [\tau,t],
\end{align}
leads to zero regret of the Y-player:
\begin{align}\label{FTL:no_regret}
\Reg_t^{Y} = 0.
\end{align} 
\end{proposition}

\begin{remark}\label{remark: details_y}
The strategy \eqref{FTL} arises naturally as an optimal response (or follow-the-leader, e.g. cp. \cite[Algorithm 5]{fenchelgames}) strategy associated with the loss in \eqref{eq:Y player loss} up to time $s$: 
\begin{equation}
Y_s = 
\argmin_Y \left\{A_\tau\left(\mathbb{E}\left[V^*(Y)-\langle \bar{X}_\tau,Y\rangle\right]\right) + \int_\tau^s\mathbb{E}\left[V^*(Y)-\langle X_u,Y\rangle \right]\alpha_s \dd u \right\}.
\end{equation}
Notice the difference to the competitor $Y^{\text{comp}}_t$ in \eqref{eq:Y comp def}, which is chosen once for the terminal horizon $t$ and held constant throughout the rounds.
\end{remark}

\begin{proof}
The proof follows \cite[Theorem 4.1]{kwon2014continuous} closely, as our stochastic setting poses no essential difficulties: all arguments can be performed pathwise.
We first compute the accumulated loss of the competitor in \eqref{conv:y_regret}:
By the Fenchel--Young inequality \eqref{eq:Fenchel IE}, for every random variable $Y$, we have that  
$V^*(Y)-\langle \bar{X}_t,Y\rangle
    \geq -V(\bar{X}_t)$,
with equality for $Y=\nabla V(\bar{X}_t)$, see \eqref{FY_Eq}. Hence
\begin{align}\label{eq:Y_comp_cost}
    \inf_Y \left\{
    A_t\mathbb{E}\left[V^*(Y)-\langle \bar{X}_t,Y\rangle\right]
    \right\}
    =
    -A_t\mathbb{E}\left[V(\bar{X}_t)\right].
\end{align}
It therefore remains to show that the accumulated loss \eqref{eq:Y player loss} of the strategy
\eqref{FTL} equals the right-hand side of \eqref{eq:Y_comp_cost}.
Again, by \eqref{FY_Eq}, we have that $
    V^*(Y_s)-\langle X_s,Y_s\rangle
    =
    -V(\bar{X}_s)
    +\langle \bar{X}_s-X_s,\nabla V(\bar{X}_s)\rangle$.
In combination with \eqref{eq:average evolution}, we see that 
\begin{align*}
    \frac{\dd}{\dd s}\left(A_sV(\bar{X}_s)\right)
    &=
    \alpha_s V(\bar{X}_s)
    +A_s\left\langle
        \nabla V(\bar{X}_s),
        \frac{\dd}{\dd s}\bar{X}_s
    \right\rangle 
    \\
    & =
    \alpha_s\left(
        V(\bar{X}_s)
        +\langle \nabla V(\bar{X}_s),X_s-\bar{X}_s\rangle
    \right) =
    -\alpha_s\left(
        V^*(Y_s)-\langle X_s,Y_s\rangle
    \right).
\end{align*}
Integration from $\tau$ to $t$
and taking expectations yields
\begin{equation}
\label{eq:running cost Y}
    \int_\tau^t
    \mathbb{E}\left[
        V^*(Y_s)-\langle X_s,Y_s\rangle
    \right]
    \alpha_s \dd s
    =
    A_\tau\mathbb{E}\left[V(\bar{X}_\tau)\right]
    -
    A_t\mathbb{E}\left[V(\bar{X}_t)\right].
\end{equation}
By \eqref{FY_Eq} and
$Y_\tau=\nabla V(\bar{X}_\tau)$, the $A_\tau$-contribution is $
    A_\tau\mathbb{E}\left[
        V^*(Y_\tau)-\langle\bar{X}_\tau,Y_\tau\rangle
    \right]
    =
    -A_\tau\mathbb{E}\left[V(\bar{X}_\tau)\right]$.
Combining this with \eqref{eq:running cost Y} gives
\begin{align*}
    &A_\tau\mathbb{E}\left[
        V^*(Y_\tau)-\langle\bar{X}_\tau,Y_\tau\rangle
    \right]
    +
    \int_\tau^t
    \mathbb{E}\left[
        V^*(Y_s)-\langle X_s,Y_s\rangle
    \right]\alpha_s \dd s =
    -A_t\mathbb{E}\left[V(\bar{X}_t)\right],
\end{align*}
implying the claim.
\end{proof}

\subsection{Regret of the X-player}\label{convex:noisy}
 
We now estimate the regret of the X-player's strategy, which (in accordance with \eqref{convex_game}) is given by
\begin{align}\label{noisy}
\dd X_t & = -\alpha_tY_t \dd t + \sqrt{3\alpha_t} \dd W_t,\qquad t \geq \tau >0, 
\end{align}
with an $\mathbb{R}^d$-valued initial random variable $X_{\tau}$ and where $Y_t$ is assumed to be $\sigma(\bar{X}_t)$-measurable (as is the case for \eqref{eq:Yt XY section}). We will show the following:

\begin{proposition}[Regret bound for the X-player]\label{prop:xplayer_convex} 
Fix $t \geq \tau > 0$. Denote by $r_s:\mathbb{R}^d\rightarrow \mathbb{R}^d$ the optimal transport map from $\operatorname{Law}(\bar{X}_s)$ to $\pi$, $s \in [\tau,t]$. Then, under Assumption \ref{ass:FP-regularity} (and imposing the correspondence \eqref{eq:QP-X_c}), the solution to \eqref{noisy} satisfies  
\begin{equation}
\label{eq:Reg X bound}
     \Reg_t^{X}\leq A_\tau\mathbb{E}_{\bar{X}_\tau}[\operatorname{KL}(\operatorname{Law} (X_{\tau}\|\bar{X}_\tau)\,|\,\mathcal{N}(r_\tau(\bar{X}_\tau), A_\tau I_d)].
\end{equation}
\end{proposition}
\begin{remark}\label{remark:choice_Psit}
The proof of Proposition \ref{prop:xplayer_convex} proceeds by showing that
\begin{equation}\label{eq:def_Psi_t}
t \mapsto \underbrace{A_t\mathbb{E}_{\bar{X}_t}[\operatorname{KL}(\operatorname{Law} (X_{t}|\bar{X}_t)\,\|\,\mathcal{N}(r_t(\bar{X}_t), A_t I_d)]}_{=: \Psi_t} + \Reg_t^X
\end{equation}
is non-increasing, where we have defined the energy function $\Psi_t$. For intuition, consider Proposition \ref{prop:x-player_N} from the Appendix, which states the regret bound for the X-player's strategy in the (deterministic) convex optimisation game. There, the strategy is given by 
\begin{align*}
    \dot{x}_t = - \alpha_t y_t,
\end{align*}
thus by a deterministic analogue of the stochastic strategy in \eqref{noisy}. The corresponding energy function is given by 
\begin{align}\label{eq:Lyapunov}
    \psi_t = \frac{1}{2}\|x_t - x^*\|^2,
\end{align}
tracking the distance of $x_t$ to the minimiser $x^* \in \text{argmin}_{x \in \mathbb{R}^d} f(x)$. The energy function $\Psi_t$ \eqref{eq:def_Psi_t} extends this ansatz in terms of probability laws: it compares the law of $X_t$ given $\bar{X}_t$ to the time-dependent reference measure $\mathcal{N}(r_t(\bar{X}_t), A_t I_d)$. Here, $r_t:\mathbb{R}^d \rightarrow \mathbb{R}^d$ denotes the Wasserstein-2 optimal transport map (also called Brenier map in the following) from $\operatorname{Law}(\bar{X}_t)$ to $\pi$. We remark that writing out $\Psi_t$ as 
\begin{align}\label{eq:motivated_ef}
A_t\mathbb{E}_{\bar{X}_t}[\operatorname{KL}(\operatorname{Law} (X_t|\bar{X}_t)\,\|\,\mathcal{N}(r_t(\bar{X}_t), A_t I_d)] = A_t\mathbb{E}_{\bar{X}_t}[\operatorname{Ent}(X_t|\bar{X}_t)] + \frac{1}{2}\mathbb{E}[\|X_t - r_t(\bar{X}_t)\|^2] + \frac{dA_t}{2}\ln\big({2\pi A_t}\big),
\end{align}
makes the analogy with \eqref{eq:Lyapunov} and with other work as discussed in Remark \ref{rem:previous work technical} (b) more explicit.
\end{remark}

Before proving Proposition \ref{prop:xplayer_convex}, we state two supporting lemmas. The following first lemma provides an estimate for the time evolution of the second term on the right-hand side of \eqref{eq:motivated_ef}.
\begin{lemma}\label{lemma:L2-distance_convex}
Denote by $r_t:\mathbb{R}^d\rightarrow \mathbb{R}^d$ the optimal transport map from $\operatorname{Law }(\bar{X}_t)$ to $\pi$. Along the solution to \eqref{noisy}, we have
\begin{equation}
\label{eq:integrated-main}
\begin{aligned}
&\frac12\E\big[\|X_t-r_t(\bar X_t)\|^2\big]
 -\frac12\E\big[\|X_s-r_s(\bar X_s)\|^2\big]
\\
&\quad\leq
\int_s^t\Bigg[
 -\E \langle X_u-r_u(\bar X_u),Y_u \rangle
 -\frac{1}{A_u}
   \E\bigl[C_u^X(\bar X_u):\nabla r_u(\bar X_u)\bigr]
 +\frac32 d
\Bigg]\alpha_u\dd u,
\end{aligned}
\end{equation}
where $C^X_t(\bar{x})$ denotes the conditional covariance matrix of $X_t$ given $\bar{X}_t = \bar{x}$,
\begin{align*}
    C^X_t(\bar{x}) = \mathbb{E}\big[(X_t-\mathbb{E}[X_t|\bar{X}_t]) \otimes (X_t-\mathbb{E}[X_t|\bar{X}_t]) | \bar{X}_t = \bar{x}\big].
\end{align*}
\end{lemma}

\begin{proof}
See Appendix \ref{sec:proof_L2-distance_convex}.
\end{proof}

\begin{remark}[Regularity]
\label{rem:Brenier regularity}
The estimate in \eqref{eq:integrated-main} is essentially a differential inequality controlling the evolution of the distance $\tfrac12\E\|X_t-r_t(\bar X_t)\|^2$. However, since $u \mapsto r_u$ may not be differentiable in time \cite{Letrouit2026Unstable,Gigli2011Holder}, Lemma \ref{lemma:L2-distance_convex} is stated in integrated form.
\end{remark}

Before the statement of the next lemma, recall that the Fisher information of a probability measure $\nu \in \mathcal{P}(\mathbb{R}^d)$ is given by
\begin{align*}
    J(\nu) = \int_{\mathbb{R}^d} \|\nabla_x \ln{\nu(x)}\|^2\,\nu(\dd x),
\end{align*}
and that its relative Fisher information with respect to a probability measure $\mu \in \mathcal{P}(\mathbb{R}^d)$ is given by 
\begin{align*}
\mathcal{I}\bigl(\nu\|\,\mu\bigr) = \int_{\mathbb{R}^d} \Big\|\nabla_x \ln{\frac{\nu(x)}{\mu(x)}}\Big\|^2\,\nu(\dd x)
\end{align*}
whenever the right-hand sides are well-defined and finite (as usual, we denote the Lebesgue density of $\nu$, assumed to exist, by the same symbol).

\begin{lemma}[Entropy-Transport Cram{\'e}r-Rao inequality]
\label{lemma:ent-cramer-rao}
Let $\rho,\pi\in\mathcal P_2(\mathbb R^d)$ have strictly positive
Lebesgue densities and finite entropies, and let
$r:\mathbb R^d\to\mathbb R^d$ be the Brenier map transporting
$\rho$ to $\pi$. Let $\{\nu_{\bar{x}}\}_{\bar{x}\in\mathbb R^d}$ be a
$\rho$-measurable family of probability measures on $\mathbb R^d$.
Assume that, for $\rho$-almost every $\bar{x}$,
$\nu_{\bar{x}}$ has a strictly positive density, finite second moment,
and finite Fisher information. Denote its covariance matrix and Fisher
information by $C(\nu_{\bar{x}})$ and $J(\nu_{\bar{x}})$, respectively. Then
\begin{align}
\label{eq:transport-fisher-compensation}
    \Ent(\rho)-\Ent(\pi)
    +
    2d 
    \le \int_{\mathbb R^d}
    \left(
J(\nu_{\bar{x}})
+ C(\nu_{\bar{x}}):\nabla r(\bar{x})
\right) \rho(\dd \bar{x})
    .
\end{align}
Here $\nabla r$ denotes the $\rho$-almost-everywhere Alexandrov derivative \cite[Appendix]{mccann2013five}
of the Brenier map. 
\end{lemma}

\begin{proof}
See Appendix \ref{app:lemmas}. Conceptually, the inequality \eqref{eq:transport-fisher-compensation} combines the Cram{\'e}r-Rao inequality \cite{KAGAN19997}, which compares Fisher information and covariance, and the change-in-entropy formula $\Ent(\rho) - \Ent(\pi) = \int_{\mathbb{R}^d}\ln \det \nabla r \dd \rho$.    
Ideas along similar lines have been used in  \cite[Lemmas 5.1 and 5.2]{lu2026sharp}.
\end{proof}

\begin{remark}
 Note that the first term on the right-hand side of \eqref{eq:integrated-main} in Lemma \ref{lemma:L2-distance_convex} allows us to track the non-entropic part in the regret of the X-player \eqref{conv:x_regret}. The remaining two terms, resulting from the noise in \eqref{noisy}, are -- in combination with the remaining terms in \eqref{eq:motivated_ef} -- responsible for estimating the difference in entropy between $\operatorname{Law}(\bar{X}_t)$ and $\pi$ in \eqref{conv:x_regret}. Establishing this is the core task in the proof of Proposition \ref{prop:xplayer_convex}; it is achieved by applying a change of variables, the log-Sobolev inequality of the standard Gaussian, and Lemma \ref{lemma:ent-cramer-rao}. 
\end{remark}

We now state the proof of Proposition~\ref{prop:xplayer_convex}.  

\begin{proof}[Proof of Proposition~\ref{prop:xplayer_convex}]
Consider the energy function
\begin{align*}
    \Psi_t
    &=
    A_t \mathbb{E}_{\bar{X}_t}
    \left[
        \KL
        \left(
            \Law(X_t|\bar{X}_t)
            \,\middle|\,
            \mathcal{N}(r_t(\bar{X}_t), A_t I_d)
        \right)
    \right]
    \\
    &=
    A_t \mathbb{E}_{\bar{X}_t}
    [\Ent(X_t|\bar{X}_t)]
    + \frac{1}{2}\mathbb{E}
    [\|X_t-r_t(\bar{X}_t)\|^2]
    + A_t\frac{d}{2}\ln(2\pi A_t).
\end{align*}
For notational convenience, write
\begin{align*}
    H_t
    :=
    \mathbb{E}_{\bar{X}_t}
    [\Ent(X_t|\bar{X}_t)], \qquad 
    J(X_t|\bar{X}_t)
    :=
    \mathbb{E}_{\bar{X}_t}
    \big[
        J(X_t|\bar{X}_t=\bar{x})
    \big],
\end{align*}
for the expected conditional entropy and conditional Fisher information.
Using the Fokker--Planck equation for the joint distribution of
$(\bar{X}_u,X_u)$ and Assumption \ref{ass:FP-regularity}, we obtain 
\begin{align}
\label{eq:integrated-conditional-entropy}
    A_t H_t-A_s H_s
    &=
    \int_s^t
    \left[
         H_u
        + \Ent(\bar{X}_u)
        + d
        -A_u\frac{3}{2}
        J(X_u|\bar{X}_u)
    \right]
    \alpha_u\dd u
    -A_t\Ent(\bar{X}_t)
    +A_s\Ent(\bar{X}_s),
\end{align}
for $\tau\leq s\leq t$. Note also that
\begin{align}
\label{eq:integrated-gaussian-normalization}
    A_t\frac{d}{2}\ln(2\pi A_t)
    -
    A_s\frac{d}{2}\ln(2\pi A_s)
    =
    \int_s^t
    \left[
        \frac{d}{2}\ln(2\pi A_u)
        +\frac{d}{2}
    \right]
    \alpha_u\dd u.
\end{align}
Combining \eqref{eq:integrated-conditional-entropy},
\eqref{eq:integrated-gaussian-normalization}, and 
\eqref{eq:integrated-main} from Lemma \ref{lemma:L2-distance_convex}, we obtain
\begin{subequations}
\label{eq:PsitPsis}
\begin{align*}
    \Psi_t-\Psi_s
    &\leq
    \int_s^t
    \Bigg[
        H_u
        +\Ent(\bar{X}_u)
        +3d
        -A_u\frac{3}{2}
        J(X_u|\bar{X}_u)
        -
        \mathbb{E}
        \big[
            \langle
                X_u-r_u(\bar{X}_u),
                y_u(\bar{X}_u)
            \rangle
        \big]
        \\
        &\qquad\qquad
        -\frac{1}{A_u}
        \mathbb{E}
        \big[
            C_u^X(\bar{X}_u):
            \nabla_{\bar{x}}r_u(\bar{X}_u)
        \big]
        +\frac{d}{2}\ln(2\pi A_u)
    \Bigg]
    \alpha_u\dd u
    -A_t\Ent(\bar{X}_t)
    +A_s\Ent(\bar{X}_s).
    \tag{\ref{eq:PsitPsis}}
\end{align*}
\end{subequations}
Note that we have
\begin{equation}
\label{eq:coupling}
    \mathbb{E}
    \big[
        \langle
            r_u(\bar{X}_u),
            Y_u
        \rangle
    \big]
    \leq
    \sup_{R\sim\pi}
    \mathbb{E}[\langle R,Y_u\rangle],
\end{equation}
since the pair $(r_u(\bar{X}_u),y_u(\bar{X}_u))$ is one admissible coupling. Furthermore, by the definition in \eqref{conv:x_regret},
\begin{subequations}
\label{eq:RegtRegs}
\begin{align*}
    \Reg_t^X-\Reg_s^X
    &=
    \int_s^t
    \mathbb{E}[\langle X_u,Y_u\rangle]\alpha_u \dd u
    -
    \int_s^t
    \sup_{R\sim\pi}
    \mathbb{E}[\langle R,Y_u\rangle]\alpha_u \dd u
    \\
    & \qquad +A_t\Ent(\bar{X}_t)
    -A_s\Ent(\bar{X}_s)
    -(A_t-A_s)\Ent(\pi).
    \tag{\ref{eq:RegtRegs}}
\end{align*}
\end{subequations}
Adding \eqref{eq:PsitPsis} and \eqref{eq:RegtRegs}, and using the coupling inequality \eqref{eq:coupling}, we obtain
\begin{align}
\label{eq:energy-regret-before-rescaling}
    (\Psi_t+\Reg_t^X)
    -
    (\Psi_s+\Reg_s^X)
    & \leq
    \int_s^t
    \Bigg[
        H_u
        +3d
        -\frac{3}{2}A_uJ(X_u|\bar{X}_u)
        -\frac{1}{A_u}
        \mathbb{E}
        \big[
            C_u^X(\bar{X}_u):
            \nabla_{\bar{x}}r_u(\bar{X}_u)
        \big]
        \\
        & \qquad +\frac{d}{2}\ln(2\pi A_u)
        +\Ent(\bar{X}_u)
        -\Ent(\pi)
    \Bigg]
    \alpha_u\dd u.
\end{align}
At this point, the integrand on the right-hand side simplifies by
introducing the auxiliary variable
\begin{align*}
    P_u
    =
    \frac{1}{\sqrt{A_u}}
    (X_u-\bar{X}_u).
\end{align*}
Specifically, we make use of the following relations
(see~Lemma~\ref{lemma:scaling_shifting} in Appendix~\ref{app:lemmas}):
\begin{subequations}
\begin{align}
    \Ent(X_u|\bar{X}_u=\bar{x})
    &=
    \Ent(P_u|\bar{X}_u=\bar{x})
    -\ln A_u^{\frac{d}{2}},
    \label{simplification_0}
    \\
    A_uJ(X_u|\bar{X}_u=\bar{x})
    &=
    J(P_u|\bar{X}_u=\bar{x}),
    \\
    \frac{1}{A_u}C_u^X(\bar{x})
    &=
    C_u^P(\bar{x}).
\end{align}
\end{subequations}
Furthermore,
\begin{align}
\Ent(P_u|\bar{X}_u=\bar{x})
    +\frac{d}{2}\ln(2\pi)
    \nonumber
    =
    \KL\left(
        \Law(P_u|\bar{X}_u=\bar{x})
        \,\middle|\,
        \mathcal{N}(0,I_d)
    \right)
    -
    \frac{1}{2}
    \mathbb{E}
    \left[
        \|P_u\|^2
        \,\middle|\,
        \bar{X}_u=\bar{x}
    \right],
\end{align}
and
\begin{align}
\label{relative fisher-information}
    -J(P_u|\bar{X}_u=\bar{x})+2d
    &=
    -\mathcal{I}
    \left(
        \Law(P_u|\bar{X}_u=\bar{x})
        \,\middle|\,
        \mathcal{N}(0,I_d)
    \right) 
    +
    \mathbb{E}
    \left[
        \|P_u\|^2
        \,\middle|\,
        \bar{X}_u=\bar{x}
    \right].
\end{align}
Using these identities in
\eqref{eq:energy-regret-before-rescaling}, we obtain
\begin{align}
\label{eq:intermediate-bound}
    &(\Psi_t+\Reg_t^X)-(\Psi_s+\Reg_s^X)
    \leq
    \int_s^t
    \Bigg\{
        \Ent(\bar{X}_u)-\Ent(\pi)
        +
        \mathbb{E}_{\bar{X}_u}
        \Bigg[
            \KL
            \left(
                \Law(P_u|\bar{X}_u)
                \,\middle|\,
                \mathcal{N}(0,I_d)
            \right)
            \nonumber\\
    &
            -\frac{3}{2}
            \mathcal{I}
            \left(
                \Law(P_u|\bar{X}_u)
                \,\middle|\,
                \mathcal{N}(0,I_d)
            \right)
            +
            \mathbb{E}
            [\|P_u\|^2|\bar{X}_u]
            -
            C_u^P(\bar{X}_u):
            \nabla_{\bar{x}}r_u(\bar{X}_u)
        \Bigg]
    \Bigg\}
    \alpha_u\dd u.
\end{align}

Applying the log-Sobolev inequality for the standard Gaussian,
\cite[Example 21.3]{villani2009optimal},
\begin{align}
\label{Log_sobolev}
    \KL
    \left(
        \Law(P_u|\bar{X}_u=\bar{x})
        \,\middle|\,
        \mathcal{N}(0,I_d)
    \right)
    \leq
    \frac{1}{2}
    \mathcal{I}
    \left(
        \Law(P_u|\bar{X}_u=\bar{x})
        \,\middle|\,
        \mathcal{N}(0,I_d)
    \right),
\end{align}
we obtain from \eqref{eq:intermediate-bound}
\begin{align*}
    (\Psi_t+\Reg_t^X)-(\Psi_s+\Reg_s^X)\leq &
    \int_s^t
    \alpha_u
    \Bigg\{
        \Ent(\bar{X}_u)-\Ent(\pi)
        +
        \mathbb{E}_{\bar{X}_u}
        \Bigg[
            -
            \mathcal{I}
            \left(
                \Law(P_u|\bar{X}_u)
                \,\middle|\,
                \mathcal{N}(0,I_d)
            \right)
            \\
            & +
            \mathbb{E}
            [\|P_u\|^2|\bar{X}_u]
            -
            C_u^P(\bar{X}_u):
            \nabla_{\bar{x}}r_u(\bar{X}_u)
        \Bigg]
    \Bigg\}
    \dd u.
\end{align*}
Using \eqref{relative fisher-information} again, this can be rewritten as
\begin{subequations}\label{eq:before-ent-cramer-rao}
    \begin{align}
    (\Psi_t+\Reg_t^X)-(\Psi_s+\Reg_s^X)
    & \leq
    \int_s^t
    \Bigg\{
        \Ent(\bar{X}_u)-\Ent(\pi)
        \\ 
        & \qquad +
        \mathbb{E}_{\bar{X}_u}
        \Big[
            2d
            -
            J(P_u|\bar{X}_u)
            -
            C_u^P(\bar{X}_u):
            \nabla_{\bar{x}}r_u(\bar{X}_u)
        \Big]
    \Bigg\}
    \alpha_u\dd u.
\end{align}
\end{subequations}
For each $u$, we now apply
Lemma~\ref{lemma:ent-cramer-rao} with $\rho=\operatorname{Law} \bar{X}_u$, $r=r_u$,
 and   $\nu_{\bar{x}}
    =
    \Law(P_u|\bar{X}_u=\bar{x})$
to conclude
\begin{align*}
    \Psi_t+\Reg_t^X
    \leq
    \Psi_s+\Reg_s^X,
    \qquad
    \tau\leq s\leq t.
\end{align*}
Taking $s=\tau$ and recalling that the centered regret satisfies
$\Reg_\tau^X=0$, we conclude that
\begin{align*}
    \Reg_t^X
    \leq
    \Psi_\tau-\Psi_t
    \leq
    \Psi_\tau.
\end{align*}
\end{proof}

\subsection{Overall convergence result}\label{sec:total_conv_conv}

\begin{theorem}\label{convex:convergence}
The solution to the coupled dynamics \eqref{eq:Yt XY section}-\eqref{convex_game} satisfies, for $t \geq \tau > 0$,
\begin{align*}
    \mathcal{F}(\operatorname{Law}(\bar{X}_t)) - \mathcal{F}(\pi) \leq \frac{A_{\tau}}{A_t}\operatorname{KL}(\operatorname{Law}(\bar{X}_{\tau}, X_{\tau})\, \|\, \Pi_{\tau}),
\end{align*}
where 
\begin{align}
    \Pi_{\tau}(\dd \bar{x}, \dd x) = \pi(\dd \bar{x})\mathcal{N}(r_{\tau}(\bar{x}), A_{\tau}I_d)(\dd x).
\end{align}
\end{theorem}

\begin{proof}
    Apply Propositions \ref{y_player} and \ref{prop:xplayer_convex} to Proposition \ref{theorem:sub-optimality} to conclude that
    \begin{align*}
        \operatorname{KL}(\operatorname{Law}(\bar{X}_{t}) \| \pi) = \mathcal{F}(\operatorname{Law}(\bar{X}_t)) - \mathcal{F}(\pi) \leq \frac{\Psi_{\tau}}{A_t}+\frac{A_{\tau}}{A_t}\operatorname{KL}(\operatorname{Law}(\bar{X}_{\tau}) \| \pi).
    \end{align*}
    Finally, note that, by the chain rule for the KL divergence,
    \begin{align*}
        \Psi_{\tau} + A_{\tau}\operatorname{KL}(\operatorname{Law}(\bar{X}_{\tau})\| \pi) = A_{\tau}\operatorname{KL}(\operatorname{Law}(\bar{X}_{\tau}, X_{\tau}) \,\|\,\Pi_{\tau}).
    \end{align*}.
 \end{proof}
We are now able to state the Proof of Theorem \ref{theorem:UL_convex}. 
\begin{proof}[Proof of Theorem~\ref{theorem:UL_convex}]
We start by rewriting the game dynamics \eqref{eq:Yt XY section}-\eqref{convex_game} as the physical underdamped Langevin dynamics \eqref{eq:UL} following the change of variables \eqref{eq:QP-X_c} and It\^{o}'s formula \cite[Lemma 3.2]{pavliotis2014stochastic}:
\begin{align*}
    \dd P_t & = \frac{\alpha_t}{A_t}\dd X_t - \frac{\alpha_t}{A_t}\dd \bar{X}_t + \frac{\dd}{\dd t}\left(\frac{\alpha_t}{A_t}\right)(X_t - \bar{X}_t)\dd t \\
    & = -\frac{\alpha^2_t}{A_t}\nabla V(Q_t)\dd t +\left(\frac{A_t}{\alpha_t}\frac{\dd}{\dd t}\left(\frac{\alpha_t}{A_t}\right) - \frac{\alpha_t}{A_t}\right)P_t\dd t + \frac{\alpha_t}{A_t}\sqrt{3\alpha_t}\dd W_t.
\end{align*}
Thus, the game system and the physical system with $\gamma_t = \frac{3}{t}$ coincide if the weighting is chosen as
\begin{align*}
    A_t = \frac{t^2}{4},\;\;\;\alpha_t = \frac{t}{2}.
\end{align*}
Moreover, since $P_{\tau} = \frac{\alpha_{\tau}}{A_{\tau}} (X_\tau - \bar{X}_\tau) = \frac{2}{\tau}(X_\tau - \bar{X}_\tau)$ is an invertible linear transformation, the conditional KL divergences relate as:
\begin{align*}
\operatorname{KL}\left(\operatorname{Law}(X_{\tau}|\bar{X}_{\tau})\, \|\,\mathcal{N}(r_{\tau}(\bar{X}_\tau), A_{\tau}I_d) \right) = \operatorname{KL}\left(\operatorname{Law}(P_{\tau}|\bar{X}_{\tau})\,\|\,\mathcal{N}(2\tau^{-1}(r_{\tau}(\bar{X}_\tau)-\bar{X}_\tau), I_d) \right).
\end{align*}
Thus, since $Q_t = \bar{X}_t$, Theorem \ref{convex:convergence} readily gives the result.
\end{proof}

\begin{remark}\label{remark:scalings}
While under suitable regularity assumptions, Theorem \ref{convex:convergence} holds for all admissable weighting schedules $(A_t,\alpha_t)$ satisfying $\dot A_t=\alpha_t$, the Nesterov damping \eqref{Nesterov-Schedule} yields a clearer interpretation of the auxiliary variable $P_t$ introduced in the proof of Proposition \ref{prop:xplayer_convex}: Identifying $\bar X_t=Q_t$ and choosing $A_t=\alpha_t^2$, we recover the standard underdamped Langevin velocity $P_t=\frac{\alpha_t}{A_t}(X_t-Q_t)$ as in \eqref{eq:QP-X_c}.
\end{remark}

\begin{remark}[Relationships to previous works]\label{rem:previous work technical}
\leavevmode\par
\begin{enumerate}[label=(\alph*)]
    \item The upper bound from Theorem \ref{convex:convergence} has a straightforward but insightful reformulation. Let us denote the density of $\operatorname{Law}(Q_{\tau}, P_{\tau})$ by $\widehat{\rho}_{\tau}$ and write 
\begin{subequations}\label{eq:upperbound_rewritten}
\begin{align}
    \operatorname{KL}&(\operatorname{Law}(Q_{\tau}, P_{\tau})\, \|\, \Pi_{\tau})  = \int \ln{\frac{\widehat{\rho}_{\tau}(q,p)}{\frac{1}{Z}e^{-V(q)-\frac{1}{2}\|p-\frac{2}{\tau}(r_{\tau}(q)-q)\|^2}}}\widehat{\rho}_{\tau}(q,p)\dd q\dd p \\
    & = \int \ln{\frac{\widehat{\rho}_{\tau}(q,p)}{\frac{1}{Z}e^{-V(q)-\frac{1}{2}\|p\|^2}}}\widehat{\rho}_{\tau}(q,p)\dd q\dd p + \frac{2}{\tau}\underbrace{\int \langle p , q-r_{\tau}(q)\rangle \widehat{\rho}_{\tau}(q,p)\dd q\dd p}_{=:\mathcal{C}_{OT}(\operatorname{Law}(Q_{\tau}, P_{\tau}))} + \frac{2}{\tau^2}\mathcal{W}_2^2(\operatorname{Law}(Q_{\tau}), \pi) \\
    & = \operatorname{KL}(\operatorname{Law}(Q_{\tau}, P_{\tau})\, \|\, \pi \otimes \mathcal{N}(0,I_d)) + \frac{2}{\tau}\mathcal{C}_{OT}(\operatorname{Law}(Q_{\tau}, P_{\tau})) + \frac{2}{\tau^2}\mathcal{W}_2^2(\operatorname{Law}(Q_{\tau}), \pi),
\end{align}    
\end{subequations}
which highlights the additional terms when comparing $\operatorname{Law}(P_t | Q_t)$ with the time-dependent reference measure rather than with the invariant measure. We remark that the ``freedom`` in choosing the reference measure for $\operatorname{Law}(P_t | Q_t)$ is a result of building the analysis around the KL divergence of the $Q_t$-marginal and not of the joint $(Q_t, P_t)$-system. Note that the above reformulation clearly underscores the importance of the so-called Wasserstein current corrector $\mathcal{C}_{OT}(\operatorname{Law}(Q_{\tau}, P_{\tau}))$, introduced by \cite{lu2026sharp}, in the analysis of the underdamped Langevin dynamics.

\item In \cite{wang2022accelerated} and \cite{pmlr-v97-taghvaei19a}, the Euclidean energy function \eqref{eq:Lyapunov}, also explicitly encountered as $\frac{1}{2}\|q_t + \frac{t}{2}p_t - x^*\|^2$ in the standard Lyapunov analysis of Nesterov's ODE \eqref{eq:nesterov_ODE} \cite{siegel2019accelerated,NesterovODE}, was extended to the level of probability laws as 
\begin{align}\label{eq:WL_Lyapunov}
    \Psi_t = \frac{1}{2}\mathbb{E}_{Q_t}[\|Q_t + \frac{t}{2}P_t(Q_t) - r_t(Q_t)\|^2]
\end{align}
to cover accelerated dynamics on $\mathcal{P}_2(\mathbb{R}^d)$ that consider velocities given as deterministic functions of the position. In contrast, our energy function \eqref{eq:def_Psi_t} accounts for uncertainty in velocity given the position. In particular, it can be motivated by noting that a natural reference particle for $X_t$ is of the form $Q_{\infty}+A_t^{1/2}P_{\infty} = Q_{\infty}+\frac{t}{2}P_{\infty}$, with $(Q_{\infty}, P_{\infty}) \sim \pi \otimes \mathcal{N}(0,I_d)$, following the change of variables \eqref{eq:QP-X_c}, Nesterov's schedule \eqref{Nesterov-Schedule} and the invariant distribution of the physical system \eqref{eq:qp target}. The conditional representation then replaces $Q_{\infty}$ by $r_t(\bar{X}_t)$ at time $t$, similarly as in \eqref{eq:WL_Lyapunov}.
\end{enumerate}
\end{remark}

Based on the reformulation of the upper bound in Remark \ref{rem:previous work technical} (a), we state the proof of Corollary \ref{cor:UL_convex}.

\begin{proof}[Proof of Corollary~\ref{cor:UL_convex}]
Consider the convergence statement in Theorem \ref{theorem:UL_convex}. If $\operatorname{Law}(P_{\tau}\, | \, Q_\tau) = \mathcal{N}(0,I_d)$,
the joint KL divergence can be written as $\operatorname{KL}(\operatorname{Law}(Q_{\tau}) \| \pi) + \frac{2}{\tau^2} \mathcal{W}_2^2(\operatorname{Law}(Q_{\tau}), \pi)$; in particular since the Wasserstein current corrector in \eqref{eq:upperbound_rewritten} vanishes. This shows the claim.
\end{proof}

\section{Continuous-time Fenchel games for strongly log-concave sampling}
\label{sec:main_sconvex}

We now consider target distributions $\pi = \frac{1}{Z} e^{-V(q)} \dd q$ with $\sigma$-strongly convex $V$, with the ultimate goal of proving Theorem \ref{theorem:UL_strongly-convex}. We proceed as described in Section \ref{sec:games} and follow Section \ref{sec:main_convex} closely. Thus, the proof strategy is -- apart from some modifications that will be discussed below -- analogous: Two fictitious players, the X- and Y-players, are set to formally solve the saddle-point problem \eqref{eq:saddle G}. By comparing their accumulated losses with those of competitors, through the so-called regrets $\operatorname{Reg}_t^X$ and $\operatorname{Reg}_t^Y$, we can quantify the performance of their strategy and obtain an optimality estimate of the form \eqref{eq:regret_to_suboptimality}. After a change of variables of the form \eqref{eq:QP-X}, the game strategies can be related to the underdamped Langevin dynamics \eqref{eq:UL}, proving Theorem \ref{theorem:UL_strongly-convex}. \\
Due to many conceptual parallels, we also refer to Section \ref{sec:polyak} in the Appendix, which embeds Polyak's ODE \eqref{eq:polyak_ODE} into a game-based framework and similarly derives its convergence rate. \\
\\
We start by adapting the payoff function of the game. Since $\tilde{V}(x) := V(x) - \frac{\sigma}{2}\|x\|^2$ is convex, we may replace $\mathcal{G}$ in \eqref{eq:G} by 
\begin{align}\label{eq:G_sc}
\mathcal{G}(X,Y) := \mathbb{E}[\langle X,Y\rangle - \tilde{V}^*(Y)] +\frac{\sigma}{2}\mathbb{E}[\|X\|^2] + \operatorname{Ent}(X),
\end{align}
thus assigning the convex part of the potential, $\tilde{V}$, to the Y-player and the strongly convex part of it, $\frac{\sigma}{2}\|x\|^2$, to the X-player. This decomposition is further motivated in Remark \ref{rem: V split sc} and follows \cite[Section 4.6]{fenchelgames}. Furthermore, in this section, there is no need to omit $t = 0$ as was required in Section \ref{sec:main_convex}, see Remark \ref{remark:technical_tau}. Thus, from now on, the game, the integrated weight function $A_t$ in \eqref{eq:A def complicated}, and the time average $\bar{X}_t$ in \eqref{eq:X bar complicated} start at $\tau = 0$, where $\alpha$ and $A_0 > 0$ are seen as prescribed. Furthermore, $\bar{X}_0$ is chosen with $ \operatorname{Law}(\bar{X}_0) = \operatorname{Law}(Q_0)$. \\
\\
The competitors remain the same as in Section \ref{sec:main_convex}: The Y-competitor is constant in time, but plays the in-hindsight optimal strategy \eqref{eq:Y comp convex} (with $\tau = 0$). The X-competitor again corresponds to the probability measure $\pi$, playing  $X_s^{\text{comp}} = \argmax_{X \sim \pi} \mathbb{E} [\langle X, Y_s \rangle]$ in round $s \in [0, t]$, see \eqref{eq:X comp def}. The following definition specifies the abstract regrets \eqref{eq:regrets_intro} relative to these competitors, while accommodating the updated payoff function $\mathcal{G}$ in \eqref{eq:G_sc}. 

\begin{definition}[Regrets for the strongly convex setting]
\label{def:regrets_sc}
For played actions $\{(X_s,Y_s)\}_{0 \leq s\leq t}$ and abbreviating $\rho_t = \operatorname{Law}(\bar{X}_t)$, we define the regrets at time $t\geq 0$ as 
\begin{subequations}
\label{eq:X regret sc}
\begin{align}   
\label{eq:X regret main sc}
\Reg_t^{X}  = & \int_{0}^t \mathbb{E}\left[\langle X_s,Y_s\rangle \right]\alpha_s\dd s + A_t(\operatorname{Ent}\left(\rho_t\right) + \frac{\sigma}{2}\mathbb{E}[\|\bar{X}_t\|^2]) \\ 
& -\int_0^t\sup_{X \sim \pi}\mathbb{E}\left[\langle X,Y_s\rangle\right]\alpha_s\dd s - A_t(\operatorname{Ent}\left(\pi\right) + \frac{\sigma}{2}\mathbb{E}_{X \sim \pi}[\|X\|^2])
\\
\label{eq:RegX offset sc}
& - A_0(\Ent(\rho_0) + \frac{\sigma}{2}\mathbb{E}[\|\bar{X}_0\|^2]) + A_0 (\Ent(\pi)+\frac{\sigma}{2}\mathbb{E}_{X \sim \pi}[\|X\|^2])
\end{align}
\label{sc conv:x_regret}
\end{subequations}
and
\begin{subequations}
\label{eq:Y regret sc}
\begin{align}
\label{eq:Y player loss sc}
\Reg_t^{Y} & = A_0\mathbb{E}\left[\tilde{V}^*(Y_0)-\langle \bar{X}_0,Y_0\rangle\right] + \int_0^t\mathbb{E}\left[\tilde{V}^*(Y_s)-\langle X_s,Y_s\rangle \right]\alpha_s\dd s \\ & 
\label{sconv:y_regret}
\;\;\;\ - \inf_Y \left \{ A_t \mathbb{E} \left[ \tilde{V}^*(Y) - \langle \bar{X}_t,Y \rangle   \right]\right\}. \end{align} \end{subequations}
\end{definition}

\begin{remark}
We remark that the offset in \eqref{eq:RegX offset sc} ensures $\operatorname{Reg}_0^X = 0$.
\end{remark}

\begin{remark}\label{rem: V split sc}
Since $\tilde{V}$ is convex, the regret of the Y-player is recorded as in the log-concave sampling game in Section \ref{sec:main_convex}, see \eqref{eq:Y regret}. We have already seen in Proposition \ref{y_player} that employing the follow-the-leader strategy, as described in Remark \ref{remark: details_y}, leads to zero regret of the Y-player. Making the Y-player pay according to the whole function $V$ could not improve this estimate. This motivates instead charging the X-player the quadratic cost $\frac{\sigma}{2}\mathbb{E}[\|\bar{X}_t\|^2]$, thus making this cost explicit to her and, in turn, allowing her to adapt her strategy accordingly.
\end{remark}

Again, these regrets are designed to bound the sub-optimality in terms of the free energy $\mathcal{F}$:

\begin{proposition}\label{theorem:sub-optimality_sc}
Denoting $\rho_t = \Law(\bar{X}_t)$, the regrets from Definition \ref{def:regrets_sc} satisfy 
\begin{align*}
   \KL(\rho_t \| \pi) =  \mathcal{F}(\rho_t) -  \mathcal{F}(\pi) \leq \frac{\Reg_t^{X} + \Reg_t^{Y}}{A_t} + \frac{A_{0}}{A_t}
\left(\mathcal{F}(\rho_0) -  \mathcal{F}(\pi) \right), \qquad  t \geq 0.
\end{align*}
\end{proposition}

\begin{proof}
The proof is analogous to that of Proposition \ref{theorem:sub-optimality}, using $V(x) = \tilde{V}(x) + \frac{\sigma}{2}\|x\|^2$.
\end{proof}

\textbf{From Fenchel games to underdamped Langevin dynamics.} 
In Sections \ref{sc:yplayer}-\ref{sc:convergence} below, we consider the general game given by the strategies
\begin{subequations}
\label{sc:game}    
\begin{align}
    Y_t & = \nabla \tilde{V}(\bar{X}_t),\quad\text{where }\bar{X}_t = \frac{A_{0}}{A_t}\bar{X}_{0} + \frac{1}{A_t}\int_{0}^t X_s\alpha_s\dd s, \label{eq: y player sc} \\
    \dd X_t & = - \beta_t(Y_t + \sigma X_t)\dd t + \sqrt{\beta_t(3+\sigma \beta_t^2)}\dd W_t, \label{eq: x player sc}
\end{align}
\end{subequations}
where $\beta_t := \frac{A_t}{\alpha_t}$. Note that the deterministic drift of the X-player's strategy can also be found in the strongly convex optimisation game in Appendix \ref{sec:polyak}, see \eqref{eq:game_Rd_P2}. Again, the correspondence to the underdamped Langevin system  \eqref{eq:UL} can be established by the change of variables of the form \eqref{eq:QP-X}:
\begin{align}
\label{eq:QP-X_sc}
Q_t = \bar{X}_t, \qquad P_t = \frac{\alpha_t}{A_t} (X_t - \bar{X}_t), \quad t \geq 0. 
\end{align}
Specifically, consider the weighting given by the \textit{Polyak schedule} 
\begin{align}\label{Polyak-Schedule}
    A_t = e^{\sqrt{\sigma}t},\;\;\; \alpha_t = \sqrt{\sigma}e^{\sqrt{\sigma}t},
\end{align}
this means $\beta_t \equiv \frac{1}{\sqrt{\sigma}}$. The game dynamics \eqref{eq: y player sc}-\eqref{eq: x player sc} then correspond to \eqref{eq:UL} if $\gamma_t \equiv  2\sqrt{\sigma}$; this in turn justifies the exact noise coefficient in \eqref{eq: x player sc}.

\subsection{Regret of the Y-player}\label{sc:yplayer}

Since $\tilde{V}$ is convex, the Y-player follows in \eqref{eq: y player sc} the same follow-the-leader strategy as in the log-concave sampling game, and the results from Section \ref{sec:FTL convex} readily apply.

\begin{proposition}[Regret bound for the Y-player]
\label{y_player_sc}
Fix $t \geq 0$. The strategy
\begin{align}
    Y_s = \nabla \tilde{V}(\bar{X}_{s}),\quad s \in [0,t]
\end{align}
leads to zero regret of the Y-player:
\begin{align}\label{FTL:no_regret_sc}
\Reg_t^{Y} = 0.
\end{align} 
\end{proposition}
\begin{proof}
Since $\tilde{V}$ is assumed to be convex, the proof is analogous to that of Proposition \ref{y_player} with $V \equiv \tilde{V}$ and $\tau = 0$.
\end{proof}

\subsection{Regret of the X-player}\label{sc:noisy}

We now analyse the regret of the X-player's strategy,
\begin{align}\label{noisy_sc}
        \dd X_t & = - \beta_t(Y_t + \sigma X_t)\dd t + \sqrt{\beta_t(3+\sigma \beta_t^2)}\dd W_t,\qquad t \geq 0,
\end{align}
starting from an $\mathbb{R}^d$-valued initial random variable $X_{0}$, where $Y_t$ is assumed to be $\sigma(\bar{X}_t)$-measurable, and where 
$\beta_t = \frac{A_t}{\alpha_t}$. 

\begin{proposition}[Regret bound for the X-player]\label{prop:xplaer_sconvex} 
Fix $t \geq 0$. Denote by $r_s:\mathbb{R}^d\rightarrow \mathbb{R}^d$ the optimal transport map from $\operatorname{Law}(\bar{X}_s)$ to $\pi$, $s \in [0,t]$. Then, under Assumption \ref{ass:FP-regularity} (and imposing the correspondence \eqref{eq:QP-X_sc}), the solution to \eqref{noisy_sc} satisfies   
\begin{equation}\label{regret_estimate_sc}
    \operatorname{Reg}_t^{X} \leq \Psi_0 + \frac{1}{2}\int_0^t \big(\frac{1}{\alpha_s}\frac{\dd}{\dd s}(\frac{\alpha_s^2}{A_s})-\sigma \big)\mathbb{E}[\|X_s - r_s(\bar{X}_s)\|^2]\,\alpha_s\dd s + \frac{d}{2}\int_0^t (\sigma \beta_s^2-1)\alpha_s\dd s + d\int_0^t \dot{\beta}_s\alpha_s\dd s,
\end{equation}
with
\begin{align*}
    \Psi_0 = A_0\mathbb{E}_{\bar{X}_0}[\operatorname{KL}(\operatorname{Law} (X_0|\bar{X}_0)\,\|\,\mathcal{N}(r_0(\bar{X}_0), \beta_0^2 I_d))].
\end{align*}
\end{proposition}

The proof of Proposition \ref{prop:xplaer_sconvex} follows the same lines as that of Proposition \ref{prop:xplayer_convex}. In particular, it is based on estimating the growth of
\begin{equation}\label{eq:def_Psi_t_sc}
t \mapsto \underbrace{A_t\mathbb{E}_{\bar{X}_t}[\operatorname{KL}(\operatorname{Law} (X_{t}|\bar{X}_t)\,\|\,\mathcal{N}(r_t(\bar{X}_t), \beta_t^2 I_d))]}_{=: \Psi_t} + \Reg_t^X,
\end{equation}
from its initial value, which will turn out to be controlled by the time integrals in \eqref{regret_estimate_sc}. Note that the energy function 
\begin{subequations}
    \begin{align}
    \Psi_t & = A_t\mathbb{E}_{\bar{X}_t}[\operatorname{KL}(\operatorname{Law} (X_{t}|\bar{X}_t)\,\|\,\mathcal{N}(r_t(\bar{X}_t), \beta_t^2 I_d))] \\
    & = A_t \mathbb{E}_{\bar{X}_t}[\operatorname{Ent}(X_t|\bar{X}_t)] + \frac{A_t}{2\beta^2_t}\mathbb{E}[\|X_t - r_t(\bar{X}_t)\|^2] + A_t\frac{d}{2}\ln\big({2\pi \beta^2_t}\big) \label{eq:lyapunov_sc_2}
\end{align}
\end{subequations}
is the very same one from Section \ref{convex:noisy}, where we had specified $\beta^2_t = \frac{A^2_t}{\alpha_t^2} = A_t$ as in \eqref{Nesterov-Schedule}, see also Remark \ref{remark:choice_Psit}. For convenience in the proof of Proposition \ref{prop:xplaer_sconvex}, we again derive an explicit bound on the time evolution of the quadratic term in \eqref{eq:lyapunov_sc_2}:

\begin{lemma}\label{L2-distance_sc}
Denote by $r_t:\mathbb{R}^d\rightarrow \mathbb{R}^d$ the optimal transport map from $\text{Law } \bar{X}_t$ to $\pi$. Along the solution to \eqref{noisy_sc}, we have
    \begin{align}\label{L2_derivative_sc}
        & \frac{1}{2}\mathbb{E}\big[\|X_t-r_t(\bar{X}_t) \|^2\big]-\frac{1}{2}\mathbb{E}\big[\|X_s-r_s(\bar{X}_s) \|^2\big]\\
        & \quad \leq \int_s^t \Bigg[-\beta_u \mathbb{E}[\langle X_u - r_u(\bar{X}_u), Y_u +\sigma X_u\rangle] - \frac{1}{\beta_u} \mathbb{E}[C^X_u(\bar{X}_u) : \nabla r_u(\bar{X}_u)]  + \frac{\beta_u(3+\sigma \beta_u^2)}{2} d\Bigg]\dd u,
    \end{align}
where $C^X_t(\bar{x})$ is defined as
\begin{align*}
    C^X_t(\bar{x}) = \mathbb{E}\big[(X_t-\mathbb{E}[X_t|\bar{X}_t]) \otimes (X_t-\mathbb{E}[X_t|\bar{X}_t]) | \bar{X}_t = \bar{x}\big],
\end{align*}
i.e. it is the conditional covariance matrix of $X_t$ given $\bar{X}_t = \bar{x}$. 
\end{lemma}

\begin{proof}
The proof is given in Appendix \ref{sec:proof_L2-distance_convex}.
\end{proof}

Apart from Lemma \ref{L2-distance_sc}, we rely again on Lemma \ref{lemma:ent-cramer-rao} to show \eqref{regret_estimate_sc}. In this sense, the proof below contains no conceptual novelties compared to that of Proposition \ref{prop:xplayer_convex}; however, some algebraic reformulations are needed to relate the impact of the drift $-\beta_t \sigma X_t$ in \eqref{eq: x player sc} back to the regret \eqref{eq:X regret sc}.

\begin{proof}[Proof of Proposition~ \ref{prop:xplaer_sconvex}]
Consider the energy function
\begin{align*}
    \Psi_t & = A_t\mathbb{E}_{\bar{X}_t}[\operatorname{KL}(\operatorname{Law} (X_t|\bar{X}_t)\,\|\,\mathcal{N}(r_t(\bar{X}_t), \beta_t^2 I_d))] \\
    & = A_t \mathbb{E}_{\bar{X}_t}[\operatorname{Ent}(X_t|\bar{X}_t)] + \frac{A_t}{2\beta^2_t}\mathbb{E}[\|X_t - r_t(\bar{X}_t)\|^2] + A_t\frac{d}{2}\ln\big({2\pi \beta^2_t}\big).
\end{align*}
For notational convenience, write again
\begin{align*}
    H_t
    :=
    \mathbb{E}_{\bar{X}_t}
    [\Ent(X_t|\bar{X}_t)], \qquad 
    J(X_t|\bar{X}_t)
    :=
    \mathbb{E}_{\bar{X}_t}
    \big[
        J(X_t|\bar{X}_t=\bar{x})
    \big],
\end{align*}
for the expected conditional entropy and conditional Fisher information. Using the Fokker--Planck equation for $(\bar{X}_u,X_u)$ and Assumption \ref{ass:FP-regularity}, we see that 
\begin{align}\label{eq:integrated-conditional-entropy sc}
A_t H_t - A_s H_s & = \int_s^t \left[\alpha_u H_u +  \alpha_u \operatorname{Ent}(\bar{X}_u) + (1+\sigma \beta_u^2)\alpha_u d - A_u \frac{\beta_u(3+\sigma \beta_u^2)}{2} J(X_u|\bar{X}_u) \right]
    \dd u \\
    & \qquad - A_t \operatorname{Ent}(\bar{X}_t) + A_s \operatorname{Ent}(\bar{X}_s),
\end{align}
for $0 \leq s \leq t$. Note also that
\begin{align}\label{eq:integrated-gaussian-normalization_sc}
    A_t\frac{d}{2}\ln(2\pi \beta^2_t)
    -
    A_s\frac{d}{2}\ln(2\pi \beta^2_s)
    =
    \int_s^t
    \left[
        \alpha_u\frac{d}{2}\ln(2\pi \beta^2_u)
        + A_u d \frac{\dot{\beta}_u}{\beta_u}
    \right]
    \dd u.
\end{align}
Combining \eqref{eq:integrated-conditional-entropy sc}, \eqref{eq:integrated-gaussian-normalization_sc}, and \eqref{L2_derivative_sc} from Lemma \ref{L2-distance_sc}, we obtain
\begin{subequations}
\label{eq:PsitPsis_sc}
\begin{align*}
    \Psi_t-\Psi_s
    &\leq
    \int_s^t
    \Bigg[
        \alpha_u H_u
        +\alpha_u\Ent(\bar{X}_u)  + (1+\sigma \beta_u^2)\alpha_u d - A_u \frac{\beta_u(3+\sigma \beta_u^2)}{2} J(X_u|\bar{X}_u) 
         \\
         &\qquad\qquad -\frac{A_u}{\beta_u}
        \mathbb{E}
        \big[
            \langle
                X_u-r_u(\bar{X}_u),
                y_u(\bar{X}_u) + \sigma X_u
            \rangle
        \big]
        \\
        &\qquad\qquad
        -\frac{A_u}{\beta^3_u}
        \mathbb{E}
        \big[
            C_u^X(\bar{X}_u):
            \nabla_{\bar{x}}r_u(\bar{X}_u)
        \big]
        +\alpha_u\frac{d}{2}\ln(2\pi \beta^2_u) + A_u d \frac{\dot{\beta}_u}{\beta_u} + \frac{A_u(3+\sigma \beta_u^2)}{2\beta_u}d   \\
        &\qquad\qquad + \frac{\dd}{\dd u}\left(\frac{A_u}{\beta_u^2} \right)\frac{1}{2}\mathbb{E}[\|X_u - r_u(\bar{X}_u)\|^2]
    \Bigg]
    \dd u 
    -A_t\Ent(\bar{X}_t)
    +A_s\Ent(\bar{X}_s). \tag{\ref{eq:PsitPsis_sc}}
\end{align*}
\end{subequations}
Now, recall that by the 
definition of the regret of the X-player from Definition \ref{def:regrets_sc}:
\begin{subequations}
\label{eq:RegtRegs_sc}
\begin{align*}
    \Reg_t^X-\Reg_s^X
    &=
    \int_s^t
    \mathbb{E}[\langle X_u,Y_u\rangle]\alpha_u\dd u
    -
    \int_s^t
    \sup_{R\sim\pi}
    \mathbb{E}[\langle R,Y_u\rangle]\alpha_u\dd u+ A_t(\Ent(\bar{X}_t) + \frac{\sigma}{2}\mathbb{E}[\|\bar{X}_t\|^2]) \\
    & \qquad 
    -A_s(\Ent(\bar{X}_s) + \frac{\sigma}{2}\mathbb{E}[\|\bar{X}_s\|^2])
    -(A_t-A_s)\left(\Ent(\pi) + \frac{\sigma}{2}\mathbb{E}_{X \sim \pi}[\|X\|^2])\right).
\end{align*}
\end{subequations}
Since
\begin{align*}
    \mathbb{E}[\langle r_u(\bar{X}_u), Y_u\rangle] \leq \sup_{R \sim \pi} \mathbb{E}[\langle R,Y_u\rangle],
\end{align*}
we can in particular write 
\begin{subequations}\label{eq:tech_1}
\begin{align}
& -\int_s^t \frac{A_u}{\beta_u}\mathbb{E}\big[\langle X_u-r_u(\bar{X}_u),y_u(\bar{X}_u)\rangle
\big]\dd u- A_t \operatorname{Ent}(\bar{X}_t) + A_s \operatorname{Ent}(\bar{X}_s) \\
& \qquad\leq  - \operatorname{Reg}_t^X + \operatorname{Reg}_s^X + A_t \frac{\sigma}{2}\mathbb{E}[\|\bar{X}_t\|^2] - A_s\frac{\sigma}{2}\mathbb{E}[\|\bar{X}_s\|^2] - (A_t - A_s)(\operatorname{Ent}(\pi) + \frac{\sigma}{2}\mathbb{E}_{X \sim \pi}[\|X\|^2]).
\end{align}
\end{subequations}
Now note that since $\dd \bar{X}_t = \frac{\alpha_t}{A_t}(X_t - \bar{X}_t) \dd t$, 
\begin{align*}
  \frac{\dd}{\dd u}\left( \frac{A_u}{2} \mathbb{E}[\|\bar{X}_u\|^2]\right)& =\frac{\alpha_u}{2}\left(\E\norm{X_u}^2-\E\norm{X_u-\bar{X}_u}^2\right) \\
  & = \alpha_u\left(\mathbb{E}[\langle X_u - r_u(\bar{X}_u), X_u\rangle] + \frac{1}{2}\mathbb{E}_{X \sim \pi}[\|X\|^2] - \frac{1}{2}\mathbb{E}[\|X_u - r_u(\bar{X}_u)\|^2] -\frac{1}{2}\E\norm{X_u-\bar{X}_u}^2\right),
\end{align*}
which can be rearranged to 
\begin{subequations}\label{eq:tech_2}
\begin{align}
& A_t \frac{\sigma}{2}\mathbb{E}[\|\bar{X}_t\|^2] - A_s\frac{\sigma}{2}\mathbb{E}[\|\bar{X}_s\|^2]  - (A_t - A_s)\frac{\sigma}{2}\mathbb{E}_{X \sim \pi}[\|X\|^2] \\
& \qquad = \int_s^t \left[\sigma \mathbb{E}[\langle X_u - r_u(\bar{X}_u), X_u\rangle] - \frac{\sigma}{2}\mathbb{E}[\|X_u - r_u(\bar{X}_u)\|^2] -\frac{\sigma}{2}\E\norm{X_u-\bar{X}_u}^2\right]\alpha_u\,du.
\end{align}
\end{subequations}
Using $\beta_u = \frac{A_u}{\alpha_u}$, combining \eqref{eq:tech_1} and \eqref{eq:tech_2} and inserting them into \eqref{eq:PsitPsis_sc} yields
\begin{subequations}
\label{eq:PsitPsis_sc_2}
\begin{align*}
    \Psi_t + \Reg_t^X -(\Psi_s + \Reg_s^X)
    &\leq
    \int_s^t\Bigg[H_u+ (1+\sigma \beta_u^2)d - \frac{\beta^2_u(3+\sigma \beta_u^2)}{2} J(X_u|\bar{X}_u) -\frac{1}{\beta^2_u}\mathbb{E}\big[ C_u^X(\bar{X}_u):\nabla_{\bar{x}}r_u(\bar{X}_u)\big]\\
    &\qquad\qquad+\frac{d}{2}\ln(2\pi \beta^2_u) + \frac{3+\sigma \beta_u^2}{2}d  - \frac{\sigma}{2}\E\norm{X_u-\bar{X}_u}^2 + \Ent(\bar{X}_u) - \Ent(\pi)
    \Bigg]
    \alpha_u\dd u  \\
    &\qquad\qquad + \underbrace{\frac{1}{2}\int_s^t
    (\frac{1}{\alpha_u}\frac{\dd}{\dd u}\left(\frac{\alpha^2_u}{A_u}\right)- \sigma)\mathbb{E}[\|X_u - r_u(\bar{X}_u)\|^2]
    \alpha_u\dd u + d \int_s^t \dot{\beta}_u \alpha_u\,du}_{(\star)}.
\end{align*}
\end{subequations}
Notice that the last two terms in \eqref{eq:PsitPsis_sc_2}, denoted by $(\star)$, already appear in the upper bound in \eqref{regret_estimate_sc}. Thus, we focus on the first time integral in the remainder of the proof. Again, we introduce the random variable 
\begin{align*}
    P_u = \frac{1}{\beta_u}(X_u-\bar{X}_u),
\end{align*}
which indeed corresponds to the physical underdamped Langevin velocity $P_t$, see \eqref{eq:QP-X_sc}. Following Lemma \ref{lemma:scaling_shifting}, we have 
\begin{align}
    \operatorname{Ent}(X_u| \bar{X}_u = {\bar{x}}) & = \operatorname{Ent}(P_u| \bar{X}_u = {\bar{x}}) - \ln{\beta_u^{d}}, \label{simplification_1}\\
    \beta^2_u J(X_u | \bar{X}_u = {\bar{x}}) & =  J(P_u | \bar{X}_u = {\bar{x}}),\\ 
    \frac{1}{\beta^2_u} C^X_u({\bar{x}}) & = C^P_u({\bar{x}}).
\end{align}
Furthermore, note that 
\begin{align*}
    \operatorname{Ent}(P_u|\bar{X}_u={\bar{x}}) + \frac{d}{2}\ln{2\pi} = \operatorname{KL}(\operatorname{Law}(P_u|\bar{X}_u={\bar{x}})\,\|\,\mathcal{N}(0,I_d)) - \frac{1}{2}\mathbb{E}[\|P_u\|^2|\bar{X}_u={\bar{x}}],
\end{align*}
and
\begin{align}\label{relative fisher-information_2}
-J(P_u | \bar{X}_u = {\bar{x}}) + 2d = -\mathcal{I}(\operatorname{Law}(P_u|\bar{X}_u={\bar{x}})|\mathcal{N}(0,I_d)) +\mathbb{E}[\|P_u\|^2|\bar{X}_u={\bar{x}}].
\end{align}

With these identities, we get 
\begin{align*}
    &(\Psi_t+\Reg_t^X)-(\Psi_s+\Reg_s^X)
    \leq
    \int_s^t
    \Bigg\{
        \Ent(\bar{X}_u)-\Ent(\pi)
        +
        \mathbb{E}_{\bar{X}_u}
        \Bigg[
            \KL
            \left(
                \Law(P_u|\bar{X}_u)
                \,\middle|\,
                \mathcal{N}(0,I_d)
            \right)
            \nonumber\\
    &
            -\frac{3}{2}
            \mathcal{I}
            \left(
                \Law(P_u|\bar{X}_u)
                \,\middle|\,
                \mathcal{N}(0,I_d)
            \right)
            +
            \mathbb{E}
            [\|P_u\|^2|\bar{X}_u]
            -
            C_u^P(\bar{X}_u):
            \nabla_{\bar{x}}r_u(\bar{X}_u)
        \Bigg]
    \Bigg\}
    \alpha_u\dd u + (\star) + d\int_s^t \frac{1}{2}(\sigma \beta_u^2 -1)\alpha_u\,du \\
    & - \frac{\sigma}{2}\int_s^t \left\{\mathbb{E}[\|X_u - \bar{X}_u\|^2] + \beta_u^2\left(\mathbb{E}_{\bar{X}_u}[\mathcal{I}(\operatorname{Law}(P_u|\bar{X}_u)|\mathcal{N}(0,I_d))]-\mathbb{E}[\|P_u\|^2]\right) \right\} \alpha_u\dd u.
\end{align*}
The first time integral in the right-hand side can be bounded above by zero as in the proof of Proposition \ref{prop:xplayer_convex}, by reasoning as in \eqref{eq:intermediate-bound} - \eqref{eq:before-ent-cramer-rao} and using Lemma \ref{lemma:ent-cramer-rao}. The third term, $d\int_s^t \frac{1}{2}(\sigma \beta_u^2 -1)\alpha_u\,du$, is explicitly contained in the regret estimate in eq. \eqref{regret_estimate_sc} and needs no further treatment. Thus, it remains to show that the last time integral can be estimated by zero. For this, simply note that $X_u - \bar{X}_u = \beta_u P_u$, thus the integrand can be written as 
\begin{align*}
    & - \mathbb{E}[\|X_u - \bar{X}_u\|^2] - \beta_u^2 \left(\mathbb{E}_{\bar{X}_u}[\mathcal{I}(\operatorname{Law}(P_u|\bar{X}_u)|\mathcal{N}(0,I_d))]-\mathbb{E}[\|P_u\|^2]\right) \\
    & \qquad =  - \beta_u^2 \mathbb{E}_{\bar{X}_u}[\mathcal{I}(\operatorname{Law}(P_u|\bar{X}_u)|\mathcal{N}(0,I_d))] \leq 0.
\end{align*}
With this, we can now estimate
\begin{align*}
     &(\Psi_t+\Reg_t^X)-(\Psi_s+\Reg_s^X) \\
     & \qquad \leq \frac{1}{2}\int_s^t \big(\frac{1}{\alpha_u}\frac{\dd}{\dd u}(\frac{\alpha_u^2}{A_u})-\sigma \big)\mathbb{E}[\|X_u - r_u(\bar{X}_u)\|^2]\alpha_u \dd u + \frac{d}{2}\int_s^t (\sigma \beta_u^2-1)\alpha_u \dd u + d\int_s^t \dot{\beta}_u\alpha_u \dd u,
\end{align*}
and the statement follows for $s = 0$ since $\Reg_0^X =0$ and $\Psi_t \geq 0$.
\end{proof}

\subsection{Overall convergence result}\label{sc:convergence}

\begin{theorem}\label{theorem:convergence_sc}
If $\beta_t = \frac{1}{\sqrt{\sigma}}$, the solution to the coupled dynamics \eqref{eq: y player sc}-\eqref{eq: x player sc} satisfies, for $t \geq 0$ 
\begin{align*}
    \mathcal{F}(\operatorname{Law}(\bar{X}_t)) - \mathcal{F}(\pi) \leq e^{-\sqrt{\sigma}t}\operatorname{KL}(\operatorname{Law}(\bar{X}_0, X_0)\|\Pi_0),
\end{align*}
where 
\begin{align*}
\Pi_0(\dd \bar{x},\dd x) = \pi(\dd \bar{x})\,\mathcal{N}\big(r_0(\bar{x}), \sigma^{-1}I_d\big)(\dd x).
\end{align*}
\end{theorem}

\begin{proof}
First, note that with the choice
$$
\beta_t=\frac{A_t}{\alpha_t}\equiv\frac{1}{\sqrt{\sigma}},
$$
all time-integral terms in the upper bound \eqref{regret_estimate_sc} of Proposition \ref{prop:xplaer_sconvex} vanish. It then only remains to insert the estimates of Propositions \ref{y_player_sc} and \ref{prop:xplaer_sconvex} into the bound of Proposition \ref{theorem:sub-optimality_sc} to conclude that
    \begin{align*}
        \operatorname{KL}(\operatorname{Law}(\bar{X}_{t}) \| \pi) = \mathcal{F}(\operatorname{Law}(\bar{X}_t)) - \mathcal{F}(\pi) \leq \frac{\Psi_0}{A_t}+\frac{A_{0}}{A_t}\operatorname{KL}(\operatorname{Law}(\bar{X}_{0}) \| \pi).
    \end{align*}
    Finally, note that, by the chain rule for the KL divergence,
    \begin{align*}
        \Psi_0 + A_{0}\operatorname{KL}(\operatorname{Law}(\bar{X}_{0}) \| \pi) = A_{0}\operatorname{KL}(\operatorname{Law}(\bar{X}_{0}, X_{0}) \,\|\, \Pi_{0}).
    \end{align*}.
\end{proof}

We conclude this section by stating the proof of Theorem \ref{theorem:UL_strongly-convex}.

\begin{proof}[Proof of Theorem~\ref{theorem:UL_strongly-convex}]
Following the change of variables \eqref{eq:QP-X_sc} and It\^{o}'s formula \cite[Lemma 3.2]{pavliotis2014stochastic},  the game system \eqref{eq: y player sc}-\eqref{eq: x player sc} can be written in the form of the physical system \eqref{eq:UL} as
\begin{align*}
    \dd P_t & = \frac{\alpha_t}{A_t}\dd X_t - \frac{\alpha_t}{A_t}\dd \bar{X}_t + \frac{\dd}{\dd t}\left(\frac{\alpha_t}{A_t}\right)(X_t - \bar{X}_t)\dd t \\
    & = -\nabla V(Q_t)\dd t +\left(\frac{A_t}{\alpha_t}\frac{\dd}{\dd t}\left(\frac{\alpha_t}{A_t}\right) - \frac{\alpha_t}{A_t} - \sigma \frac{A_t}{\alpha_t}\right)P_t\dd t + \sqrt{\frac{3+\sigma \beta_t^2}{\beta_t}}\dd W_t.
\end{align*}
Thus, if $\gamma_t \equiv 2\sqrt{\sigma}$ in \eqref{eq:UL}, the dynamics coincide if the weighting is chosen as 
\begin{align*}
\beta_t=\frac{A_t}{\alpha_t}\equiv\frac{1}{\sqrt{\sigma}},
\end{align*}
meaning that
\begin{align*}
    A_t = e^{\sqrt{\sigma} t},\;\;\;\alpha_t = \sqrt{\sigma}e^{\sqrt{\sigma} t}.
\end{align*}
Since 
$P_{0} = \frac{\alpha_{0}}{A_{0}} (X_0 - \bar{X}_0) = \sqrt{\sigma}(X_0 - \bar{X}_0)$ is an invertible linear transformation, the conditional KL divergences relate as:
\begin{align*}
\operatorname{KL}\left(\operatorname{Law}(X_{0}|\bar{X}_{0})\, \|\,\mathcal{N}(r_0(\bar{X}_0), \beta^2_{0}I_d) \right) = \operatorname{KL}\left(\operatorname{Law}(P_{0}|\bar{X}_{0})\, \|\,\mathcal{N}(\sqrt{\sigma}(r_0(\bar{X}_0)-\bar{X}_0), I_d) \right).
\end{align*}
Since $Q_t = \bar{X}_t$, Theorem \ref{theorem:convergence_sc} gives the statement.
\end{proof}

We conclude this section by stating the proof of Corollary \ref{cor:UL_strongly-convex}. 

\begin{proof}[Proof of Corollary~\ref{cor:UL_strongly-convex}]
Consider the convergence statement in Theorem \ref{theorem:UL_strongly-convex}. In analogy to the extension in Remark \ref{rem:previous work technical} (a), we write
\begin{align*}
    \operatorname{KL}&(\operatorname{Law}(Q_0, P_0)\, \|\, \Pi_0)  = 
    \operatorname{KL}(\operatorname{Law}(Q_{0}, P_{0})\, \|\, \pi \otimes \mathcal{N}(0,I_d)) + \sqrt{\sigma}\mathcal{C}_{OT}(\operatorname{Law}(Q_{0}, P_{0})) + \frac{\sigma}{2}\mathcal{W}_2^2(\operatorname{Law}(Q_0), \pi).
\end{align*}    
For the choice $\operatorname{Law}(P_{0}\, | \, Q_0) = \mathcal{N}(0,I_d)$, it holds that $\mathcal{C}_{OT}(\operatorname{Law}(Q_{0}, P_{0}))=0$, and the above equation reduces to 
\begin{align}
    \operatorname{KL}&(\operatorname{Law}(Q_0, P_0)\, \|\, \Pi_0) = \operatorname{KL}(\operatorname{Law}(Q_{0}) \| \pi) + \frac{\sigma}{2}\mathcal{W}_2^2(\operatorname{Law}(Q_0), \pi),
\end{align}
which proves the claim.
\end{proof}

\begin{remark}[Relation between the strongly log-concave and log-concave sampling game]\label{rem:schedule}
The presentation in this paper moves from the log-concave sampling game to the strongly log-concave sampling game to clarify the gaming framework and the resulting proof technique. 
However, as the attentive reader might have noticed, the other direction would have been equally plausible. Setting $\sigma = 0$, the strongly log-concave sampling game recovers the log-concave one.
\leavevmode\par
\begin{enumerate}[label=(\alph*)]
    \item  We highlight in particular that the X-player's strategy \eqref{noisy}, given by $\dd X_t = -\alpha_t Y_t + \sqrt{3 \alpha_t}$, is an instance of \eqref{noisy_sc}, given by $\dd X_t = - \beta_t(Y_t + \sigma X_t)\dd t + \sqrt{\beta_t(3+\sigma \beta_t^2)}\dd W_t$. If one sets $\sigma = 0$ and chooses $\beta_t = \frac{A_t}{\alpha_t} = \alpha_t$, the strategies coincide. The latter is true for the eventually employed Nesterov schedule \eqref{Nesterov-Schedule}. 
    \item In this context, the more general presentation of the strongly log-concave game provides a reason for the explicit form of the Nesterov schedule \eqref{Nesterov-Schedule}. Setting $\sigma = 0$, the X-player's strategy \eqref{noisy_sc} becomes $\dd X_t = -\beta_t Y_t \dd t+ \sqrt{3 \beta_t}\dd W_t$, leaving $\beta_t = \frac{A_t}{\alpha_t}$ unspecified. The upper bound on the regret of the X-player \eqref{regret_estimate_sc} suggests that the weights should satisfy
\begin{align}
    \frac{\dd}{\dd t}\left(\frac{A_t}{\beta_t^2}\right) = \frac{\dd}{\dd t}\left(\frac{\alpha^2_t}{A_t}\right) = 0,
\end{align}
to cancel the integrated quadratic contribution.
This is in particular fulfilled if $A_t = \frac{t^2}{4}$, this means if $\alpha_t = \frac{t}{2}$ and $A_{\tau} = \frac{\tau^2}{4}$ by Definition \ref{def:weights}. Moreover, with this choice of weight, the remaining time integrals in \eqref{regret_estimate_sc} also cancel. In this sense, the Nesterov schedule \eqref{Nesterov-Schedule} optimises the upper bound on the X-player's regret and thus, by Proposition \ref{theorem:sub-optimality_sc}, also on the whole sub-optimality estimate.
\item Note that the Polyak schedule \eqref{Polyak-Schedule}, that was directly imposed in Theorem \ref{theorem:convergence_sc}, can also be motivated by optimising the upper bound in \eqref{regret_estimate_sc} in the case where $\sigma \neq 0$; see also Remark \ref{rem:schedule_P} in the strongly convex optimisation game.
\end{enumerate}
\end{remark}

\section{Appendix}\label{sec:Appendix}

\subsection{Fenchel games for accelerated gradient flows on $\mathbb{R}^d$}
\label{app:Fenchel Rd}
This appendix presents a continuous-time adaptation of the game-based framework for the convex minimisation task from \cite{fenchelgames}. We show that Nesterov's ODE \eqref{eq:nesterov_ODE} and Polyak's ODE \eqref{eq:polyak_ODE} can be understood as describing the time evolution of actions in this game. In particular, this interpretation allows us to quantify their convergence rates. Following lines similar to the sampling game, we refer to the main text for further remarks and references, but present a concise, streamlined version for those interested in extracting the core ideas. Section \ref{sec:nesterov_ODE} treats convex $f:\mathbb{R}^d \rightarrow \mathbb{R}$, and Section \ref{sec:polyak} treats $\sigma$-strongly convex $f$. Throughout, assume that $f$ is continuously differentiable and that $\text{argmin}_{x\in \mathbb{R}^d} f \neq \emptyset$.

\subsubsection{Nesterov's ODE}\label{sec:nesterov_ODE}

In this section, we embed Nesterov's ODE \eqref{eq:nesterov_ODE} into a game-based framework. To avoid constraints on the initial conditions, the initial time is $\tau > 0$. Thus, for $(q_{\tau},\dot{q}_{\tau}) \in \mathbb{R}^d \times \mathbb{R}^d$, we consider the solution to
\begin{align*}
    \ddot{q}_t = - \nabla f(q_t) - \frac{3}{t}\dot{q}_t, \qquad t \geq \tau > 0
\end{align*}
or, equivalently, for $(q_{\tau},p_{\tau}) \in \mathbb{R}^d \times \mathbb{R}^d$, to
\begin{subequations}
\label{eq:N}
\begin{align}
    \dot{q}_t & = p_t \\
    \dot{p}_t & = - \nabla f(q_t) - \frac{3}{t}p_t.
\end{align}
\end{subequations}
    
Our goal in this section is to verify the well-established $t^{-2}$ objective-value convergence rate from \cite{NesterovODE}, see \eqref{eq:conv_Nesterov}. However, we will only return to the exact form of Nesterov's ODE in Theorem \ref{theorem:convergence_N} once we have built up -- in analogy to \cite{fenchelgames} -- a game-based way of constructing and analysing descent dynamics to solve the unconstrained minimisation problem 
\begin{align}\label{eq:min}
\min_{x \in \mathbb{R}^d} f(x).
\end{align}
We then treat Nesterov's ODE as an example of one of them. First, we may rewrite \eqref{eq:min} via the saddle point problem
\begin{align*}
\min_{x \in \mathbb{R}^d} \max_{y \in \mathbb{R}^d}\, \langle x,y\rangle - f^*(y),
\end{align*}
in which a fictional x-player seeks to minimise, and a fictional y-player seeks to maximise the payoff function $g(x,y) = \langle x,y\rangle - f^*(y)$. Note that $f(x) = \max_{y \in \mathbb{R}^d}g(x,y)$. The x- and y-players' performance over time is quantified via regrets, which compare their actions against a competitor. For $\tau > 0$, we use a prescribed weight function $\alpha$ and initial weight $A_{\tau}$ to define the integrated weight $A_t$ as in Definition \ref{def:weights}, and use these to average the losses in time. The corresponding weighted average \eqref{eq:X bar complicated} is denoted by $\bar{x}_t$ and will serve as the output variable of the game. We now define the regrets of the x- and y-players, which compare their accumulated payoffs against those of competitors. 

\begin{definition}[Regrets for convex minimisation]\label{def:regrets_n}
For $x^* \in \text{argmin}_{x \in \mathbb{R}^d} f(x)$ and for played actions $\{(x_s,y_s)\}_{\tau \leq s\leq t}$, we define the regrets at time $t\geq\tau$ as
\begin{equation}
\label{eq:RegX Rd}
 \operatorname{Reg}_t^x  = \int_{\tau}^t \langle x_s, y_s\rangle\alpha_s\dd s  - \int_{\tau}^t \langle x^*, y_s\rangle \alpha_s\dd s,
\end{equation}
and
\begin{subequations}
\label{eq:RegY}
\begin{align}
  \operatorname{Reg}_t^y & = A_{\tau}(f^*(y_{\tau})-\langle \bar{x}_{\tau},y_{\tau}\rangle) + \int_{\tau}^t  \left(f^*(y_s) - \langle x_s, y_s\rangle\right)\alpha_s\dd s \\
  & \;\;\; - \min_{y \in \mathbb{R}^d}\left\{A_{\tau}\left(f^*(y)-\langle \bar{x}_\tau,y\rangle \right)+ \int_{\tau}^t \left( f^*(y)-\langle x_s, y\rangle \right) \alpha_s\dd s\right\}.
\end{align}
\end{subequations}
\end{definition}
Thus, the competing strategy for the x-player consists of playing $x^* \in \text{argmin}_{x \in \mathbb{R}^d}\,f(x)$ in all rounds, while the competing strategy for the y-player is an in-hindsight optimal, constant strategy. The relevance of the regrets is rooted in the fact that they are designed to control the sub-optimality of the weighted variable $\bar{x}_t$ \eqref{eq:X bar complicated}:
\begin{proposition}\label{theorem:sub-optimality_N}
The regrets from Definition \ref{def:regrets_n} satisfy 
\begin{align*}
    f(\bar{x}_t) - f(x^*) \leq \frac{\Reg_t^x + \Reg_t^y}{A_t} + \frac{A_{\tau}}{A_t}\left(f(\bar{x}_{\tau}) - f(x^*)\right), \qquad t \geq \tau > 0.
\end{align*}
\end{proposition}

\begin{proof} Adding the two regrets in \eqref{eq:RegX Rd} and \eqref{eq:RegY} yields \begin{align*} \Reg_t^x + \Reg_t^y & = - \int_{\tau}^t \langle x^*, y_s\rangle\alpha_s\dd s + A_{\tau}\left(f^*(y_{\tau})-\langle \bar{x}_{\tau}, y_{\tau}\rangle \right) + \int_{\tau}^t f^*(y_s)\alpha_s\dd s\\ & \;\;\; - \min_{y \in \mathbb{R}^d}\left\{ A_{\tau}\left(f^*(y)-\langle \bar{x}_\tau,y\rangle \right) + \int_{\tau}^t \left(f^*(y)-\langle x_s,y\rangle\right)\alpha_s\dd s \right\}. \end{align*} 
The last line can be rewritten as \begin{align*} \min_{y \in \mathbb{R}^d}\left\{ A_{\tau}\left(f^*(y)-\langle \bar{x}_\tau,y\rangle \right) + \int_{\tau}^t \left(f^*(y)-\langle x_s,y\rangle\right)\alpha_s\dd s \right\} = -A_t\max_{y \in \mathbb{R}^d} \{\langle \bar{x}_t, y\rangle - f^*(y)\} = -A_t f(\bar{x}_t). \end{align*} For the remaining term, the Fenchel--Young inequality \eqref{eq:Fenchel IE} gives \begin{align*} A_{\tau}\left(f^*(y_{\tau})-\langle \bar{x}_{\tau}, y_{\tau}\rangle \right) \geq -A_{\tau}f(\bar{x}_{\tau}). \end{align*} Moreover, by \eqref{eq:biconjugate_id}, it holds that \begin{align*} \int_{\tau}^t \langle x^*, y_s\rangle\alpha_s\dd s - \int_{\tau}^t f^*(y_s)\alpha_s\dd s\leq \int_{\tau}^t \max_{y \in \mathbb{R}^d}\left\{ \langle x^*, y\rangle - f^*(y)\right\}\alpha_s\dd s = (A_t - A_{\tau})f(x^*). \end{align*} Rearranging yields the claim. \end{proof}

We now propose possible strategies for the x- and y-players. Let $(\bar{x}_{\tau}, x_{\tau}) \in \mathbb{R}^d \times \mathbb{R}^d$ be given. For $0 < \tau \leq t$, consider 
\begin{subequations}
\label{eq:game_Rd}
\begin{align}
y_t & = \nabla f(\bar{x}_t) \label{eq:game_Rd_N1}\\
\frac{d x_t}{d t} & = - \alpha_t y_t\label{eq:game_Rd_N2}.
\end{align}
\end{subequations}
The y-player plays the standard follow-the-leader strategy and the x-player an online gradient flow (see \cite[Section 2.2]{shalevshwartz2012online}, \cite[Chapter 6]{modernintroductiononlinelearning} or \cite[Algorithm 4 and 5]{fenchelgames} for discrete-time versions). Provided that we can bound the regrets $\operatorname{Reg}_t^X$ and $\operatorname{Reg}_t^Y$ along \eqref{eq:game_Rd_N1}-\eqref{eq:game_Rd_N2} uniformly in time, Proposition \ref{theorem:sub-optimality_N} implies a decay in the objective value $f$ with rate $A_t^{-1}$. Thus, the following two Propositions \ref{prop:y-player_N} and \ref{prop:x-player_N} estimate the regrets of the individual players, which are combined to estimate the convergence of the coupled system in Theorem \ref{theorem:convergence_game_N}.
\begin{proposition}[Regret bound for the y-player]\label{prop:y-player_N}
Fix $t \geq \tau$. The strategy
\begin{align}\label{FTL_N}
    y_s = \nabla f(\bar{x}_{s}),\quad s \in [\tau,t],
\end{align}
leads to zero regret of the y-player:
\begin{align}\label{FTL:no_regret_N}
\operatorname{Reg}_t^{y} = 0.
\end{align} \end{proposition}
\begin{proof}
Recall that the regret of the y-player is, by Definition \ref{def:regrets_n}, given by 
\begin{align*}
\operatorname{Reg}_t^y & = A_{\tau}(f^*(y_{\tau})-\langle \bar{x}_{\tau},y_{\tau}\rangle) + \int_{\tau}^t \left(f^*(y_s) - \langle x_s, y_s\rangle\right) \alpha_s\dd s \\
  & \;\;\; - \min_{y \in \mathbb{R}^d}\left\{A_{\tau}\left(f^*(y)-\langle \bar{x}_\tau,y\rangle \right)+ \int_{\tau}^t \left( f^*(y)-\langle x_s, y\rangle \right) \alpha_s\dd s\right\}.
\end{align*}
The second line, recording the accumulated loss of the competitor, can be rewritten as
\begin{align*}
    \min_{y \in \mathbb{R}^d}\left\{A_{\tau}\left(f^*(y)-\langle \bar{x}_\tau,y\rangle \right)+ \int_{\tau}^t \left( f^*(y)-\langle x_s, y\rangle \right) \alpha_s\dd s\right\} = \min_{y \in \mathbb{R}^d}\left\{A_t\left(f^*(y)-\langle \bar{x}_t,y\rangle \right)\right\} = -A_tf(\bar{x}_t)
\end{align*}
by property \eqref{eq:biconjugate_id} of the Fenchel conjugates. Since $y_{\tau} = \nabla f(\bar{x}_{\tau})$, it further holds that $A_{\tau}(f^*(y_{\tau})-\langle \bar{x}_{\tau},y_{\tau}\rangle) = -A_{\tau}f(\bar{x}_{\tau})$ due to property \eqref{FY_Eq}, and the regret can be written as 
\begin{align*}
\operatorname{Reg}_t^y & = A_{\tau}f(\bar{x}_{\tau}) - A_t f(\bar{x}_t)  + \int_{\tau}^t \left(f^*(y_s) - \langle x_s, y_s\rangle\right)\alpha_s\dd s.
\end{align*}
Thus, to show the result \eqref{FTL:no_regret_N} it remains to show that
\begin{align*}
    \frac{\dd}{\dd s}\left(A_s f(\bar{x}_s)\right) = -(f^*(y_s) - \langle x_s, y_s\rangle)\alpha_s.
\end{align*}
Indeed, since $\bar{x}_s = \frac{A_{\tau}}{A_s}\bar{x}_{\tau} + \frac{1}{A_s}\int_\tau^sx_u\alpha_u \dd u$ and using again \eqref{FY_Eq}:
\begin{align*}
    \frac{\dd}{\dd s}\left(A_s f(\bar{x}_s)\right) & = \alpha_s f(\bar{x}_s) + A_s \langle \nabla f(\bar{x}_s), \frac{\alpha_s}{A_s}(x_s - \bar{x}_s)\rangle = \alpha_s \left(f(\bar{x}_s) + \langle \nabla f(\bar{x}_s), x_s - \bar{x}_s\rangle \right) \\
    & = - \alpha_s\left( f^*(\nabla f(\bar{x}_s)) - \langle \nabla f(\bar{x}_s),x_s\rangle \right) =  - \alpha_s( f^*(y_s) - \langle y_s,x_s\rangle),
\end{align*}
which shows the claim. 
\end{proof}
\begin{proposition}[Regret bound for the x-player]\label{prop:x-player_N}
Fix $t \geq \tau > 0$. The solution to \eqref{eq:game_Rd_N2} satisfies
\begin{align*}
     \operatorname{Reg}_t^x \leq \frac{1}{2}\|x_{\tau}-x^*\|^2.
\end{align*}
\end{proposition}
\begin{proof}
Differentiating the energy function $\psi_t = \frac{1}{2}\|x_t-x^*\|^2$ yields,
\begin{align*}
    \frac{\dd}{\dd t}\psi_t = \langle x_t - x^*, \dot{x}_t\rangle = -\alpha_t \langle x_t - x^*, y_t\rangle = - \frac{\dd}{\dd t}\operatorname{Reg}_t^x.
\end{align*}
We integrate from $\tau$ to $t$ and rearrange to 
\begin{align*}
    \operatorname{Reg}_t^X=\psi_{\tau} - \psi_{t}\leq \psi_{\tau}.
\end{align*} 
This proves the claim.
\end{proof}
\begin{theorem}\label{theorem:convergence_game_N}
Fix $t \geq \tau > 0$. The combined strategies \eqref{eq:game_Rd_N1}-\eqref{eq:game_Rd_N2} yield
\begin{align*}
    f(\bar{x}_t) - f(x^*) \leq \frac{\frac{1}{2}\|x_{\tau} - x^*\|^2 + A_{\tau}\left(f(\bar{x}_{\tau}) - f(x^*)\right)}{A_t}.
\end{align*}
\end{theorem}
\begin{proof}
Apply Propositions \ref{prop:y-player_N} and \ref{prop:x-player_N} to Proposition \ref{theorem:sub-optimality_N}.
\end{proof}
We now use the previous considerations to show the convergence rate along the solution to Nesterov's ODE \eqref{eq:N}:
\begin{theorem}\label{theorem:convergence_N}
For $q^* \in \text{argmin}_{q \in \mathbb{R}^d} f(q)$, the solution to Nesterov's ODE \eqref{eq:N} satisfies  
\begin{align*}
    f(q_t) - f(q^*) \leq \frac{2\|q_{\tau}+\frac{\tau}{2}p_{\tau} -  q^*\|^2+\tau^2\left(f(q_{\tau}) - f(q^*)\right)}{t^2}.
\end{align*}
\end{theorem}
\begin{proof}
The idea is to use Theorem \ref{theorem:convergence_game_N} with the weighting $A_t = \frac{t^2}{4}$, $\alpha_t = \frac{t}{2}$. Thus, it only remains to verify that the corresponding game dynamics \eqref{eq:game_Rd_N1}-\eqref{eq:game_Rd_N2} are equivalent to \eqref{eq:N}. First, set $q_t = \bar{x}_t$. Since $\dot{q}_t = p_t$, it follows that $p_t = \frac{\alpha_{t}}{A_t}(x_t - q_t)$. Plugging in the x- and y-players' strategies, we obtain
\begin{align*}
    \dot{p}_t & = \big(\frac{\alpha_t}{A_t}\big)'(x_t - q_t) + \frac{\alpha_t}{A_t}(- \alpha_t \nabla f(q_t) - p_t) \\
    & = \big(\big(\frac{\alpha_t}{A_t}\big)'\frac{A_t}{\alpha_t}-\frac{\alpha_t}{A_t} ) p_t - \frac{\alpha_t^2}{A_t}\nabla f(q_t).
\end{align*}
With the choice $A_t = \frac{t^2}{4}$, $\alpha_t = \frac{t}{2}$ this indeed yields
\begin{align*}
    \dot{p}_t = - \nabla f(q_t) - \frac{3}{t}p_t,
\end{align*}
and the claim follows.
\end{proof}

\begin{remark}
We remark that the above regret-based analysis is closely related to the usual Lyapunov function for Nesterov's ODE, given by
\begin{align*}
    \psi_t = \frac{t^2}{4}(f(q_t) - f(x^*)) + \frac{1}{2}\|q_t + \frac{t}{2}p_t - x^*\|^2 = \frac{t^2}{4}(f(\bar{x}_t) - f(x^*)) + \frac{1}{2}\|x_t - x^*\|^2, 
\end{align*}
see \cite{NesterovODE}. The above presentation essentially builds up this Lyapunov function in a structured way and, in doing so, offers a novel interpretation of Nesterov's ODE as a combination of game strategies. Moreover, it motivates extending this reasoning from the task of minimisation to that of sampling.
\end{remark}

\begin{remark}
    We highlight the difference between the minimisation game and the sampling game in terms of the cost of the x-player, see \eqref{conv:x_regret} and \eqref{eq:RegX Rd}. In the sampling game, she has to pay the differential entropy for her averaged actions, which forces her to stochasticise the deterministic optimisation strategy.
    \end{remark}

\subsubsection{Polyak's ODE}\label{sec:polyak}

We continue by elaborating how Polyak's ODE \eqref{eq:polyak_ODE} can likewise be regarded as coupled strategies for a strongly convex optimisation game. Based on this, we will confirm the objective-value convergence \eqref{eq:conv_Polyak} in Theorem \ref{theorem:convergence_P}. Recall that Polyak's ODE is given by
\begin{align*}
    \ddot {q}_t = - \nabla f(q_t) - 2\sqrt\sigma \dot{q}_t,
\end{align*}
with $(q_0,\dot{q}_0) \in \mathbb{R}^d \times \mathbb{R}^d$, or equivalently, 
\begin{subequations}\label{eq:P}
    \begin{align}
    \dot{q}_t & = p_t \label{eq:P_1}\\ 
    \dot{p}_t & = - \nabla f(q_t) - 2\sqrt{\sigma}p_t\label{eq:P_2},
\end{align}
\end{subequations}
with $(q_0, p_0) \in \mathbb{R}^d \times \mathbb{R}^d$. 
Here, by introducing the convex function $\tilde{f}(x) = f(x) - \frac{\sigma}{2}\|x\|^2$, we replace the minimisation problem 
\begin{align*}
    \min_{x \in \mathbb{R}^d} f(x).
\end{align*}
by the saddle point problem
\begin{align}\label{eq:sp_sc}
    \min_{x \in \mathbb{R}^d} \max_{y \in \mathbb{R}^d}\,\langle x,y\rangle - \tilde{f}^*(y) + \frac{\sigma}{2}\|x\|^2.
\end{align}

As in the previous section, we now define appropriate regrets of the x- and y-players, which will allow us to quantify the sub-optimality in terms of $f$ in Proposition \ref{theorem:sub-optimality_P}. Starting the game at $\tau = 0$, we prescribe a weighting function $\alpha$ and an initial weight $A_0 > 0$ to obtain the integrated weight $A_t$ as given in Definition \ref{def:weights}. We denote by $\bar{x}_t$ the accordingly weighted variable \eqref{eq:X bar complicated}. We now define the regrets of the strongly convex game, based on the adapted payoff function in \eqref{eq:sp_sc}. 

\begin{definition}[Regrets for strongly convex minimisation]\label{def:regrets_p}
For $x^* \in \text{argmin}_{x \in \mathbb{R}^d} f(x)$ and for played actions $\{(x_s,y_s)\}_{0 \leq s\leq t}$, we define the regrets at time $t\geq 0$ as
\begin{align*}
 \operatorname{Reg}_t^x & =  \int_{0}^t \langle x_s, y_s\rangle\alpha_s\dd s + \frac{\sigma A_t}{2}\|\bar{x}_t\|^2 - \frac{\sigma A_0}{2}\|\bar{x}_0\|^2 - \int_{0}^t \langle x^*, y_s\rangle\alpha_s \dd s -\frac{\sigma (A_t-A_0)}{2}\|x^*\|^2, \\
  \operatorname{Reg}_t^y & = A_{0}(\tilde{f}^*(y_{0})-\langle \bar{x}_{0},y_{0}\rangle) + \int_{0}^t (\tilde{f}^*(y_s) - \langle x_s, y_s\rangle) \alpha_s \dd s\\
  & \;\;\; - \min_{y \in \mathbb{R}^d}\big\{A_{0}(\tilde{f}^*(y)-\langle \bar{x}_0,y\rangle)+\int_{0}^t (\tilde{f}^*(y)-\langle x_s, y\rangle)\alpha_s \dd s\big\}.
\end{align*}
\end{definition}
Note that we have $\operatorname{Reg}_0^x = 0$ by construction. The following theorem establishes that these regrets, too, can be used to assess the sub-optimality of $\bar{x}_t$ in terms of $f$.
\begin{theorem}\label{theorem:sub-optimality_P}
The regrets from Definition \ref{def:regrets_p} satisfy 
\begin{align*}
    f(\bar{x}_t) - f(x^*) \leq \frac{\operatorname{Reg}_t^x + \operatorname{Reg}_t^y}{A_t} + \frac{A_{0}}{A_t}\left(f(\bar{x}_{0}) - f(x^*)\right),\qquad t \geq 0.
\end{align*}
\end{theorem}
\begin{proof}
The proof is analogous to that of Proposition \ref{theorem:sub-optimality_N}.
\end{proof}

Let $(\bar{x}_{0}, x_{0}) \in \mathbb{R}^d \times \mathbb{R}^d$ be given. We now analyse the regrets of the following strategies:
\begin{subequations}
\label{eq:game_Rd_P}
    \begin{align}
    y_t & = \nabla \tilde{f}(\bar{x}_t),  \label{eq:game_Rd_P1}\\
    \frac{d x_t}{d t} & = -\beta_t (y_t + \sigma x_t) \label{eq:game_Rd_P2},
\end{align}
\end{subequations}
for $\beta_t = \frac{A_t}{\alpha_t}$ and $t \geq 0$. Again, the y-player plays the follow-the-leader strategy, but on the convex function $\tilde{f}$ instead of the strongly convex function $f$. The x-player plays an online gradient flow on the function $x \mapsto \langle x, y_t\rangle + \frac{\sigma}{2}\|x\|^2$. Since \eqref{eq:game_Rd_P1} coincides with the strategy in \eqref{eq:game_Rd_N1}, where $f$ was assumed to be convex, its regret vanishes again, see Proposition \ref{prop:y-player_P} below. In particular, this motivates attributing the strongly convex term $\frac{\sigma}{2}\|x\|^2$ to the x-player in \eqref{eq:sp_sc}: attributing it to the y-player would not improve the regret estimate. By contrast, making the x-player pay $\frac{\sigma}{2}\|x\|^2$ directly translates to a curvature-informed gradient flow \eqref{eq:game_Rd_P2}, whose regret is given in Proposition \ref{prop:x-player_P}.
\begin{proposition}[Regret bound for the y-player]\label{prop:y-player_P}
Fix $t \geq 0$. The strategy
\begin{align}\label{FTL_P}
    y_s = \nabla \tilde{f}(\bar{x}_{s}),\quad s \in [0,t]
\end{align}
leads to zero regret of the y-player:
\begin{align}\label{FTL:no_regret_P}
\operatorname{Reg}_t^{y} = 0.
\end{align} \end{proposition}
\begin{proof}
The proof is analogous to that of Proposition  \ref{prop:y-player_N}, replacing $f$ by $\tilde{f}$.
\end{proof}
\begin{proposition}[Regret bound for the x-player]\label{prop:x-player_P}
Fix $t \geq 0$. The solution to \eqref{eq:game_Rd_P2} satisfies
\begin{align}\label{eq:x_player_reg_P}
     \operatorname{Reg}_t^x \leq \frac{\alpha^2_0}{2A_0}\|x_{0}-x^*\|^2 + \frac{1}{2} \int_0^t \left(\frac{1}{\alpha_s}\frac{\dd}{\dd s}\big(\frac{\alpha_s^2}{A_s} \big) - \sigma \right)\|x_s - x^*\|^2\,\alpha_s\dd s.
\end{align}
\end{proposition}
\begin{proof}
Differentiating the energy function $\psi_t = \frac{A_t}{2\beta_t^2}\|x_t-x^*\|^2$ yields
\begin{align*}
    \frac{\dd}{\dd t}\psi_t & = \big(\frac{A_t}{2\beta_t^2} \big)'\|x_t - x^*\|^2 + \frac{A_t}{\beta_t^2}\langle x_t - x^*, \dot{x}_t\rangle \\
    & = \big(\frac{A_t}{2\beta_t^2} \big)'\|x_t - x^*\|^2 - \frac{A_t}{\beta_t}\langle x_t - x^*, y_t + \sigma x_t \rangle.
\end{align*}
Using that
\begin{align*}
    \langle x_t - x^*, x_t\rangle = \frac{1}{2}\big(\|x_t\|^2 + \|x_t-x^*\|^2 - \|x^*\|^2\big)
\end{align*}
and
\begin{align*}
    \frac{1}{\alpha_t}\frac{\dd}{\dd t}(\frac{A_t}{2}\|\bar{x}_t\|^2) = \frac{1}{2} \|\bar{x}_t\|^2 +  \langle \bar{x}_t, x_t - \bar{x}_t\rangle  = \frac{1}{2} \|x_t\|^2 -  \frac{1}{2}\|\bar{x}_t - x_t\|^2,
\end{align*}
allows to rewrite 
\begin{align*}
    -\frac{A_t \sigma}{\beta_t} \langle x_t - x^*, x_t \rangle & = -\frac{A_t \sigma}{2\beta_t}\|x_t - x^*\|^2  + \frac{A_t \sigma}{2\beta_t}\|x^*\|^2 - \frac{A_t \sigma}{2\beta_t}\|x_t\|^2 \\
    & = -\frac{A_t \sigma}{2\beta_t}\|x_t - x^*\|^2  + \frac{A_t \sigma}{2\beta_t}\|x^*\|^2 - \frac{A_t \sigma}{\beta_t}\big(\frac{1}{\alpha_t}\frac{\dd}{\dd t}(\frac{A_t}{2}\|\bar{x}_t\|^2)+ \frac{1}{2}\|\bar{x}_t - x_t\|^2 \big) \\
    & = -\frac{A_t \sigma}{2\beta_t}\big(\|x_t - x^*\|^2 + \|\bar{x}_t - x_t\|^2\big) - \frac{A_t}{\alpha_t\beta_t}\frac{\dd}{\dd t}(\frac{A_t\sigma }{2}\|\bar{x}_t\|^2-\frac{A_t\sigma}{2}\|x^*\|^2) \\
    & \leq -\frac{A_t \sigma}{2\beta_t}\|x_t - x^*\|^2 - \frac{A_t}{\alpha_t\beta_t}\frac{\dd}{\dd t}(\frac{A_t\sigma }{2}\|\bar{x}_t\|^2-\frac{A_t\sigma}{2}\|x^*\|^2) 
\end{align*}
Thus, it holds that
\begin{align*}
    \frac{\dd}{\dd t}\psi_t & \leq \Big(\big(\frac{A_t}{2\beta_t^2} \big)' - \frac{A_t \sigma}{2\beta_t}\Big) \|x_t - x^*\|^2 - \frac{A_t}{\beta_t}\langle x_t - x^*, y_t \rangle - \frac{A_t}{\alpha_t\beta_t}\frac{\dd}{\dd t}(\frac{A_t\sigma }{2}\|\bar{x}_t\|^2-\frac{A_t\sigma}{2}\|x^*\|^2) \\
    & = \Big(\big(\frac{\alpha_t^2}{2A_t} \big)' - \frac{\alpha_t \sigma}{2}\Big) \|x_t - x^*\|^2 - \alpha_t \langle x_t - x^*, y_t \rangle - \frac{\dd}{\dd t}(\frac{A_t\sigma }{2}\|\bar{x}_t\|^2-\frac{A_t\sigma}{2}\|x^*\|^2)  \\
    & = \Big(\big(\frac{\alpha_t^2}{2A_t} \big)' - \frac{\alpha_t \sigma}{2}\Big) \|x_t - x^*\|^2 - \frac{\dd}{\dd t} \operatorname{Reg}_t^x,
\end{align*}
where we used $\beta_t = \frac{A_t}{\alpha_t}$ to simplify the expressions. Integration from $0$ to $t$ and rearranging yields the regret estimate.
\end{proof}
Based on the preceding considerations, the coupled game dynamics \eqref{eq:game_Rd_P} satisfy the following optimality estimate:
\begin{theorem}\label{theorem:convergence_game_P}
Fix $t \geq 0$. The combined strategies \eqref{eq:game_Rd_P1}-\eqref{eq:game_Rd_P2} yield
\begin{align}\label{eq:xplayer_suboptimality_P}
    f(\bar{x}_t) - f(x^*) \leq  \frac{A_0}{2\beta^2_0A_t}\|x_{0}-x^*\|^2 + \frac{A_{0}}{A_t}\left(f(\bar{x}_{0}) - f(x^*)\right)+\frac{1}{2A_t} \int_0^t \left(\frac{1}{\alpha_s}\frac{\dd}{\dd s}\big(\frac{\alpha_s^2}{A_s} \big) - \sigma \right)\|x_s - x^*\|^2\,\alpha_s\dd s.
\end{align}
\end{theorem}
\begin{proof}
Apply Propositions \ref{prop:y-player_P} and \ref{prop:x-player_P} to Proposition \ref{theorem:sub-optimality_P}.
\end{proof}
\begin{remark}\label{rem:schedule_P}
We can observe that specific weighting choices lead to a smaller bound on the regret of the x-player in \eqref{eq:x_player_reg_P} and thus on the sub-optimality in \eqref{eq:xplayer_suboptimality_P}: If
\begin{align*}
    \frac{\dd}{\dd s}\big(\frac{A_s}{\beta_s^2} \big) = \frac{\dd}{\dd s}\big(\frac{\alpha_s^2}{A_s} \big) = \alpha_s \sigma, 
\end{align*}
the integrated quadratic contributions vanish. If $\sigma > 0$, this suggests $\alpha_s = \sqrt{\sigma}e^{\sqrt{\sigma}s}$, $A_0 = 1$, and thus by Definition \ref{def:weights} $A_s = e^{\sqrt{\sigma}s}$.
\end{remark}
We conclude this section by noting that Polyak's ODE \eqref{eq:P} and the game dynamics \eqref{eq:game_Rd_P} are related by a change of variables, allowing us to establish the following convergence result: 
\begin{theorem}\label{theorem:convergence_P}
For $q^* \in \text{argmin}_{q \in \mathbb{R}^d} f(q)$, the solution to Polyak's ODE \eqref{eq:P_1}-\eqref{eq:P_2} satisfies
\begin{align*}
    f(q_t) - f(q^*) \leq e^{-\sqrt{\sigma}t}\left( \frac{\sigma}{2}\|q_{0} + \frac{1}{\sqrt{\sigma}}p_0 - q^*\|^2 + f(q_0) - f(q^*)\right).
\end{align*}
\end{theorem}
\begin{proof}
To use Theorem \ref{theorem:convergence_game_P}, we need to establish which game strategies of the general form \eqref{eq:game_Rd_P1}-\eqref{eq:game_Rd_P2} correspond to Polyak's ODE \eqref{eq:P}. Choose the schedule $A_t = e^{\sqrt{\sigma}t}$, $\alpha_t = \sqrt{\sigma}e^{\sqrt{\sigma}t}$. Then, \eqref{eq:game_Rd_P1}-\eqref{eq:game_Rd_P2} can be combined to  
\begin{align*}
    \frac{\dd x_t}{\dd t} & = -\frac{1}{\sqrt{\sigma}}(\nabla \tilde{f}(\bar{x}_t) + \sigma x_t) = -\frac{1}{\sqrt{\sigma}}(\nabla f(\bar{x}_t) + \sigma (x_t - \bar{x}_t)).
\end{align*}
With the change of variables $q_t = \bar{x}_t$ and $p_t = \sqrt{\sigma}(x_t - q_t)$, this transforms to Polyak's ODE,
\begin{align*}
    \frac{\dd p_t}{\dd t} = - \nabla f(q_t) - 2\sqrt{\sigma}p_t. 
\end{align*}
Applying Theorem \ref{theorem:convergence_game_P} and considering Remark \ref{rem:schedule_P}, the claim follows.
\end{proof}

\subsection{Saddle point considerations}\label{sec:saddle_point}

The analysis of the sampling games above is motivated by the introduction of the formal saddle point problem 
\begin{align}\label{eq:saddle_point_appendix}
    \min_{X} \max_{Y} \mathcal{G}(X,Y),
\end{align}
where $X$ and $Y$ denote appropriately integrable $\mathbb{R}^d$-valued random variables and
\begin{align}\label{eq:G_appendix}
    \mathcal{G}(X,Y) := \mathbb{E}[\langle Y, X \rangle - V^*(Y)] + \operatorname{Ent}(X),
\end{align}
see Section \ref{sec:games}. A natural way of reading this is that the outer problem fixes the marginal distribution in $X$ (which is assumed to have a Lebesgue density), while the inner problem looks for random variables that are optimally coupled to $X$. However, considering the meaning of the reversed problem
\begin{align}\label{eq:sp_rev}
    \max_Y \min_X \mathcal{G}(X,Y),
\end{align}
shows that one has to be careful when expressing the problem purely for jointly defined random variables; exchanging the order of optimisation does not simply exchange the roles of the two marginals. Interpreting \eqref{eq:sp_rev} as meaning that the marginal distribution of $Y$ is determined in the outer problem, and that the inner problem looks for $X$ optimally coupled to $Y$, makes the inner problem unbounded, as can be seen by taking $X \perp Y$, with $X \sim \mathcal{N}(0,\varepsilon^2)$, $\varepsilon > 0$, and letting $\varepsilon \rightarrow \infty$. Thus, a more sensible problem formulation would be based on defining the payoff function directly for a probability measure $\rho$ and Markov kernel $K:x \mapsto K_x(\dd y)$ as 
\begin{align}\label{eq:sp_g}
    \mathcal{\widehat{G}}(\rho, K) = \int \int (\langle x, y\rangle - V^*(y))K_x(\dd y) \rho(\dd x) + \operatorname{Ent}(\rho).
\end{align}
However, for ease of presentation in this paper and a more explicit connection to the accelerated gradient flows on $\mathbb{R}^d$ mentioned above, we retain random variable formulations throughout. In this sense, we interpret a saddle point $(X^*, Y^*)$ of $\mathcal{G}$ in \eqref{eq:G_appendix} as induced by a saddle point $(\rho^*, K^*)$ of $\mathcal{\widehat{G}}$ in \eqref{eq:sp_g}, meaning $(X^*, Y^*) \sim \rho^*(\dd x)K^*(x,\dd y)$, where 
\begin{align}\label{eq:sp_def}
    \mathcal{\widehat{G}}(\rho^*,K) \leq \mathcal{\widehat{G}}(\rho^*,K^*) \leq \mathcal{\widehat{G}}(\rho,K^*),
\end{align}
for all admissible probability measures $\rho$ and Markov kernels $K$. By using the properties of the Fenchel conjugate \eqref{eq:biconjugate_id} and \eqref{gradient_maximises}, one can argue that the saddle point of $\mathcal{\widehat{G}}$, in the sense of \eqref{eq:sp_def}, is given by 
\begin{align*}
    (\rho^*, K^*)=(\frac{1}{Z}e^{-V(x)}\dd x, \delta_{\nabla V(x)}(dy)),
\end{align*}
implying that the corresponding saddle point of $\mathcal{G}$ is marginally distributed as those solving the sampling problem $\min_X \mathcal{F}(\operatorname{Law } X)$.

\subsection{Auxiliary lemmas}
\label{app:lemmas}

\begin{lemma}[Differential Entropy, Fisher-Information and Covariance Matrix under Rescaling and Shifting]\label{lemma:scaling_shifting}

Let $X$ be an $\mathbb{R}^d$-valued random variable with density $\rho_X$. For $\theta > 0$ and $a \in \mathbb{R}^d$, define 
    $U := \theta(X-a)$.
Then, denoting the density of $U$ by $\rho_U$, the differential entropies of $X$ and $U$ are related by
\begin{align}\label{entropy_scaled_shifted}
    \operatorname{Ent}(\rho_X) = \operatorname{Ent}(\rho_U) + \ln{\theta^{d}}.
\end{align}
Their Fisher information values are related as
\begin{align}\label{fi_scaled_shifted}
    \int_{\mathbb{R}^d} \| \nabla_x \ln{\rho_X}(x)\|^2\,d\rho_X(x) = \theta^2 \int_{\mathbb{R}^d} \|\nabla_u \ln{\rho_U}(u)\|^2\,d\rho_U(u).
\end{align}
Finally, the covariance matrix of $X$, $C_X$, and that of $U$, $C_U$, satisfy
\begin{align}\label{cov_scaled_shifted}
    C_X = \frac{1}{\theta^2} C_U.
\end{align}
The identities \eqref{entropy_scaled_shifted}, \eqref{fi_scaled_shifted}, and \eqref{cov_scaled_shifted} hold whenever the quantities appearing on either side are well-defined and finite.
\end{lemma}

\begin{proof}
By the change of variables formula, the densities of $X$ and $U$ are related as 
\begin{align*}
    \rho_U(u) = \theta^{-d} \rho_X \left(\frac{1}{\theta} u +a \right).
\end{align*}
Plugging this into $\operatorname{Ent}(\rho_U) = \int \rho_U(u)\ln{\rho_U(u)}\dd u$ and $\int \rho_U(u)\|\nabla_u \ln{\rho_U}(u)\|^2\dd u$, respectively, yields  \eqref{entropy_scaled_shifted} and  \eqref{fi_scaled_shifted}. Moreover, since $\mathbb{E}[U]-U = \theta(\mathbb{E}[X]-X)$, we have \eqref{cov_scaled_shifted}.
\end{proof}

\begin{lemma}[Entropy Change]\label{ma}
Let $\rho = f(x)\dd x, \, \mu=g(x)\dd x \in \mathcal P_2(\mathbb R^d)$ be given and assume that $\operatorname{Ent}(\rho), \operatorname{Ent}(\mu) \in \mathbb{R}$. Denote by $T:\mathbb R^d\to\mathbb R^d$ the unique optimal transport map for the quadratic cost from $\rho$ to $\mu$ and by $\nabla T$ its Alexandrov derivative. It holds that 
\begin{align}\label{eq:diff_entr}
        \operatorname{Ent}(\rho) = \operatorname{Ent}(\mu) + \int_{\mathbb{R}^d} \ln \det(\nabla T(x))f(x)\dd x.
    \end{align}
Moreover, $\nabla T(x) \succ 0$, $\rho$-almost everywhere. 
\end{lemma}
\begin{proof}
Under the given assumptions, it holds $\rho\text{-almost everywhere}$ that
\begin{align}\label{eq:monge-ampere}
f(x)=g(T(x))\det(\nabla T(x)),
\end{align}
see \cite[Example 11.2]{villani2009optimal} and
\cite[Section 1.4]{mccann2013five}. Moreover, we have both $0 < f(x) <\infty$ and $0 < g(T(x)) < \infty$ for $\rho$\text{-almost every} $x$, since $T_{\#}\rho = \mu$. Thus, \eqref{eq:monge-ampere} implies $\det{\nabla T(x)} > 0$ for $\rho$-almost every $x$. However, since $T(x)$ is the gradient of a convex function (see e.g. \cite[Theorem 1.12]{villani2003topics}), $\nabla T(x) \succeq 0$. In combination, this implies that $\nabla T(x) \succ 0$ for $\rho$-almost every $x$. Further, we may take logarithms in \eqref{eq:monge-ampere} and integrate with respect to $\rho$ to obtain \eqref{eq:diff_entr}.
\end{proof}

\begin{lemma}\label{fisher-information-and-traces}
Let $\rho \in \mathcal{P}_2(\mathbb{R}^d)$ have a Lebesgue density. Denote its Fisher information matrix by $\mathcal{J}$, its Fisher information by $J = Tr(\mathcal{J})$, and its covariance matrix by $C$. Assume that $\mathcal{J} \succ 0$. Then, for any matrix $A \succ 0$ we have
    \begin{align*}
        C:A + J \geq 2\Tr(A^{1/2}).
    \end{align*}
\end{lemma}

\begin{proof}
By the Cramer-Rao inequality (see e.g. \cite{KAGAN19997} and note that by assumption $C \succ 0$), the Fisher information matrix $\mathcal{J}$ and the covariance matrix $C$ are related as 
\begin{align*}
    \mathcal{J} \succeq C^{-1},
\end{align*}
meaning that $\mathcal{J} - C^{-1}$ is positive semidefinite. This implies
\begin{align*}
    J_X \geq \operatorname{Tr}(C^{-1}_X).
\end{align*}
Now, for a matrix $X \in \mathbb{R}^{d \times d}$, denote its Frobenius norm by  $\|X\|_F = \sqrt{\operatorname{Tr}(X^TX)}$. Then, we obtain the desired result by writing
\begin{align*}
    C:A + J \geq \operatorname{Tr}(C^T A) + \operatorname{Tr}(C^{-1}) = \|A^{1/2}C^{1/2} - C^{-1/2}\|^2_{F}  + 2\operatorname{Tr}(A^{1/2}) \geq 2\operatorname{Tr}(A^{1/2}).
\end{align*}
\end{proof}

\begin{proof}[Proof of Lemma \ref{lemma:ent-cramer-rao}]
By Lemma~\ref{ma}, we have
\begin{align*}
    \Ent(\rho)-\Ent(\pi)
    =
    \int_{\mathbb R^d}
    \ln\det\bigl(\nabla r(\bar{x})\bigr)
    \,\rho(\dd\bar{x}).
\end{align*}
Hence, subtracting the left-hand side from the right-hand side in \eqref{eq:transport-fisher-compensation}, it remains to show that 
\begin{align*}
    \int_{\mathbb R^d}
    \Big[
        2d
        -J(\nu_{\bar{x}})
        -C(\nu_{\bar{x}}):\nabla r(\bar{x})
        +\ln\det\bigl(\nabla r(\bar{x})\bigr)
    \Big]
    \,\rho(\dd\bar{x})
\end{align*}
is non-positive. By Lemma~\ref{fisher-information-and-traces}, applied with
$A=\nabla r(\bar{x})\succ0$, we have, for $\rho$-almost every $\bar{x}$,
\begin{align*}
    J(\nu_{\bar{x}})
    +
    C(\nu_{\bar{x}}):\nabla r(\bar{x})
    \geq
    2\Tr\bigl(\nabla r(\bar{x})^{1/2}\bigr).
\end{align*}
Therefore,
\begin{align*}
    &2d
    -J(\nu_{\bar{x}})
    -C(\nu_{\bar{x}}):\nabla r(\bar{x})
    +\ln\det\bigl(\nabla r(\bar{x})\bigr)
    \le
    2d
    -2\Tr\bigl(\nabla r(\bar{x})^{1/2}\bigr)
    +\ln\det\bigl(\nabla r(\bar{x})\bigr).
\end{align*}
Let $\lambda_1,\ldots,\lambda_d>0$ denote the eigenvalues of
$\nabla r(\bar{x})$, see Lemma~\ref{ma}. Then
\begin{equation}
 \label{eq:log inequality}  .
    2d
    -2\Tr\bigl(\nabla r(\bar{x})^{1/2}\bigr)
    +\ln\det\bigl(\nabla r(\bar{x})\bigr)
=
    2\sum_{i=1}^d
    \left(
        1-\sqrt{\lambda_i}
        +\ln\sqrt{\lambda_i}
    \right)
    \le 0,
\end{equation}
where the last inequality follows from
$\ln z\leq z-1$ for $z>0$.
Integrating with respect to $\rho$ proves
\eqref{eq:transport-fisher-compensation}.
\end{proof}

\subsection{Proof of Lemmas~\ref{lemma:L2-distance_convex} and \ref{L2-distance_sc}}
\label{sec:proof_L2-distance_convex}

For convenience, define the shorthand
\begin{equation}
    F(s)
    :=
    \frac12
    \mathbb E\big[
        \|X_s-r_s(\bar X_s)\|^2
    \big].
\end{equation}
As hinted at in Remark
\ref{rem:Brenier regularity}, the Brenier map $r_s$ (and hence $F$) is in general not differentiable in $s$. To work around this obstacle, we define the surrogate
\begin{equation}
\label{eq:def F tilde}
    \widetilde F_t(s)
    :=
 \frac12
    \mathbb E\big[
        \|X_s-\widetilde{r}_{t,s}(\bar X_s)\|^2
    \big] ,
\end{equation}
where
\begin{equation}
\label{eq:def r tilde}
    \widetilde r_{t,s}(\Phi_{t,s}(\bar x))
    :=
    r_t(\bar x),
\end{equation}
i.e., $\widetilde{r}_{t,s}$ is a transported version of the Brenier map $r_t$, using the marginal-flow $\Phi_{t,s}$ defined in \eqref{eq:flow ODE}.
\begin{equation}
    \widetilde r_{t,t}=r_t,
    \qquad
    (\widetilde r_{t,s})_\#\rho_s=\pi,
\end{equation}
for all $s,t \in I$ (where $I$ is the compact interval from Assumption \ref{ass:FP-regularity}), and that $\widetilde{r}_{t,s}$ is differentiable in $s$. For reference below, the ODE \eqref{eq:flow ODE} can be written in $(\bar{X}_t,X_t)$-coordinates as
\begin{equation}
\label{eq:flow property}
    \frac{\dd}{\dd s}\Phi_{t,s}(\bar x)
    =
    \frac{\alpha_s}{A_s}
    \left(
        \mathbb{E}\!\left[
            X_s
            \,\middle|\,
            \bar X_s=\Phi_{t,s}(\bar x)
        \right]
        -
      \Phi_{t,s}(\bar x)
    \right).
\end{equation}
To prove Lemma \ref{lemma:L2-distance_convex}, we will first differentiate $\widetilde{F}_t$ along the game strategies, see Lemma \ref{lemma:frozen-map-derivative} and \ref{lemma:frozen-map-derivative_sc}, and then establish a comparison estimate between $\widetilde{F}_t$ and $F$ in Lemma \ref{lemma:frozen-map-comparison}. We note that similar proof strategies (freezing optimal transport maps and comparing to smooth-in-time comparators) are common; see, for example, \cite[Sections 8.4 and 8.5]{ambrosio2008gradient} or the recent related work \cite[proof of Lemma 4.1]{lu2026sharp}.
\begin{lemma}
\label{lemma:frozen-map-derivative} Along the solution to \eqref{noisy},
\begin{align}
\label{eq:frozen-map-derivative}
    \left.
    \frac{\dd}{\dd s}\widetilde F_t(s)
    \right|_{s=t}
    ={}&
    -\alpha_t
    \mathbb{E}\!\left[
        \left\langle
            X_t-r_t(\bar X_t),Y_t
        \right\rangle
    \right]
    -\frac{\alpha_t}{A_t}
    \mathbb{E}\!\left[
        C_t^X(\bar X_t):
        \nabla r_t(\bar X_t)
    \right]
    +\frac{3}{2}\alpha_t d.
\end{align}
\end{lemma}

\begin{proof}
In the following, we will denote by $\eta_t$ the Lebesgue density of $(\bar{X}_t,X_t)$. Note that, since by \eqref{eq:QP-X_c}, $(\bar{X}_t,X_t)$ is a linear transformation from $(Q_t,P_t)$ with locally bounded coefficients, such a density exists and inherits the regularity from $\widehat{\rho}_t$ in Assumption \ref{ass:FP-regularity}.
We differentiate \eqref{eq:def F tilde} at \(s=t\):
\begin{align}
\label{eq:frozen-F-derivative-start}
    \left.
    \frac{\dd}{\dd s}\widetilde F_t(s)
    \right|_{s=t}
    =
    \frac12
    \int_{\mathbb{R}^{2d}}
    \|x-r_t(\bar x)\|^2
    \partial_t\eta_t(\bar x,x)
    \dd\bar x\dd x
    -
    \int_{\mathbb{R}^{2d}}
    \left\langle
        x-r_t(\bar x),
        \left.
        \partial_s\widetilde r_{t,s}(\bar x)
        \right|_{s=t}
    \right\rangle
    \eta_t(\bar x,x)
    \dd\bar x\dd x .
\end{align}
The Fokker--Planck equation for $\eta_t$ is  
\begin{equation}
\label{eq:eta FKP}
\partial_t\eta_t(\bar x,x)
    =
    -\frac{\alpha_t}{A_t}
    \nabla_{\bar x}\cdot
    \bigl((x-\bar x)\eta_t(\bar x,x)\bigr)
    +
    \alpha_t y_t(\bar x)\cdot\nabla_x\eta_t(\bar x,x)
    +
    \frac32\alpha_t\Delta_x\eta_t(\bar x,x),
\end{equation}
where we have used the notation $Y_t = y_t(\bar{X}_t)$.
Substituting \eqref{eq:eta FKP} into the first term of
\eqref{eq:frozen-F-derivative-start} and integrating by parts yields
\begin{align*}
    \frac12
    \int_{\mathbb{R}^{2d}}
    \|x-r_t(\bar x)\|^2
    \partial_t\eta_t(\bar x,x)
    \dd\bar x\dd x
    ={}&
    -\frac{\alpha_t}{A_t}
    \int_{\mathbb{R}^{2d}}
    \left\langle
        x-r_t(\bar x),
        \nabla_{\bar x}r_t(\bar x)(x-\bar x)
    \right\rangle
    \eta_t(\bar x,x)
    \dd\bar x\dd x
    \\
    &-
    \alpha_t
    \mathbb{E}\!\left[
        \left\langle
            X_t-r_t(\bar X_t),Y_t
        \right\rangle
    \right]
    +
    \frac32\alpha_t d .
\end{align*}
For the second term in \eqref{eq:frozen-F-derivative-start}, we differentiate \eqref{eq:def r tilde} with respect to $s$ and, using \eqref{eq:flow property}, we see that 
\begin{align}
\label{eq:transported-map-equation}
    \partial_s\widetilde r_{t,s}(\bar x)
    =
    -\frac{\alpha_s}{A_s}
    \nabla_{\bar x}\widetilde r_{t,s}(\bar x)
    \left(
        \mathbb{E}[X_s\mid\bar X_s=\bar x]-\bar x
    \right).
\end{align}
Plugging this into \eqref{eq:frozen-F-derivative-start} and rearranging, we arrive at 
\begin{align*}
\left.
    \frac{\dd}{\dd s}\widetilde F_t(s)
    \right|_{s=t} = & \, - \alpha_t
    \mathbb{E}\!\left[
        \left\langle
            X_t-r_t(\bar X_t),Y_t
        \right\rangle
    \right]
    +
    \frac32\alpha_t d 
    \\
    &    -\frac{\alpha_t}{A_t}
    \int_{\mathbb{R}^{2d}}
    \left\langle
        x-r_t(\bar x),
        \nabla_{\bar x}r_t(\bar x)
        \left(
            x-\mathbb{E}[X_t\mid\bar X_t=\bar x]
        \right)
    \right\rangle
    \eta_t(\bar x,x)
    \dd\bar x\dd x
\end{align*}
Using
$   \mathbb{E}\left[
        X_t-\mathbb{E}[X_t\mid\bar X_t]
        \,\middle|\,
        \bar X_t
    \right]
    =0$ and the definition of the covariance $C_t^X$, we finally see that 
\begin{align*}
    \int_{\mathbb{R}^{2d}}
    \left\langle
        x-r_t(\bar x),
        \nabla_{\bar x}r_t(\bar x)
        \left(
            x-\mathbb{E}[X_t\mid\bar X_t=\bar x]
        \right)
    \right\rangle
    \eta_t(\bar x,x)
    \dd\bar x\dd x =
\mathbb{E}\!\left[
C_t^X(\bar X_t):
        \nabla r_t(\bar X_t)
    \right],
\end{align*}
proving the lemma.
\end{proof}

\begin{lemma}
\label{lemma:frozen-map-derivative_sc} Along the solution to \eqref{noisy_sc},
\begin{align}
\label{eq:frozen-map-derivative_sc}
    \left.
    \frac{\dd}{\dd s}\widetilde F_t(s)
    \right|_{s=t}
    ={}&
    -\beta_t
    \mathbb{E}\!\left[
        \left\langle
            X_t-r_t(\bar X_t),Y_t + \sigma X_u
        \right\rangle
    \right]
    -\frac{1}{\beta_t}
    \mathbb{E}\!\left[
        C_t^X(\bar X_t):
        \nabla r_t(\bar X_t)
    \right]
    +\frac{\beta_t(3+\sigma \beta_t^2)}{2} d .
\end{align}
\end{lemma}
\begin{proof}
The proof is analogous to that of Lemma \ref{lemma:frozen-map-derivative}. We only need to replace the Fokker--Planck equation by  
\begin{align*}
\partial_t\eta_t(\bar x,x)
    =
    -\frac{\alpha_t}{A_t}
    \nabla_{\bar x}\cdot
    \bigl((x-\bar x)\eta_t(\bar x,x)\bigr)
    +
    \beta_t \nabla_x\cdot ((y_t(\bar x)+\sigma x)\eta_t(\bar x,x))
    +
    \frac{\beta_t(3+\sigma \beta_t^2)}{2}\Delta_x\eta_t(\bar x,x).
\end{align*}
\end{proof}

\begin{lemma}
\label{lemma:frozen-map-comparison}
For every compact interval $I \subset (0,\infty)$, we have 
\begin{equation}
\int_I\frac{(F(t+h) - \widetilde{F}_t(t+h))_+}{h} \dd t \rightarrow 0, \qquad h \rightarrow 0^+.
\end{equation}
\end{lemma}

\begin{proof}
For $h > 0$, expand
\begin{subequations}
\begin{align}
& F(t + h) - \widetilde{F}_t(t+h)  = \frac{1}{2} \mathbb{E} \left[ \| X_{t + h} -  r_{t+ h} (\bar{X}_{t+h}) \|^2 \right] - \frac{1}{2} \mathbb{E} \left[ \| X_{t + h} -  \widetilde{r}_{t,t+ h} (\bar{X}_{t+h}) \|^2 \right]
\\
& = \mathbb{E} \left[ \langle X_{t + h}, \widetilde{r}_{t, t+ h} (\bar{X}_{t+h}) - r_{t+h}(\bar{X}_{t + h}) \rangle \right] + \frac{1}{2} \left( \mathbb{E} \left[\|r_{t + h}(\bar{X}_{t + h})\|^2 \right]- \mathbb{E} \left[ \|\widetilde{r}_{t,t + h}(\bar{X}_{t + h})\|^2 \right] \right).
\label{eq:F comparison2}
\end{align}
\end{subequations}
By construction, both $r_{t + h}(\bar{X}_{t + h})$ and $\widetilde{r}_{t,t + h}(\bar{X}_{t + h})$ are distributed according to $\pi$. Therefore, the last two terms in \eqref{eq:F comparison2} cancel. To treat the first term, write 
\begin{subequations}
\label{eq:add subtract}
\begin{align}
\mathbb{E} \left[ \langle X_{t + h}, \widetilde{r}_{t, t+ h} (\bar{X}_{t+h}) - r_{t+h}(\bar{X}_{t + h}) \rangle \right] = & \mathbb{E} \left[ \langle \bar{X}_{t + h}, \widetilde{r}_{t, t+ h} (\bar{X}_{t+h}) - r_{t+h}(\bar{X}_{t + h}) \rangle \right]
\\
& + \mathbb{E} \left[ \langle X_{t + h } - \bar{X}_{t + h}, \widetilde{r}_{t, t+ h} (\bar{X}_{t+h}) - r_{t+h}(\bar{X}_{t + h}) \rangle \right]
\end{align}
\end{subequations}
Because both $\widetilde{r}_{t, t+ h} (\bar{X}_{t+h})$ and $r_{t+h}(\bar{X}_{t + h})$ have law $\pi$, and, by construction, $(\bar{X}_{t + h},r_{t+h}(\bar{X}_{t + h}))$ is a Wasserstein-2 optimal coupling, we see that the first term on the right-hand side of \eqref{eq:add subtract} is non-positive. For the second term in \eqref{eq:add subtract}, observe that
\begin{subequations}
\begin{align}
\mathbb{E} & \left[ \langle X_{t + h } - \bar{X}_{t + h},  \widetilde{r}_{t, t+ h} (\bar{X}_{t+h}) - r_{t+h}(\bar{X}_{t + h}) \rangle \right] = \frac{A_{t+h}}{\alpha_{t+h}} \mathbb{E} \left[ \langle v_{t+h}(\bar{X}_{t+h}), \widetilde{r}_{t, t+ h} (\bar{X}_{t+h}) - r_{t+h}(\bar{X}_{t + h}) \rangle \right]
\\
 &   = \frac{A_{t+h}}{\alpha_{t+h}}
    \int_{\mathbb R^d}
    \Big\langle
        v_{t+h}\bigl(\Phi_{t,t+h}(\bar x)\bigr),
        r_t(\bar x)
        -
        r_{t+h}\bigl(\Phi_{t,t+h}(\bar x)\bigr)
    \Big\rangle
    \rho_t(\dd\bar x),
\end{align}
\end{subequations}
where we have used \eqref{eq:QP-X} and \eqref{eq:def v} in the first line and \eqref{eq:flow property} in the second line. Adding and subtracting $h^{-1} (\Phi_{t,t+h}(\bar{x}) - \bar{x})$, which approximates $v_{t+h}(\Phi_{t,t+h}(\bar{x}))$, the last expression is equal to 
\begin{subequations}
\begin{align}
\label{eq:correlation comparison}
    &-\frac{A_{t+h}}{\alpha_{t+h}h}
    \int_{\mathbb R^d}
    \Big\langle
        \Phi_{t,t+h}(\bar x)-\bar x,
        r_{t+h}\bigl(\Phi_{t,t+h}(\bar x)\bigr)-r_t(\bar x)
    \Big\rangle
    \rho_t(\dd\bar x)
 \\
    &\quad
    +
    \frac{A_{t+h}}{\alpha_{t+h}}
    \int_{\mathbb R^d}
    \Bigg\langle
        \frac{\Phi_{t,t+h}(\bar x)-\bar x}{h}
        -
        v_{t+h}\bigl(\Phi_{t,t+h}(\bar x)\bigr),
        r_{t+h}\bigl(\Phi_{t,t+h}(\bar x)\bigr)-r_t(\bar x)
    \Bigg\rangle
    \rho_t(\dd\bar x).
\label{eq:velocity approx}
\end{align}
\end{subequations}
We argue that the term in \eqref{eq:correlation comparison} is non-positive: Indeed, for $\bar{x} \sim \rho_t$, both $r_{t+h}\bigl(\Phi_{t,t+h}(\bar x)\bigr)$ and $r_t(\bar{x})$ are distributed according to $\pi$, and $(\bar{x}, r_t(\bar{x}))$ is an optimal coupling. Therefore, $$\int_{\mathbb{R}^d}\langle \bar{x},  r_{t+h}\bigl(\Phi_{t,t+h}(\bar x)\bigr) -  r_t(\bar{x})\rangle \rho_t(\dd\bar x) \le 0.$$ Similarly, $(\Phi_{t,t+h}(\bar x),r_{t+h}\bigl(\Phi_{t,t+h}(\bar x)\bigr) )$ is an optimal coupling, hence $$\int_{\mathbb{R}^d}\langle \Phi_{t,t+h}(\bar x),r_{t+h}\bigl(\Phi_{t,t+h}(\bar x)\bigr) -  r_t(\bar{x})\rangle \rho_t(\dd\bar x) \ge 0.$$
Collecting the arguments so far, we see that $F(t+h) - \widetilde{F}_t(t+h)$ is bounded from above by \eqref{eq:velocity approx}. Using Cauchy-Schwarz, we arrive at
\begin{equation}
F(t+h) - \widetilde{F}_t(t+h) \le \frac{A_{t+h}}{\alpha_{t+h}}
    \left\|
        \frac{\Phi_{t,t+h}-I_d}{h}
        -
        v_{t+h} \circ \Phi_{t,t+h}
\right\|_{L^2(\rho_t)}
    \left\|
        r_{t+h}\circ\Phi_{t,t+h}-r_t
    \right\|_{L^2(\rho_t)}
\label{eq:CS}
\end{equation}
By \cite[Theorem 1.7.7]{panaretos2020invitation}, the second factor in \eqref{eq:CS} converges to zero as $h \rightarrow 0$. Indeed, since $(\Phi_{t,t+h})_\#\rho_t=\rho_{t+h}$ and
$\Phi_{t,t+h}\to \operatorname{Id}$ in $L^2(\rho_t)$, we have
$\rho_{t+h}\to\rho_t$ in $W_2$; hence \cite[Theorem 1.7.7]{panaretos2020invitation}
gives $r_{t+h}\to r_t$ locally uniformly, which, together with
$r_{t+h}\circ\Phi_{t,t+h}\sim\pi$, $r_t\sim\pi$ and
$\pi\in\mathcal P_2(\mathbb R^d)$, implies
\begin{equation}
\label{eq:second factor convergence}
\|r_{t+h}\circ\Phi_{t,t+h}-r_t\|_{L^2(\rho_t)}\to0.
\end{equation}

Finally, for the first factor in \eqref{eq:CS}, we note that along the characteristic flow in Assumption \ref{ass:FP-regularity}(ii), we have
\[
    \frac{\dd}{\dd s}
    v_s(\Phi_{t,s}(q))
    =
    a_s(\Phi_{t,s}(q)),
\]
where $a$ is the acceleration defined in \eqref{eq:acceleration}. Consequently,
\begin{align*}
&\frac{\Phi_{t,t+h}(q)-q}{h}
    -
    v_{t+h}(\Phi_{t,t+h}(q))
    =
    -\frac{1}{h}
    \int_t^{t+h}
    (s-t)a_s(\Phi_{t,s}(q))\dd s.
\end{align*}
Using $(\Phi_{t,s})_\#\rho_t=\rho_s$ and Minkowski's inequality, we can thus estimate
\begin{align*}
    \left\|
        \frac{\Phi_{t,t+h}-\operatorname{Id}}{h}
        -
        v_{t+h}\circ\Phi_{t,t+h}
    \right\|_{L^2(\rho_t)}
    &\leq
    \frac{1}{h}
    \int_t^{t+h}
    (s-t)\|a_s\|_{L^2(\rho_s)}\dd s \le 
    \frac{h}{2}
    \sup_{s\in[t,t+h]}
    \|a_s\|_{L^2(\rho_s)}
\le h C_I,
\end{align*}
where the last estimate follows from \eqref{eq:acceleration bound}, with a constant $C_I$ that only depends on the interval $I$.
Together with \eqref{eq:second factor convergence} and the fact that $\tfrac{A_{t+h}}{\alpha_{t+h}}$ is uniformly bounded on $I$, the claim follows from \eqref{eq:CS} by integration and dominated convergence. Indeed, we have that 
\begin{equation}
\left\|
        r_{t+h}\circ\Phi_{t,t+h}-r_t
    \right\|^2_{L^2(\rho_t)} \le 4  \int_{\mathbb{R}^d} |x|^2 \pi(\dd x), 
\end{equation}
because both $(r_{t+h}\circ\Phi_{t,t+h})(x)$ and $r_t(x)$ are distributed according to $\pi$ if $x \sim \rho_t$, so that passing to the limit as $h \rightarrow 0^+$ is justified. 
\end{proof}

Combining Lemmas \ref{lemma:frozen-map-derivative} and \ref{lemma:frozen-map-comparison}, we can now conclude as follows:
\begin{proof}[Proof of Lemmas \ref{lemma:L2-distance_convex} and \ref{L2-distance_sc}]
Because $F(u) = \widetilde{F}_u(u)$, we can write
\begin{equation*}
F(u+h) - F(u) = F(u+h) - \widetilde{F}_u(u+h) + \widetilde{F}_u(u+h) - \widetilde{F}_u(u) \le  (F(u+h) - \widetilde{F}_u(u+h))_+ + \widetilde{F}_u(u+h) - \widetilde{F}_u(u).
\end{equation*}  Integrating against a test function $0 \le \phi \in C_c^\infty((0,\infty))$ and dividing by $h$, we obtain
\begin{subequations}
\begin{align*}
\int_{\mathbb{R}_{>0}} \phi(u) \frac{F(u+h) - F(u)}{h} \dd u \le &  \int_{\mathbb{R}_{>0}} \phi(u) \frac{(F(u+h) - \widetilde{F}_u(u + h))_+}{h} \dd u
\\
& + \int_{\mathbb{R}_{>0}} \phi(u) \frac{\widetilde{F}_u(u+h) - \widetilde{F}_u(u)}{h} \dd u.
\end{align*}
\end{subequations}
Taking $h \rightarrow 0^+$, see that
\begin{equation}
- \int_{\mathbb{R}_{>0}} F(u) \phi'(u) \dd u \le \int_{\mathbb{R}_{>0}} G(u) \phi(u) \dd u,
\end{equation}
using Lemma \ref{lemma:frozen-map-comparison} and Lemma \ref{lemma:frozen-map-derivative} or \ref{lemma:frozen-map-derivative_sc}, respectively, and denoting the right-hand side of \eqref{eq:frozen-map-derivative} or \eqref{eq:frozen-map-derivative_sc}, respectively, by $G$. This inequality means that $\partial_t F\leq G$ in the sense
of distributions. Hence,
$u\mapsto F(u)-\int_{u_0}^u G(v)\dd v$
has a non-increasing representative. Moreover, it is continuous since $F$ is continuous (following, for instance, from \cite[Theorem 1.7.7]{panaretos2020invitation}). As it agrees almost everywhere with a non-increasing representative, it is itself non-increasing, which implies the claim.
\end{proof}

\bibliographystyle{abbrvnat} 
\bibliography{references}

\end{document}